\documentclass[a4paper,10pt,oneside]{amsart}
\pdfoutput=1
\usepackage[utf8]{inputenc}
\usepackage[T1]{fontenc}
\usepackage[english]{babel}
\usepackage{setspace,marginnote}
\usepackage[left=2.2cm,right=2.2cm,top=2cm,bottom=2cm,footskip=1cm]{geometry}
\usepackage{microtype,lmodern}

\usepackage{enumitem}
\setlist[enumerate]{label=\textit{(\roman*)},ref=\textit{(\roman*)}}

\usepackage{amsmath,amssymb,amsthm,mathtools,amsfonts}
\usepackage[ruled,vlined]{algorithm2e}
\usepackage{tikz,mdframed}

\newtheoremstyle{slthm}
{0pt}
{0pt}
{\normalfont}
{}
{\bfseries}
{.}
{.5em}
{\thmname{#1}\thmnumber{ #2}\thmnote{ (#3)}}
\theoremstyle{slthm}

\usepackage{lastpage}

\usepackage[
colorlinks,
linkcolor=black!40!red,
citecolor=green!40!black,
urlcolor=red,
pdfauthor={Richard Mycroft and T\'assio Naia},
backref
]{hyperref}

\usepackage{fancyhdr}

\newlength{\mythmbar}
\newlength{\myinsep}
\newlength{\myleftinsep}
\newlength{\mylen}
\mdfsetup{skipabove=.5\baselineskip,skipbelow=.5\baselineskip,usetwoside=false,nobreak}

\colorlet{darkred}{red!60!black}
\colorlet{claimColor}{gray!60!white}
\colorlet{definitionColor}{orange!60!white}

\mdfdefinestyle{tn/minimal}{
  linewidth=\mythmbar,
  bottomline=false,
  topline=false,
  rightline=false,
  innerleftmargin=\myleftinsep,
  leftmargin=-\myinsep,
  innerrightmargin=\myinsep,
  innertopmargin=1pt,
  innerbottommargin=0,
  rightmargin=\myinsep,
  userdefinedwidth=\mylen,
  linecolor=green!60!black}

\mdfdefinestyle{tn/minimal claim}{
  linewidth=\mythmbar,
  bottomline=false,
  topline=false,
  rightline=false,
  innerleftmargin=\myleftinsep,
  leftmargin=-\myinsep,
  innerrightmargin=\myinsep,
  innertopmargin=1pt,
  innerbottommargin=0,
  rightmargin=\myinsep,
  userdefinedwidth=\mylen,
  linecolor=claimColor}

\mdfdefinestyle{tn/minimal definition}{
  linewidth=\mythmbar,
  bottomline=false,
  topline=false,
  rightline=false,
  innerleftmargin=\myleftinsep,
  leftmargin=-\myinsep,
  innerrightmargin=\myinsep,
  innertopmargin=1pt,
  innerbottommargin=0,
  rightmargin=\myinsep,
  userdefinedwidth=\mylen,
  linecolor=definitionColor}

\newcommand{\activestyle}{tn/minimal}
\newcommand{\activeCLAIMstyle}{tn/minimal claim}
\newcommand{\activeDEFINITIONstyle}{tn/minimal definition}
\newmdtheoremenv[style=\activestyle]{theorem}[algocf]{Theorem}
\newmdtheoremenv[style=\activestyle]{lemma}[algocf]{Lemma}
\newmdtheoremenv[style=\activestyle]{proposition}[algocf]{Proposition}
\newmdtheoremenv[style=\activeCLAIMstyle]{claim}[algocf]{Claim}
\newmdtheoremenv[style=\activestyle]{corollary}[algocf]{Corollary}
\newmdtheoremenv[style=\activestyle]{fact}[algocf]{Fact}
\newmdtheoremenv[style=\activeDEFINITIONstyle]{definition}[algocf]{Definition}
\newmdtheoremenv[style=\activeDEFINITIONstyle]{remark}[algocf]{Remark}

\newcommand{\defi}[1]{{\textcolor{darkred}{\emph{#1}}}}

\newcommand{\oldqed}{}
\def\endofClaim{\hfill{$\diamond$}}

\newenvironment{claimproof}[1][Proof]
{ \renewcommand{\oldqed}{\qedsymbol}
  \renewcommand{\qedsymbol}{\endofClaim}
  \begin{proof}[#1]}
{ \end{proof}
  \renewcommand{\qedsymbol}{\oldqed}}

\newcommand{\cala}{\ensuremath{\mathcal{A}}}
\newcommand{\cald}{\ensuremath{\mathcal{D}}}
\newcommand{\cale}{\ensuremath{\mathcal{E}}}
\newcommand{\cali}{\ensuremath{\mathcal{I}}}
\newcommand{\calm}{\ensuremath{\mathcal{M}}}
\newcommand{\calp}{\ensuremath{\mathcal{P}}}

\newcommand{\calt}{\ensuremath{\mathcal{T}}}
\newcommand{\calu}{\ensuremath{\mathcal{U}}}
\newcommand{\calv}{\ensuremath{\mathcal{V}}}
\newcommand{\bft}{\ensuremath{\mathbf{t}}}
\newcommand{\ee}{\mathrm{e}}
\newcommand{\bbe}{\mathbb{E}}
\newcommand{\bbn}{\mathbb{N}}
\newcommand{\bbp}{\mathbb{P}}
\newcommand{\bbr}{\mathbb{R}}

\newcommand{\littleo}{\ensuremath{\mathrm{o}}}
\newcommand{\Qunder}{\ensuremath{Q_{0}}}
\newcommand{\Vground}{V_{\mathrm{ground}}}
\newcommand{\treeAttach}{\ensuremath{A}}
\newcommand{\pathRoot}{\ensuremath{\Pi_1}}
\newcommand{\pathLast}{\ensuremath{\Pi_7}}
\newcommand{\pathOther}{\ensuremath{\Pi_{\mathrm{other}}}}
\newcommand{\verticesUsedAtStart}{\ensuremath{\Xi}}
\newcommand{\vtau}{v_\tau}
\newcommand{\ttau}{{t_\tau}}
\let\phi\varphi
\let\emptyset\varnothing
\newcommand{\eps}{\varepsilon}
\renewcommand{\setminus}{\smallsetminus}
\newcommand{\deq}{\coloneqq}
\newcommand{\tand}{\text{ and }}
\newcommand{\ink}{{\ensuremath{i\in[k]}}}
\newcommand{\jnk}{{\ensuremath{j\in[k]}}}
\newcommand{\rarr}{\rightarrow}
\newcommand{\larr}{\leftarrow}
\newcommand{\imm}{{i-1}}
\newcommand{\ipp}{{i+1}}
\newcommand{\dcup}{\mathbin{\dot{\cup}}}
\newcommand{\dplus}{{\ensuremath{d_\geq}}}
\newcommand{\commneigh}[3]{\ensuremath{N^{#1}(#2,#3)}}
\newcommand{\Jblow}{\ensuremath{J^{\mathrm{blow}}}}
\DeclareMathOperator{\dist}{dist}
\DeclareMathOperator{\bigoh}{O}
\DeclareMathOperator{\Var}{Var}
\DeclareMathOperator{\polylog}{polylog}
\DeclareMathOperator{\used}{used}
\DeclareMathOperator{\open}{open}
\DeclareMathOperator{\reserved}{reserved}

\newcommand{\gadget}[4]{%
  \ensuremath{\textstyle #1\,\genfrac{<}{>}{0pt}{}{#2}{#3}\, #4}}
\newsavebox{\ldiamond}
\savebox{\ldiamond}{\tikz[baseline={(0,.3ex)}]{\draw[semithick] (0,.3ex) --(1.3ex,0) (0,1.3ex) --(1.3ex,1.6ex);}}
\newsavebox{\rdiamond}
\savebox{\rdiamond}{\reflectbox{\usebox{\ldiamond}}}
\renewcommand{\gadget}[4]{%
  \ensuremath{\textstyle #1{\usebox{\ldiamond}}\genfrac{.}{.}{0pt}{}{#2}{#3}{\usebox{\rdiamond}} #4}}

\newcommand{\farc}{\textcolor{black}{{}\to{}}}
\newcommand{\barc}{\textcolor{black}{{}\leftarrow{}}}
\newcommand{\rstar}{{\ensuremath{R^\star}}}

\newcommand{\phase}[1]{
  \smallskip\par\noindent
  \makebox[0pt]{
    \hspace*{-1.5\parindent}
    \textbf{\color{darkred}{\scriptsize\raisebox{.4ex}{$\blacktriangleright$}}}}
  \textbf{#1}}

\begin{document}

\begin{abstract}
  We prove that every oriented tree on $n$~vertices with bounded
  maximum degree appears as a spanning subdigraph of every directed
  graph on $n$~vertices with minimum semidegree at least
  $n/2+\littleo(n)$. This can be seen as a directed graph analogue of
  a well-known theorem of Koml\'os, S\'ark\"ozy and~Szemer\'edi.
  Our result for trees follows from a more general result, allowing
  the embedding of arbitrary orientations of a much wider class of
  spanning ``tree-like'' structures, such as collections of at most $O(n^{0.99})$ pairwise vertex-disjoint cycles
  and subdivisions of graphs $H$ with~$|H|< \exp\bigl(\sqrt{\bigoh(\log n)}\,\bigr)$
  in which each edge is subdivided at least once.
\end{abstract}

\title{Trees and treelike structures in dense digraphs}

\author{Richard Mycroft}
\address{Richard Mycroft, University of Birmingham, Birmingham, B15\,2TT, United Kingdom}
\email{r.mycroft@bham.ac.uk}

\thanks{Richard Mycroft acknowledges support from EPSRC (Standard Grant
  EP/R034389/1).}

\author{T\'assio Naia}

\address{T\'assio Naia, Centre de Recerca Matemàtica,
 Campus Bellaterra, 08193, Barcelona, Spain}
\email{tnaia@member.fsf.org}

\thanks{T\'assio Naia acknowledges support from CNPq
  (201114/2014-3) and FAPESP (2019/04375-5).}

\maketitle

\section{Introduction}
\label{s:introduction}

A celebrated theorem of Koml\'os, S\'ark\"ozy
and~Szemer\'edi~\cite{komlos95:_proof_bollob} states that if~$G$ is a
graph of order~$n$ with~$\delta(G)\geq n/2+\littleo(n)$, then $G$
contains every tree of order~$n$ with bounded maximum degree.

\begin{theorem}\label{t:KSS}
  \cite{komlos95:_proof_bollob}
  For all~$\Delta\in\bbn$ and $\alpha > 0$ there exists $n_0$ such
  that every graph~$G$ of order $n\ge n_0$ with
  $\delta(G)\ge \bigl(\frac{1}{2}+\alpha\bigr)n$ contains every tree
  $T$ of order~$n$ with~$\Delta(T)\le \Delta$.
\end{theorem}

Koml\'os, S\'ark\"ozy and~Szemer\'edi later strengthened
Theorem~\ref{t:KSS}, replacing the constant bound~$\Delta$ by
$cn/\log n$, where $c$ is some constant depending
on~$\alpha$~\cite{komlos01:spanning_trees_dense}. Many variations and
extensions of Theorem~\ref{t:KSS} have been investigated,
e.g.,~\cite{balogh2011local,BHKMPP18,BST09:bandwidth,clemens2015building,KrivKwanSud17:trees_rand_pert}.
We prove the following directed graph (digraph) analogue of Theorem~\ref{t:KSS},
where minimum degree is replaced by minimum
semidegree~$\delta^0(\cdot)$ (the minimum of in- and outdegrees over
all vertices) and the maximum degree is replaced by the maximum total
degree~$\Delta(\cdot)$ (maximum degree in the underlying tree).

\begin{theorem}\label{t:KSS-analogue}
  For all~$\Delta\in\bbn$ and $\alpha > 0$ there exists $n_0$ such
  that every digraph~$G$ of order $n\ge n_0$ with
  $\delta^0(G)\ge \bigl(\frac{1}{2}+\alpha\bigr)n$ contains every oriented tree
  $T$ of order~$n$ with~$\Delta(T)\le \Delta$.
\end{theorem}

In fact we prove a stronger result (Theorem~\ref{t:treelike}), which allows the
embedding of a large class of treelike graphs, and which implies
both Theorem~\ref{t:KSS-analogue} and the more general embedding
result for trees below.
For this we define a bare path $P=p_1p_2\cdots p_n$ in
a (di)graph $G$ to be a path whose internal vertices
$p_2,\ldots,p_{n-1}$ each have degree~$2$ in (the underlying graph of)~$G$.

\begin{theorem}\label{t:many-paths-or-leaves}
  Suppose $\cramped{\frac{1}{n}}\ll\lambda\ll\alpha$. If $G$
  is a digraph of order~$n$ with~$\delta^0(G)\geq (1/2+\alpha)n$, then
  $G$~contains every oriented tree $T$
  of order~$n$ with~$\Delta(T)\leq \exp(\sqrt{\log n})$ such that
   $T$~contains either
  \begin{enumerate}
  \item at least $\lambda n$ pairwise vertex-disjoint bare paths of order~$7$, or
  \item at least $\lambda n$ pairwise disjoint edges incident to leaves.\label{i:mpl/leaves}
  \end{enumerate}
\end{theorem}

Note that for all positive $\eps$ we have
  $\polylog(n) \asymp \exp\bigl(\Theta(\log \log n)\bigr) \lesssim \exp(\sqrt{\bigoh(\log n)}) \lesssim n^\eps$.
Our main result allows embedding families of sparse graphs which arise from an
arbitrary small graph by numerous applications of the following operations:
\begin{enumerate}[label=(\Alph*)]
\item \label{i:add-leaf}
  append a leaf (i.e., add a new vertex connected to the graph by
  a single edge);
\item \label{i:split-edge}
  subdivide an edge (i.e., replace some edge $uv$ by a path $uxv$,
  where $x$ is a new vertex).
\end{enumerate}
Throughout the text, vertices of degree one are called \defi{leaves} (even in graphs other than trees),
and $|G|$ denotes the order of the graph~$G$.

\begin{theorem}\label{t:treelike}
  Suppose~$\cramped{\frac{1}{n}}\ll\lambda\ll\alpha$. Fix a graph~$\Qunder$
  and let $Q$ be a graph of order~$n$ obtained from~$\Qunder$ by
  a~sequence of operations \ref{i:add-leaf}~and~\ref{i:split-edge}
  in which each edge of~$\Qunder$ is subdivided at least once.
  Suppose additionally that~$|\Qunder| \le n^{0.99}$
  and $\Delta(Q) \le \exp(\sqrt{\log n})$, and let $G$ be a
  digraph with~$\delta^0(G)\geq (1/2 + \alpha)|G|$.
  \begin{enumerate}[label=(\arabic*)]
  \item \label{i:treelike-quasi-spanning}
    If $|G|\ge (1+\alpha)n$, then $G$ contains every orientation of~$Q$.
  \item \label{i:treelike-spanning} If $|G|=n$ and $Q$ contains
    either $\lambda n$ pairwise vertex-disjoint bare paths of order $7$ or
    $\lambda n$ pairwise disjoint edges incident to leaves, then $G$
    contains every orientation of~$Q$.
  \end{enumerate}
\end{theorem}

Theorem~\ref{t:treelike} can be used to embed a wide range of
spanning treelike subdigraphs in a digraph of high minimum
semidegree. For example, it implies that every digraph of order~$n$
with minimum semidegree at least~$n/2+\littleo(n)$ contains every
orientation of a Hamilton cycle. This gives an asymptotic version of
recent results by DeBiasio and~Molla~\cite{debiasio15:antidir_Ham}
and by~DeBiasio, K\"uhn, Molla,
Osthus and~Taylor~\cite{debiasio2015arbitrary}, which can be stated
jointly as the following theorem (the statement for directed cycles
had previously been obtained by~Ghouila-Houri~\cite{Ghouila60}).

\begin{theorem}\label{t:ham-cycle-n/2+1}\label{t:ham-cycle-n/2}\cite{debiasio2015arbitrary,debiasio15:antidir_Ham}
There exists $n_0\in\bbn$ such that the following holds for every
digraph~$G$ of order~$n\ge n_0$.
\begin{enumerate}
\item If $\delta^0(G)\ge n/2 +1$, then $G$ contains every orientation of a Hamilton cycle.

\item If $\delta^0(G)\ge n/2$, then $G$ contains every
  orientation of a Hamilton cycle, except perhaps for the
  anti-directed orientation in which each vertex has either no
  inneighbours or no outneighbours.
\end{enumerate}
\end{theorem}

In the same way we can embed every orientation of a disjoint union of $O(n^{0.99})$~cycles, a result which may
be of independent interest.

\begin{corollary}\label{cor:cycles}
  For all $\alpha > 0$ there exists $n_0$ such that the following
  holds for every digraph~$G$ of order~$n\ge n_0$
  with~$\delta^0(G)\ge (1/2+\alpha)n$. If $H$ is a graph of order at
  most~$n$ consisting of at most~$\frac{1}{4}n^{0.99}$ pairwise vertex-disjoint
  cycles, then~$G$ contains every orientation of~$H$.
\end{corollary}

\begin{proof}
  Cycles of length at most~$5$ cover at most~$\frac{5}{4}n^{0.99}<\alpha n/2$ vertices
  of~$H$, so we may embed all such cycles greedily, whereupon the
  subdigraph~$G'$ induced by the $n'$ uncovered vertices satisfies
  $\delta^0(G)\ge (1/2 + \alpha/2)n'$. The remaining cycles
  each have length at least~$6$ and so are
  subdivisions of triangles where each edge is subdivided at least
  once. We may therefore apply Theorem~\ref{t:treelike}\,\ref{i:treelike-spanning}
  with~$\alpha/2$ and~$n'$ in place of~$\alpha$ and~$n$ respectively to embed these
  cycles in~$G'$, completing the embedding of~$H$ in~$G$.
\end{proof}

We also consider embeddings of random trees.
Moon~\cite{moon70:trees} showed that a uniformly-random labelled
$n$-vertex tree~$T$ has sub-polylogarithmic maximum degree with high probability. It is not difficult to check that with high probability~$T$
also satisfies condition~\ref{i:mpl/leaves} of Theorem~\ref{t:many-paths-or-leaves} (see,
e.g.,~\cite{MyNa18}). Together these observations imply the following
corollary, for which we denote by $\calt_n$ the set of oriented trees with
vertex set~$[n]$.

\begin{corollary}\label{cor:typical-trees}
  Fix $\alpha >0$. If $T$ is chosen uniformly at
  random from~$\calt_n$, then
  with high probability we have~$T\subseteq G$ for every
  digraph $G$ of order~$n$ with~$\delta^0(G)\ge (1/2 + \alpha)n$.
\end{corollary}

While this manuscript was under review, Kathapurkar and Montgomery~\cite{kathapurkar2022spanning} announced a stronger version of Theorem~\ref{t:KSS-analogue}, in which the constant bound on the maximum degree $\Delta(T)$ is replaced by a best-possible bound of $\Delta(T) \leq cn/\log n$. This impressive breakthrough uses very different methods to those used in this paper, which cannot hope to succeed for trees with maximum degree even close to that size. However, the methods Kathapurkar and Montgomery used appear quite specific to trees, whereas our approach allows a much wider class of ``treelike graphs'' to be handled similarly, as in Theorem~\ref{t:treelike} and its applications. For this reason we believe the methods and results of this paper should still be of widespread use and interest even in the light of this new advance.

In the next section of this paper we outline the key ideas used to prove our main results. Section~\ref{s:prelim} then introduces the main definitions and results that we combine to give the full proofs in Section~\ref{s:treelike}.

\section{Key proof ideas}
\label{s:outline}
Very broadly speaking, we use the following approach to prove Theorem~\ref{t:treelike}, where our aim is to embed a treelike oriented graph $Q$ into a directed graph $G$ with high minimum semidegree. First, we \emph{allocate} each vertex of $Q$ to a cluster of a reduced graph $R$ of $G$ obtained by an application of the Szemer\'edi regularity lemma for digraphs; this allocation should respect the directions of edges in $G$ and $Q$ by having the property that if $x$ is an outneighbour of $y$ in $Q$, then the cluster $V_x$ to which $x$ is allocated should be an outneighbour in $R$ of the cluster $V_y$ to which $y$ is allocated (in other words, the allocation of vertices to clusters should be a homomorphism from $Q$ to $R$). Having done this, we then \emph{embed} each vertex of $Q$ within the cluster to which it was allocated so as to form a copy of $Q$ in $G$. This two-step process of first allocating vertices to clusters of the reduced graph, then embedding within the clusters, was previously used by K\"uhn, Mycroft and Osthus~\cite{KMO11:sumner_approximate} and subsequently further developed by Mycroft and Naia~\cite{MyNa18}, in both cases to embed trees within tournaments. For our present application a more significant development of this approach is required, with many new ideas, to reflect the setting of a digraph of high minimum semidegree and the fact that the graph to be embedded may no longer be a tree (in particular, the previous works relied heavily on the fact that each cluster within a tournament induced a subtournament, giving far greater freedom to embed vertices within a cluster; this is no longer possible in our more general setting as each cluster could be an independent set).

\subsection{Sketch proof for non-spanning trees}
Consider the case of Theorem~\ref{t:treelike}(1) in which $Q$ is a tree on $n$ vertices and $G$ is a directed graph on $(1+\alpha)n$ vertices with $\delta^0(G) \geq (1/2 +\alpha)|G|$. Let $T$ be an orientation of~$Q$; our goal is then to embed $T$ within $G$. In this case we may proceed as follows.
\begin{enumerate}
    \item Apply the regularity lemma to partition $V(G)$ into clusters $V_1, \dots, V_k$ of equal size, as well as a small set of exceptional vertices $V_0$ which can safely be ignored, so that almost all pairs of clusters form regular pairs in each direction. In particular each cluster has size $|V_i| \geq (1+\alpha/2)n/k$. Form a reduced graph $R^\star$ with vertex set $[k]$ in which $i \to j$ is an edge if and only if the graph $G[V_i \to V_j]$ of edges directed from $V_i$ to $V_j$ is regular and dense. The reduced graph $R^\star$ then inherits an analogous minimum semidegree condition from $G$, namely that $\delta^0(R^\star) \geq (1/2+\alpha/2)k$. The full definitions and details for this step are presented in a more general form in Section~\ref{s:regularity}.
    \item The minimum semidegree of $R^\star$ is enough to ensure that $R^\star$ contains a spanning subgraph $R$ which is a $d$-regular expander, meaning that every proper nonempty subset $S \subseteq V(R)$ has $|N^+(S)|,\, |N^-(S)| > |S|$ and every vertex $i \in V(R)$ has both indegree and outdegree of precisely $d$. A more general version of this statement, along with a key mixing property of expander digraphs, is presented in Section~\ref{s:regular-expander}.
    \item {\bf Allocation:} We now allocate the vertices of $T$ to clusters using a randomised allocation algorithm, which builds a homomorphism $\varphi: T \to R$ one vertex at a time. First, choose a root $r$ of $T$ and define $\varphi(r) \in [k]$ arbitrarily. Next iterate the following step: choose a vertex $x \in V(T)$ for which $\varphi(x)$ has been defined, say $\varphi(x) = u_x$, but for which $\varphi(y)$ has not been defined for any child $y$ of $x$. Choose vertices $u^+_x \in N^+_R(u_x)$ and $u^-_x \in N^-_R(u_x)$ uniformly at random, and set $\varphi(y) = u^+_x$ for every child $y$ of $x$ which is an outneighbour of $x$, and $\varphi(y) = u^-_x$ for every child $y$ of $x$ which is an inneighbour of $x$. Proceed in this manner until $\varphi(x)$ is defined for every vertex of $T$, and observe that the resulting map $\varphi$ is then a homomorphism $\varphi: T \to R$. Moreover, and crucially, so long as the maximum degree $\Delta(T)$ of $T$ is not too large, the mixing property of expander digraphs ensures that with high probability $\varphi$ allocates the vertices of $T$ approximately uniformly to the clusters of $R$, with, say, for each $i \in [k]$ at most $(1+\alpha/3)n/k$ vertices $x \in V(T)$ having $\varphi(x) = i$.
    The proof of this statement, for a more general version of this algorithm, is presented in Section~\ref{s:gen-allocation-algo}.
    \item {\bf Embedding:} Finally we form a copy of $T$ in $G$ by greedily embedding each vertex $x$ of $T$ within the cluster $V_{\varphi(x)}$ of $G$ to which it was allocated, starting with the root $r$. Each time we embed a vertex $x$ in a cluster $V_{u_x}$ we reserve sets of size $\bigoh(\sqrt{n})$ in the clusters $V_{u^-_x}$ and $V_{u^+_x}$ in which the children of $x$ which are respectively inneighbours and outneighbours of $x$ will be embedded; no other vertices may be embedded within these sets until all children of $x$ have been embedded. In this way we avoid ``treading on our toes'' by occupying all of the inneighbours or outneighbours of $x$ whilst they are still required for children of $x$; moreover the fact that edges of $R$ correspond to dense regular pairs in $G$ is sufficient for us to successfully choose appropriate embeddings within the specified sets. For this process to succeed we need the order in which we proceed through the vertices of $T$ to have the property that at any time there are not too many vertices which have been embedded but have a child vertex yet to be embedded; the notion of a \emph{tidy ancestral order} presented in Section~\ref{s:trees} captures what we need for this, whilst the way in which we choose the reserved sets is explained in Section~\ref{s:homomorphisms}.
\end{enumerate}
A bound on $\Delta(T)$ is crucial for this argument to succeed, since we need the randomised allocation algorithm to distribute the vertices of $T$ approximately uniformly among the $k$ clusters (by contrast, the vertices of a star would be allocated in a highly-unbalanced way). To achieve this, we specifically need the distance between almost all pairs of vertices of $T$ to be significantly larger than $k$, and indeed the bound on $\Delta(T)$ given in Theorem~\ref{t:treelike} is chosen to achieve precisely this property (see Sections~\ref{s:regular-expander} and~\ref{s:gen-allocation-algo} for more details).

\subsection{Sketch proof for spanning trees}
Let us now consider the case of Theorem~\ref{t:treelike}\,\ref{i:treelike-spanning} in which $Q$ is a tree on $n$ vertices and $G$ is a directed graph on $n$ vertices with $\delta^0(G) \geq (1/2 +\alpha)n$. Again, let $T$ be an orientation of $Q$. The difference with the previous case is that our aim is now a spanning embedding of $T$ in $G$; there is no `room to spare'. However, Theorem~\ref{t:treelike}\,\ref{i:treelike-spanning} does provide a linear number of either bare paths on seven vertices or edges incident to leaves; for the sake of this discussion we assume the former (the arguments for the latter are quite similar and somewhat simpler).

We begin by splitting $T$ into two not-too-small subtrees $T_1$ and $T_2$ with one vertex $v$ in common, which we take as the root of both trees. Without loss of generality $T_1$ contains a set $\calp$ of bare paths on seven vertices such that $\calp$ has small linear size; and moreover we may insist that all paths in $\calp$ have the same `pattern' (the sequence of directions of edges along the path). As before we apply the regularity lemma to obtain clusters $V_1, \dots, V_k$ of equal size and a small set $V_0$ of bad vertices, and similarly as there we may define a reduced graph $R^\star$ with vertex set $[k]$ whose vertices correspond to clusters, and also obtain a subgraph $R$ of~$R^\star$ on vertex set $[k]$ which is a $d$-regular expander. However, we now additionally insist that $R$ contains a Hamilton cycle $H$ (this is possible by the minimum degree condition that $R$ inherits from $G$); without loss of generality we assume that $H$ has edges $1 \to 2 \to 3 \to \dots \to 1$. We also insist that for each edge $i \to j$ of $H$ the pair $G[i \to j]$ is dense and superregular; this can be achieved by deleting a small number of vertices from each cluster and adding them to $V_0$.

Next we apply a modified version of the allocation algorithm presented above to $T_1$: the change is that if $x$ is a non-initial vertex of a bare path in $\calp$, and $y$ is the parent of $x$, then we instead set $\varphi(x)$ to be $\varphi(y)+1$ if $y \in N^+(x)$ and $\varphi(y)-1$ if $y \in N^-(x)$. The effect is that internal edges of bare paths in $\calp$ are allocated deterministically within the cycle $H$ rather than randomly within the regular expander $R$. Despite this change, the allocation algorithm still allocates the vertices of $T_1$ approximately uniformly across $[k]$ with high probability; moreover, the initial vertices of bare paths in $\calp$ are also allocated approximately uniformly across $[k]$ with high probability.

We now choose pairwise disjoint sets $\calp^H, \calp^0, \calp^\diamond \subseteq \calp$ such that $|\calp^0| = |V^0|$, whilst $\calp^H$ and $\calp^\diamond$ each have small linear size. Each of these collections of bare paths will play a distinct role in completing the embedding of $T$ in $G$, as follows.
\begin{enumerate}
    \item We reallocate the internal vertices of paths $P \in \calp^0$ in such a way that the middle vertex of each path $P$ can be embedded to a corresponding vertex of $V_0$. Note that, since in the allocation step we are forming a homomorphism $\varphi: V(T) \to R^\star$, for notational convenience we actually also include each bad vertex in $V_0$ as a vertex of $R^\star$, so $V(R^\star) = [k] \cup V_0$, and for each vertex $v \in V_0$ our allocation will have $\varphi(x) = v$ for precisely one vertex $x$, which will then be embedded to $v$ at the embedding step ($x$ is the middle vertex of a path $P \in \calp^0$). In this way we ensure that all vertices of $V_0$ are covered by our embedding of $T$.
    \item We reallocate the internal vertices of paths $P \in \calp^\diamond$ to structures in $R$ called `diamonds'; the purpose of which is that for each such path $P$ we have two choices for the allocation of the middle vertex $v_4^P$ of $P$, where each choice is consistent with the allocation of the neighbours $v_3^P$ and $v_5^P$ of $v_4^P$ in $P$. This gives us the flexibility to reallocate the middle vertex $v_4^P$ of each path $P \in \calp^\diamond$ after all other vertices have been allocated, so as to slightly adjust the number of vertices allocated to each cluster. (Details of diamonds and this reallocation process can be found in Section~\ref{s:diamonds}.)
    \item Finally, we leave the allocation of paths in $\calp^H$ unchanged; the internal vertices of these paths will be the vertices we embed at the very end to `finish off' our spanning embedding of $T$ in $G$.
\end{enumerate}
Unfortunately, it is unavoidable that the reallocation of vertices in paths in $\calp^0$ and $\calp^\diamond$ may create greater imbalances in the allocation of vertices of $T_1$ to clusters than can be corrected by reallocating the middle vertices of paths in $\calp^\diamond$\!. This is the purpose of the tree $T_2$: after we have applied the randomised allocation algorithm to $T_1$, and subsequently reallocated vertices of paths in $\calp^0$ and $\calp^\diamond$\!, we now apply a biased version of the allocation algorithm to $T_2$, which with high probability allocates slightly more vertices to the clusters which were underoccupied by $T_1$, and slightly fewer vertices to the clusters which were overoccupied by $T_1$. The end result is that the overall allocation of vertices of $T$ to clusters is very close to uniform; specifically it is close enough that it can be made uniform by reallocating the middle vertices of paths in $\calp^\diamond$\!. That is, after this reallocation exactly the same number of vertices are allocated to each cluster, meaning that the allocation `fits' perfectly.

Finally, we embed the vertices of $T$ in $G$ by the same embedding approach as described previously, with two modifications. Firstly, we take special care with the embeddings of paths $P \in \calp^0$, whose middle vertices~$v_4^P$ must be embedded to the vertex $v \in V_0$ to which they were allocated; to do this we embed the remaining vertices of $P$ so that the neighbours $v_3^P$ and $v_5^P$ of $v_4^P$ are embedded to inneighbours/outneighbours (as appropriate) of the vertices in $V_0$. Secondly, we do not embed the internal vertices of paths $P \in \calp^H$ at this stage, but leave these until the very end. This ensures that there is a little `room to spare' when embedding all vertices of $T$ other than those in paths in $\calp^H$, so the previous approach of reserving sets for the children of each vertex as it is embedded is still valid. Finally, to complete the embedding we need to embed the internal vertices of the paths in $\calp^H$ to the vertices still unoccupied within each cluster. Since each path in $\calp^H$ is allocated along a sequence of superregular pairs (corresponding to edges of $H$), this can be done by applying the four-layer theorem of Koml\'os, S\'ark\"ozy and Szemer\'edi (see Section~\ref{s:regularity}).

\subsection{Sketch of the full proof}
For the full version of the lemma, the graph $Q$ that we seek to embed may not be a tree. However the conditions on $Q$ do ensure that we can find a small `ground set' $\Vground \subseteq V(Q)$ so that $\Vground$ contains all vertices of $Q_0$ as well as all short paths between vertices of $Q_0$. In particular, $Q \setminus \Vground$ is a forest $F$. We embed the vertices of $\Vground$ greedily in $G$, and then proceed similarly as in the previous case, dividing our forest $F$ into two subforests $F_1$ and $F_2$ and allocating and embedding these similarly as we described for $T_1$ and $T_2$. Some additional difficulties arise here in doing this: for example, components of $F$ may contain two vertices with neighbours in $\Vground$, and whilst one can serve as the root of the component, the other (which we call a secondary attachment) must be handled carefully to ensure that it is allocated and embedded appropriately given the embedding of its neighbour in $\Vground$. However, these difficulties can be overcome through careful modification to the allocation and embedding procedures (in~particular, this is the purpose of the set $Z$ of bad vertices in the general allocation algorithm in Section~\ref{s:gen-allocation-algo}).

\section{Auxiliary concepts and results}\label{s:prelim}

The following concepts and results play an important role in our proofs.
We follow standard graph-theoretical notation (see,
e.g.,~\cite{diestel97:_graph}). For clarity, we define some of
our notation (mostly related to digraphs) below. More specific
terms are defined in later sections.

A \defi{directed graph} $G$, or \defi{digraph} for short, is a
pair~$\bigl(V(G),E(G)\bigr)$ of sets: a vertex set~$V(G)$ and an edge set~$E(G)$, where each
edge $e\in E(G)$ is an ordered pair of distinct vertices.
The \defi{order} of~$G$, denoted $v(G)$ or~$|G|$, is defined to be $|V(G)|$
and the \defi{size} of~$G$ is~$e(G)=|E(G)|$. We think of the
edge~$(u, v)$ as being directed from~$u$ to~$v$, and write~$x\farc y$
or~$y\barc x$ to denote the edge~$(x, y)$; if the orientation of the
edge does not matter, we write~$\{u,v\}$ (or~$\{v,u\}$) instead. In
either case, $u$~and~$v$ are said to be the~\defi{endvertices}
of~$\{u,v\}$, and we also call $u$ (respectively~$v$) a
\defi{neighbour} of $v$ (respectively~$u$).

In a digraph~$G$, the \defi{outneighbourhood}~$N^+_G(x)$ of
a~vertex~$x$ is the set~$\{\,y : x\farc y\in E(G)\,\}$;
the~\defi{inneighbourhood}~$N^-_G(x)$ of~$x$
is~$\{\,y : x\barc y\in E(G)\,\}$. The~\defi{outdegree}
and~\defi{indegree} of~$x$ in~$G$ are
respectively~$\deg^+_G(x) \deq \bigl|N^+_G(x)\bigr|$
and~$\deg^-_G(x) \deq \bigl|N^-_G(x)\bigr|$, and
the~\defi{semidegree}~$\deg^0_G(x)$ of~$x$ is the minimum of the
outdegree and indegree of~$x$. We say that $G$ is \defi{$r$-regular}
if for all $x\in V(G)$ we have $\deg^-(x)=\deg^+(x)=r$. The \defi{minimum
  semidegree}~$\delta^0(G)$ of~$G$ is the minimum of~$\deg^0_G(x)$
over all~$x\in V(G)$. For any subset~$Y \subseteq V(G)$, we
write~$\deg_G^-(x,Y)$ for~$|N_G^-(x)\cap Y|$, the \defi{indegree
  of~$x$ in~$Y$}; the \defi{outdegree of~$x$ in~$Y$}, denoted
by~$\deg_G^+(x,Y)$, is~defined similarly. The \defi{semidegree of~$x$
  in~$Y$}, denoted by~$\deg_G^0(x,Y)$, is the~minimum of~those
two~values. We drop the subscript when there is no danger of
confusion, writing~$N^-(x)$,~$\deg^0(x)$, and so forth.
We sometimes symbols such as $\bullet$ or~$\diamond$
as placeholder for either $-$ or~$+$, as in ``if $x\in N^\bullet(y)$
and $y\in N^\bullet(x)$, then $x$ and~$y$ form a cycle of length~$2$'',
in which both occurrences of $\bullet$ are meant to be either $-$ or~$+$.

For
digraphs~$G$ and~$H$, we call~$H$ a \defi{subgraph} of $G$
if~$V(H) \subseteq V(G)$ and~$E(H) \subseteq E(G)$; $H$ is said to
be~\defi{spanning} if~$V(H)=V(G)$. For any set~$X \subseteq V(G)$, we
write~$G[X]$ for the subgraph of~$G$ \defi{induced} by~$X$, which has
vertex set~$X$ and whose edges are all edges of~$G$ with both
endvertices in~$X$. If~$H$ is a subgraph of~$G$ then we write~$G-H$
for~$G\bigl[V(G)\setminus V(H)\bigr]$. Likewise, for a vertex~$v$ or
set of vertices~$S$, we write~$G - v$ or~$G - S$
for~$G\bigl[V(G)\setminus \{v\}\bigr]$
or~$G\bigl[V(G)\setminus S\bigr]$ respectively.
 For disjoint subsets~$X, Y \subseteq V(G)$, where $G$~is
a~digraph, we denote by~$G[X \rarr Y]$, or equivalently
by~\mbox{$G[Y\larr X]$}, the subdigraph of~$G$ with vertex
set~$X \cup Y$ and edge set
$E\bigl(G[X \rarr Y]\bigr) \deq \{\,x\rarr y \in E(G) : x\in X, y\in
Y\,\}$.

An \defi{oriented graph} is a digraph in which
there is at most one edge between each pair of vertices.  Equivalently, an
oriented graph~$G$ can be formed by assigning an orientation to each
edge $\{u,v\}$ of some (undirected) graph~$H$, i.e. by~replacing each
$\{u,v\}\in E(H)$ by one of the possible ordered pairs $(u,v)$
or~$(v,u)$. In this case we refer to~$H$ as the \defi{underlying
  graph} of~$G$, and say that~$G$ is an \defi{orientation} of~$H$. We
refer to the~\defi{maximum degree} of an oriented graph~$G$,
denoted~$\Delta(G)$, to mean the maximum degree of the underlying
graph~$H$.

A~\defi{directed path} of length~$k$ is an oriented graph with vertices~$v_0,\ldots, v_k$ and edges~$v_\imm\rarr v_i$ for
each~$1 \leq i \leq k$.
Likewise, a~\defi{directed cycle} of
length~$k$ is an oriented graph with vertices~$v_1,\ldots,v_k$ and
edges $v_i\rarr v_\ipp$ for each~$1 \leq i \leq k$ with addition taken
modulo~$k$.

A~\defi{tree} is an acyclic connected graph, and an~\defi{oriented
tree} is an orientation of a tree. Where it is clear from
the context that a tree is oriented, we may refer to it simply as a
tree. A~\defi{leaf} in a graph or oriented graph $G$
is a vertex $v\in V(G)$ which is incident to precisely one edge; $v$
is an~\defi{in-leaf} if~$\deg_T^+(v)=1$ and an \defi{out-leaf}
otherwise (note that these definitions are only standard when $G$ is a tree, but we define
leaves, in-leaves and out-leaves for general graphs and digraphs as they are helpful concepts for
working with treelike structures also).
A~\defi{star} is a tree in which at most one~vertex
(the~\defi{centre}) is not a leaf. A \defi{subtree}~$T'$ of a tree~$T$
is a subgraph of~$T$ which is also a tree, and we define subtrees of
oriented trees similarly. A \defi{forest} is a graph in which each component is a tree,
and a \defi{oriented forest} is an orientation of a forest, that is,
a digraph in which each component is an oriented tree.

Given an integer~$d$, a graph $G$ is said to be \defi{$d$-degenerate}
if $\delta(H)\le d$ for every subgraph $H\subseteq G$.
Note that a graph is $1$-degenerate if and only if it is a forest.

If $H$ is an induced subgraph of a graph $G$, then we call $v\in V(H)$
an \defi{attachment} of $H$ if $v$ has a neighbour in~$V(G)\setminus V(H)$, and call $v$ a~\defi{pendant vertex} otherwise.
Attachments and pendant vertices in digraphs are defined similarly, considering the underlying graph.

Let~$A_1, A_2, \dots$ be a sequence of events. We say that~$A_n$ holds
\defi{with high probability} if~\makebox{$\bbp(A_n) \to 1$}
as~\makebox{$n \to \infty$}. Likewise, all occurrences of the standard
asymptotic notation~$\littleo(f)$ refer to sequences~$f(n)$ with
parameter~$n$ as~$n \to \infty$ (i.e., $g=\littleo(f)$ if
$g(n)/f(n)\to 0$ as $n\to\infty$). We will often have sets indexed
by~$\{1,2,\ldots,k\}$, such as~$V_1, \dots, V_k$, and addition of
indices will always be~performed modulo~$k$. Also,
if~$\varphi\colon A\to B$ is a function from~$A$ to~$B$
and~$A'\subseteq A$, then we write~$\varphi(A')$ for the image of~$A'$
under~$\varphi$. We omit floors and ceilings whenever they do not affect
the argument, write~$a = b \pm c$ to indicate
that~$b-c\leq a\leq b+c$, and write $abc/de\!f$ for the fraction
$(abc)/(de\!f)$. We write $\bbn\deq\{1,2,\ldots\}$
for the set of natural numbers. For each~$k \in \bbn$ we denote by~$[k]$ the
set~$\{1,2,\ldots,k\}$, and write~$\binom{S}{k}$ to denote the set of
all~$k$-element subsets of a set~$S$. For any two disjoint sets
$A$~and~$B$, we write $A\dcup B$ for their union. We use the
notation~$x\ll y$ to indicate that for every positive~$y$ there exists
a positive number~$x_0$ such that for every~$0<x<x_0$ the subsequent
statements hold. Such statements with more variables are defined
similarly. We always write~$\log x$ to mean the natural logarithm
of~$x$.

\subsection{Trees and forests}\label{s:trees}

Let~$T$ be a tree or oriented tree.
It is often helpful to nominate a
vertex~$r$ of~$T$ as the~\defi{root} of~$T$; to emphasise this fact we
sometimes refer to~$T$ as a~\defi{rooted tree}. If so, then every
vertex~$x$ other than~$r$ has a~unique~\defi{parent}; this is defined
to be the (sole) neighbour~$p$ of~$x$ in the unique path in~$T$
from~$x$ to~$r$, and~$x$ is said to be a~\defi{child}
of~$p$. In the same way we call $w$ a \defi{descendant} of $v$,
and $v$ an \defi{ancestor} of $w$, if the path from $r$ to $w$
contains $v$. We denote the set of children of $v$ by~\defi{$C(v)$},
and similarly write \defi{$C^-(v)$} for~$N^-(v)\cap C(v)$ and
\defi{$C^+(v)$} for~$N^+(v)\cap C(v)$.

Now let $F$ be a forest or oriented forest, so each component of $F$
is a tree or oriented tree. If we choose a root vertex for each component tree,
then the definitions of parent, child, descendant and ancestor extend naturally to forests:
we say that $u$ is a child (respectively parent, descendant or ancestor) of $v$ in $F$
if $u$ and $v$ are in the same component $T$ of $F$ and $u$ is a child
(respectively parent, descendant or ancestor) of $v$ in $T$.

An~\defi{ancestral order} of the vertices of a rooted tree~$T$
is an order of~$V(T)$ in which every non-root vertex appears later than its parent;
observe that this implies that the roof of $T$ appears first in the order, and that each vertex
appears prior to each of its descendants and later than each of its ancestors. In exactly the
same way, given a forest $F$ and a root vertex for each component tree in $F$,
an~\defi{ancestral order} of the vertices of $F$ is an order of~$V(F)$ in which every non-root vertex
appears later than its parent. Note that an ancestral order of the vertices of a
tree (or forest) specifies the root(s) of the tree (or forest), since the root of each
component must appear prior to every other vertex of that component. Consequently, when considering
trees or forests equipped with ancestral orders we do not separately specify the roots of the tree or tree components.

We say that an ancestral order of a forest $F$ is \defi{tidy} if for any initial segment $\cali$ of the order,
at most $\log_2 |V(F)|$ vertices in $\cali$ have a child not in
$\cali$. Such orders were considered by K\"uhn, Mycroft and Osthus~\cite{KMO11:sumner_approximate}
for the purpose of embedding trees in tournaments; in particular they proved the following lemma for trees (the statement for forests follows immediately).

\begin{lemma}\label{l:tidy}
  \cite[Lemma 2.11]{KMO11:sumner_approximate} If $F$ is a forest in which each component tree is rooted, then $F$
  admits a tidy ancestral order.
\end{lemma}

Our proof strategy at various stages requires us to divide a tree into subtrees, which we capture by the following notion.

\begin{definition}\label{d:tree-partition}
  Let~$T$ be a tree or oriented tree. A~\defi{tree-partition} of~$T$
  is a collection~$\{T_1,\ldots,T_s\}$ of pairwise edge-disjoint subtrees
  of~$T$ such
  that~$\bigcup_{i\in[s]}V(T_i) =
    V(T)$ and~$\bigcup_{i\in[s]}E(T_i) = E(T)$.
\end{definition}

Note that distinct trees in a tree-partition~\calp\ share at most
one vertex; moreover, if~\calp\ contains at least 2~trees,
then each tree has at least one vertex in common with some other tree
in~\calp. The following lemma shows that any tree admits a tree-partition into two subtrees
which splits a given set of vertices somewhat evenly.

\begin{lemma}{{\cite[Lemma 5.7]{MyNa18}}}\label{l:split-tree}
  If~$T$ is a (possibly oriented) tree and~$L\subseteq V(T)$, then $T$ admits
  a~tree-partition~$\{T_1, T_2\}$ such that~$T_1$ and~$T_2$
  each contain at~least~$|L|/3$ vertices of~$L$.
\end{lemma}

We will use the following straightforward corollary to obtain a small set of vertices which splits a forest into not-too-large parts.

\begin{corollary} \label{c:treepieces}
Every forest $F$ on $n$ vertices admits a set $X \subseteq V(F)$ with $|X| \leq 3n^{1/3}$ for which every component of $F - X$ has fewer than $n^{2/3}$ vertices.
\end{corollary}

\begin{proof}
Let $\calt_0$ be the set of components of $F$,
so each $T \in \calt_0$ is a subtree of $F$\!, and let $s \deq |\calt_0|$.
Also let $X_0 = \emptyset$. For each $i \geq 0$ in turn do the following:
if $\calt_{i}$ contains a tree $T$ with $|T - X_i| \geq n^{2/3}$,
then apply Lemma~\ref{l:split-tree} to obtain a~tree-partition~$\{T', T''\}$ of~$T$
for which~$T'$ and~$T''$ each contain at~least~$|T - X_i|/3 \geq n^{2/3}/3$ vertices of $V(T) \setminus X_i$,
and let $v_i$ be the unique vertex in $V(T') \cap V(T'')$.
Set $\calt_{i+1} \deq (\calt_{i} \setminus \{T\}) \cup \{T', T''\}$,
set $X_{i+1} \deq X_{i} \cup \{v_i\}$, and proceed to the next $i$.
If instead $\calt_i$ does not contain a tree $T$ with $|T - X_i| \geq n^{2/3}$ then we terminate
the process, writing $t$ for this terminal value of $i$, and setting $X \deq X_t$.
So $\calt_t$ does not contain a tree $T$ with $|T - X| \geq n^{2/3}$;
since each component of $F - X$ is a subtree of some tree in $\calt_t$
it follows that each component of $F - X$ has fewer than $n^{2/3}$ vertices.
Observe that our definition of $\calt_{i+1}$ and $X_{i+1}$ ensures
that $|\calt_{i+1}| = |\calt_{i}| + 1$ and $|X_{i+1}| = |X_i|+1$ for each $i \in [t]$, so by induction we have $|\calt_{t}| = t+s$ and $|X| = t$. Now consider a tree $T \in \calt_t$. Either $T$ was one of the $s$ trees in $\calt_0$ or $T \in \calt_{i+1} \setminus \calt_i$ for some $i$. In the latter case, the way we chose $T'$ and $T''$ at each step implies that $|T - X_i| \geq n^{2/3}/3$. It follows that $|T - X| \geq n^{2/3}/3 - 1$, since all vertices of $V(T) \cap X_i$ other than $v_i$ are in $V(T) \cap X$ (this is because trees in $\calt_t$ only intersect at vertices in $X$, and $T$ was not subsequently split since $T \in \calt_{t}$). We conclude that all but at most $s$ trees $T$ in $\calt_t$ have $|T - X| \geq n^{2/3}/3 - 1$, so $\calt_t$ contains at least $t = |X|$ trees with this property.
Since each vertex of $V(F) \setminus X$ is in exactly one tree in $\calt_t$, it follows that $|X| (n^{2/3}/3 - 1) \leq |V(F) \setminus X| = n - |X|$, so $|X| \leq 3n^{1/3}$, as required.
\end{proof}

The following proposition will be used at the very start of the proof of Theorem~\ref{t:treelike} to obtain the set $\Vground$; the stated properties will simplify the allocation and embedding process for the tree components that remain after this set is deleted. Note that in the proof we use the folklore result that every tree $T$ contains a vertex $x$ such that each component of $T - x$ has at most $|T|/2$ vertices; this can be proved by orienting each edge of $T$ in the direction that leads to the most vertices (choosing arbitrarily if tied) and taking $x$ to be a sink of the resulting orientation (which must exist since $T$ contains no cycle).

\begin{proposition}\label{p:treeset}
If $T$ is a tree or oriented tree, and $X \subseteq V(T)$, then there exists a set $Y \subseteq V(T)$ with the following properties.
\begin{enumerate}
    \item $X \subseteq Y$,
    \item $1 \leq |Y| \leq \max(6|X|, 1)$,
    \item each component of $T- Y$ contains at most $|T|/2$ vertices, and
    \item \label{p:treeset-iv} each component $T'$ of $T - Y$ has either one or two attachments (i.e. vertices with neighbours in $Y$), each of which has only one neighbour in $Y$. Moreover, if $T'$ has two attachments then these are not adjacent.
\end{enumerate}
\end{proposition}

\begin{proof}
Choose a vertex $x_0 \in V(T)$ such that every component of $T - x_0$ has at most $|T|/2$ vertices. If $|X| = 0$ then $Y := \{x_0\}$ has the desired properties, so assume that $|X| \geq 1$.
Write $X = \{x_1, \dots, x_t\}$ and for each $i \in [t]$ let $P_i$ be the unique path in $T$ from $x_0$ to $x_i$. Then $P^* = \bigcup_{i \in [t]} P_i$ is a subtree of $T$. Let $A$ be the set of vertices of $P^*$ with degree at least three in $P^*$; since $P^*$ has at most $t+1$ leaves we have $|A| \leq t-1$. Now let $\calp$ be the set of all paths $P$ in $P^*$ for which $P$ has length at most three and for which the vertices of $P$ which are in the set $\{x_0\} \cup X \cup A$ are precisely the endvertices of $P$. So in particular $|\calp| \leq \bigl|\{x_0\} \cup X \cup A\bigr|$. Let $B$ be the set of all vertices in paths in $\calp$ which are not in $\{x_0\} \cup X \cup A$, so $|B| \leq 2 |\calp| \leq 2\bigl|\{x_0\} \cup X \cup A\bigr| \leq 4t$. Finally, set $Y = \{x_0\} \cup X \cup A \cup B$, so $|Y| \leq 6t = 6 |X|$ and $X \subseteq Y$; observe also that $Y \subseteq V(P^*)$.

Since $x_0 \in Y$ each component of $T-Y$ contains at most $|T|/2$ vertices. It therefore remains only to verify~\ref{p:treeset-iv}, so let $T'$ be a component of $T - Y$. If $T'\!$ contains no vertices of $V(P^*)$, then there is a unique vertex of $T'$ with a neighbour in $Y$, and that neighbour is also unique. So we may assume that $T'$ contains a vertex of $V(P^*)$, and therefore that $V(T') \cap V(P^*)$ induces a path $P$ in $P^*$ with at least 3 vertices (since vertices in shorter paths were in $B$ so cannot be in $T'$). Note that each vertex of $P$ has degree two in $P^*$ (since higher degree vertices were in $A$ so cannot be in $T'$). It follows that the two endvertices of $P$\,---\,which are not adjacent\,---\,are the only vertices of $T'$ with a neighbour in $Y$, and moreover that each of these endvertices has only one neighbour in $Y$, as required.
\end{proof}

Our final result on trees is a simple proposition which enables us to assume without loss of generality that the oriented graph $Q$ that we wish to embed has at most one component which is a tree.

\begin{proposition}\label{p:addedges}
Fix $\Delta \geq 4$. Let $F$ be a forest with maximum degree $\Delta(F) \leq \Delta$ which contains either at least $p$ pairwise vertex-disjoint bare paths of order 7 or at least $p$ pairwise disjoint edges incident to leaves. Then we may add edges to $F$ so that the resulting graph is a tree $T^*$ with the same properties.
\end{proposition}

\begin{proof}
Let $\calt$ be the set of components of $F$\!, so each $T \in \calt$ is a tree. We form $T^*$ by adding to $F$ the edges of a path which meets each $T \in \calt$ in a single vertex $v_T$; note that we then have $\deg_{T^*}(v_T) \leq \deg_F(v_T) + 2$ for each $T \in \calt$ and $\deg_{T^*}(v) = \deg_F(v)$ for every vertex $v \notin \{v^T : T \in \calt\}$. Suppose first that $F$ contains a set $\calp$ of at least $p$ pairwise vertex-disjoint bare paths of order 7, and for each $T \in \calt$ choose $v_T$ to be a leaf of $T$ (or the unique vertex of $T$ if $|T| = 1$). So each $v_T$ has $\deg_F(v_T) \leq 1$, so $\Delta(T^*) \leq \Delta$, and $\calp$ is a set of bare paths of order 7 in $T^*$, as required. Now suppose instead that $F$ contains a set $\cale$ of at least $p$ pairwise disjoint edges incident to leaves, and consider each $T \in \calt$. If $T$ contains a vertex which is adjacent to at least two leaves of~$T$, then at least one of these leaves is not in an edge of $\cale$, and we choose this leaf as $v_T$. On the other hand, if every vertex of $T$ is adjacent to at most one leaf of $T$ and $|T| \geq 2$, then let $x_1$ and $x_2$ be the first and second vertices respectively of a longest path $P$ in $T$ and set $v_T = x_2$. Since $x_1$ is the only neighbour of $x_2$ which is a leaf, the extremality of $P$ implies that $x_2$ has no neighbours outside $P$. Finally, if $|T| = 1$ then let $v_T$ be the unique vertex of $T$. In each case we have $\deg_F(v_T) \leq 2$ and so $\Delta(T^*) \leq \Delta$. Moreover, $\cale$ is a set of pairwise disjoint edges incident to leaves in $T^*$, since no leaf in an edge in $\cale$ was selected for $v_T$.
\end{proof}

\subsection{Estimates and bounds}
\label{s:bounds}

We write~$\bbe(X)$ for the expectation of a random variable~$X$,
and write $\bbp(A)$ for the probability of an event~$A$.
We use the following well-known Chernoff-type bounds.

\begin{theorem}\cite[Corollary 2.3 and~Theorem~2.10]{JLR00:randomgraphs} \label{t:exp}
  If~$0<a<3/2$ and~$X$ has binomial or hypergeometric distribution,
  then~$\bbp\bigl(\,|X-\bbe(X)|\geq a\bbe(X)\,\bigr)\leq 2\exp(-a^2\bbe(X)/3)$.
\end{theorem}

We also use a concentration result
of McDiarmid~\cite{mcdiarmid98:_concen},
in a form stated by Sudakov and Vondr\'ak~\cite{sudakov10}.

\begin{lemma}
  \label{l:martingale}\cite{mcdiarmid98:_concen,sudakov10}
  Fix~$n\in\bbn$ and let~$X_1,\ldots,X_n$ be random variables taking
  values in~$[0,1]$ such that
  \(\bbe(\,X_i\mid X_1,\ldots,X_\imm\,)\leq a_i\) for each~$i\in[n]$.
  If~$\mu\geq\sum_{i=1}^n a_i$, then for every~$\delta$
  with~$0<\delta<1$ we have
  \[
    \bbp\Bigl(\,\sum_{i=1}^n X_i > (1+\delta)\mu\,\Bigr)\leq \ee^{-\delta^2\mu/3}.
  \]
\end{lemma}

\subsection{Regularity}\label{s:regularity}
As well as the strict notion
of regularity\,---\,that $G$ is $r$-regular if every vertex in $G$ has degree precisely $r$\,---\,we
will also work with an approximate notion of regularity for bipartite graphs, where a graph is~`regular' if its edges are `random-like' in the
sense that they are distributed roughly uniformly. More formally,
let~$G$ be a bipartite graph with vertex classes~$A$ and~$B$. For
any sets~$X \subseteq A$ and~$Y \subseteq B$, we write~$G[X,Y]$ for
the bipartite subgraph of~$G$ with vertex classes~$X$ and~$Y$
whose edges are the edges of~$G$ with one endvertex in each of the
sets~$X$ and~$Y,$ and define the~\defi{density}~$d_G(X,Y)$ of edges
between~$X$ and~$Y$ to be \[d_G(X,Y) \deq \frac{e\bigl(G[X,Y]\bigr)}{|X||Y|}.\]
Let~$d, \eps > 0$. We say that~$G$
is~\defi{$(d,\eps)$-regular} if for all~$X\subseteq A$ and
all~$Y\subseteq B$ such that~$|X|\geq \eps|A|$ and~$|Y|\geq \eps|B|$
we have~$d_G(X,Y) = d \pm \eps$. (To avoid confusion, note that the notion of being $(d, \eps)$-regular will always have two parameters, whereas the notion of an $r$-regular graph or digraph has only one parameter.) The following well-known result
is immediate from this definition.

\begin{lemma}[Slicing lemma]\label{l:slice-pair}
  Fix~$\alpha, \eps, d > 0$ and let~$G$ be a~$(d,\eps)$-regular
  bipartite graph with vertex classes~$A$ and~$B$.  If~$A'\subseteq A$
  and~$B'\subseteq B$ have sizes $|A'| \geq \alpha|A|$
  and~$|B'| \geq \alpha|B|$, then~$G[A',B']$
  is~$(d, \eps/\alpha)$-regular.
\end{lemma}

We say that a bipartite graph~$G$ with vertex classes $A$ and $B$
is~\defi{$(\dplus, \eps)$-regular} if~$G$
is~$(d'\!, \eps)$-regular for some~$d' \geq d$. For small~$\eps$, if~$G$
is~$(d, \eps)$-regular then almost all vertices of~$A$ have degree
close to~$d|B|$ in~$B$ and vice-versa. We say that~$G$ is
`superregular' if no vertex has degree much lower than this. More
precisely,~$G$ is~\defi{$(d,\eps)$-superregular} if~$G$
is~$(\dplus,\eps)$-regular and also for every~$a\in A$ and~$b\in B$ we
have~$\deg(a,B) \geq (d-\eps)|B|$ and~$\deg(b,A) \geq (d-\eps)|A|$. The analogous statement to Lemma~\ref{l:slice-pair} for superregular graphs does not hold in general, since a subset of one vertex class may entirely avoid the neighbourhood of a vertex in the other vertex class. However, we can give a similar statement for uniformly-random subsets of each vertex class.

\begin{lemma}\label{l:sliceSRpair}
Suppose that $1/n \ll \eps \ll d, \beta$ and that $a,b,x,y$ are integers with $\beta n < x \leq a \leq n$ and $\beta n < y \leq b \leq n$.
Let $G$ be a $(d, \eps)$-superregular bipartite graph with vertex sets $A$ and $B$ of size $a$ and $b$.
If we choose $X \subseteq A$ with $|X| = x$ and $Y \subseteq B$ with $|Y| = y$ uniformly at random among all subsets of these sizes, then $G[X, Y]$ is $(d, \eps/\beta)$-superregular with high probability.
\end{lemma}

\begin{proof}
 For each $a \in A$ the random variable $\deg(a, Y)$ has hypergeometric distribution with expectation $\deg(a) |Y|/|B| \geq (d-\eps) |Y|$, so by Theorem~\ref{t:exp} the probability that $\deg(a, Y) < (d-\eps/\beta) |Y|$ declines exponentially with $n$. A similar argument shows that for each $b \in B$ the probability that $\deg(b, X) < (d-\eps/\beta) |X|$ also declines exponentially with $n$. Taking a union bound over these events for each of the at most $2n$ vertices of $G$ we find that, with high probability, none of these events hold. Since $G[X, Y]$ is $(\dplus, \eps/\beta)$-regular by Lemma~\ref{l:slice-pair}, we then have that $G[X, Y]$ is $(d, \eps/\beta)$-superregular.
\end{proof}

The final step in the proof of Theorem~\ref{t:treelike} is to
embed the last few vertices of $Q$ into the equally-few remaining unused vertices
of the host graph $G$. In the case where $Q$ has many pairwise disjoint edges incident to leaves,
we will achieve this by using the following well-known corollary of Hall's marriage theorem,
that every balanced superregular bipartite graph contains a perfect matching.

\begin{lemma}\label{l:matching}
  If $d\geq 2\eps$ and $G$ is a~$(d, \eps)$-superregular
  bipartite graph with equally many vertices in each vertex class, then~$G$ contains a perfect matching.
\end{lemma}

In the case where $Q$ has many pairwise vertex-disjoint edges incident to leaves,
we instead use the following lemma of Koml\'os,
S\'ark\"ozy and Szemer\'edi~\cite{komlos95:_proof_bollob} (they actually stated the lemma only for $\ell=4$, but the statement for larger $\ell$ follows immediately by combining this with Lemma~\ref{l:matching}). For this, say that
a graph~$G$ is an \defi{$\ell$-layer $(d,\eps)$-superregular} graph
if~$V(G)=\bigcup_{i\in[\ell]} V_i$, where $V_1, \dots, V_\ell$ are
pairwise disjoint sets of equal size $\ell$ and $G[V_i,V_\ipp]$ is
$(d,\eps)$-superregular for
each $i\in[\ell-1]$.

\begin{lemma}[\cite{komlos95:_proof_bollob}, Theorem 2.1] \label{l:four-layered}
  For every integer $\ell\ge 4$ and every $d >0$ there
  exist~$\eps$ and $m_0$ such that the following holds for all
  $m\ge m_0$. Let $G$ be an $\ell$-layer $(d,\eps)$-superregular
  graph on $\ell m$~vertices, and let $\pi: V_1 \to V_\ell$ be a bijection. There then exists a set $\calp$ of $m$ pairwise vertex-disjoint paths of order $\ell$ in $G$, such that for each $v \in V_1$ there is a path in $\calp$ with ends $v$ and $\pi(v)$.
\end{lemma}

Let $X$ and $Y$ be disjoint sets of vertices in a digraph $G$.
Observe that the underlying graph of~$G[X \rarr Y]$ is then a bipartite
graph with vertex classes~$X$ and~$Y$. We say that~$G[X \rarr Y]$ is
\defi{$(d, \eps)$-regular}
(respectively \defi{$(\dplus, \eps)$-regular}
or~\defi{$(d, \eps)$-superregular}) to mean that this underlying
graph is $(d, \eps)$-regular (respectively $(\dplus, \eps)$-regular
or~$(d, \eps)$-superregular).  In this way we may apply the previous
results of this subsection to digraphs.

The celebrated Regularity Lemma of
Szemer\'edi~\cite{szemeredi75,szemeredi78:_regul} states that every
sufficiently large graph admits a partition such that almost all pairs
of parts are regular in the sense we discuss here. Alon and
Shapira~\cite{alon94} gave the following analogous result for directed graphs.

\begin{lemma}[Regularity Lemma for digraphs \cite{alon94}]\label{l:regularity}
   For all positive $\eps, K'$ there
  exist $K,\,n_0$ such that if~$G$ is a digraph of order $n\geq n_0$ and
  $d\in[0,1]$, then there exist a partition $V_0,\ldots,V_k$
  of~$V(G)$ and a spanning subgraph $G'$ of~$G$ such that
  \begin{enumerate}
  \item \label{reg1} $K'\leq k\leq K$\,;
  \item \label{reg2} $|V_1|=\cdots=|V_k|$ and $|V_0|<\eps n$\,;
  \item \label{reg3} $\deg_{G'}^+(x)\geq \deg_G^+(x)-(d+\eps)n$ for all $x\in V(G)$\,;
  \item \label{reg4} $\deg_{G'}^-(x)\geq \deg_G^-(x)-(d+\eps)n$ for all $x\in V(G)$\,;
  \item \label{reg5} for all \ink\ the digraph $G'[V_i]$ has no edges\,;
  \item \label{reg6} for all distinct $i,j$ with $1\leq i,j\leq k$, either $G'[V_i \farc V_j]$ is empty or $G'[V_i \farc V_j]$ is $(\dplus, \eps)$-regular.
  \end{enumerate}
\end{lemma}

We refer to the sets~$V_1,\ldots,V_k$ as the \defi{clusters} of~$G$.
For $d\in[0,1]$, the \defi{reduced graph} $R$ with parameters $\eps,\,d$
and~$K'$ of~$G$ is a digraph we obtain by applying
Lemma~\ref{l:regularity} to~$G$ with parameters $\eps, d$ and~$K'$;
the digraph $R$ has vertex set $[k]$ and edges $i\farc j$ precisely
when $G'[V_i\farc V_j]$ has density at least~$d$.
The following lemma will be used to obtain the reduced graph
required in the allocation stage of the proof of Theorem~\ref{t:treelike}.

\begin{lemma}\label{l:reduced-graph}
  Suppose that $\cramped{\frac{1}{n}}\ll \cramped{\frac{1}{K}} \ll \cramped{\frac{1}{K'}} \ll \eps\ll d\ll\eta\ll\alpha$. If $G$ is a
  digraph of order~$n$ with
  $\delta^0(G)\geq \bigl(\frac{1}{2}+\alpha\bigr)n$, then there exist
  an integer $k$ with $K' \leq k \leq K$, a
  partition $V_0, V_1, \cdots, V_k$ of $V(G)$ and a digraph
  \rstar\ with $V(\rstar)=V_0\dcup [k]$ with the following properties.
  \begin{enumerate}[label={(\alph*)}]
  \item\label{rg-i}
    $|V_0|<\eps n$ and $m\deq|V_1|=\cdots=|V_k|$.

  \item\label{rg-ii}

    The pairs $G[V_1\rarr V_{2}],G[V_2\rarr V_{3}],\ldots, G[V_{k-1}\rarr V_{k}],G[V_k\rarr V_1],$ are $(d,\eps)$-superregular.

    \item\label{rg-rstar-ij-edge}
    For all $i,j\in[k]$ we have $i\farc j\in E(\rstar)$ precisely when $G[V_i\rarr V_j]$ is $(\dplus,\eps)$-regular.

  \item\label{rg-rstar-V0-edge} For all $v\in V_0$ and all \ink\ we have $v\barc i\in E(\rstar)$ precisely when $\deg^-(v,V_i)\geq (1/2+\eta)m$,
    and $v\farc i\in E(\rstar)$ precisely when $\deg^+(v,V_i)\geq (1/2+\eta)m$\,.

  \item\label{rg-rstar-deg-i}
    For all \ink\
    we have $\deg_\rstar^0(i,[k])\geq (1/2+\eta)k$.
  \item\label{rg-rstar-deg-ii}
    For all $v\in V_0$ we have $\deg_\rstar^0(v,[k])>\alpha k/2$.
  \end{enumerate}
\end{lemma}

\begin{proof}
  Introduce constants $\eps',d'$ with
  $\cramped{\frac{1}{K'}} \ll \eps'\ll\eps\ll d\ll d'\ll \eta$ and apply the digraph version of
  the Regularity Lemma (Lemma~\ref{l:regularity}) to $G$ with $\eps'$ and $d'$ in place of $\eps'$ and $d$. This yields an integer $k$ with $K' \leq k \leq K$,
  a partition $V(G) = V_0'\dcup V_1'\dcup\cdots\dcup V_k'$ and a spanning subgraph
  $G'$ of $G$ satisfying properties~\ref{reg1}--\ref{reg6} of Lemma~\ref{l:regularity}. Define a graph $R$ with vertex set $[k]$ in which $i \to j$ is an edge of $R$ precisely if $G'[V'_i \to V'_j]$ is $(d'_\geq, \eps')$-regular. For each $i \in [k]$, if we choose a vertex $x \in V'_i$ then $x$ has at least one outneighbour in each of at least
  $(\deg^+_{G'}(x) - |V_0'|)/|V'_1| \geq ((1/2 + \alpha)n - (d+\eps)n - \eps n)/(n/k) \geq (1/2 + \alpha/2)k$
  clusters $V'_j$ with $j \in [k] \setminus \{i\}$. It follows that $\delta^+(R) \geq (1/2 + \alpha/2)k$, and the same argument for inneighbours shows that $\delta^-(R) \geq (1/2 + \alpha/2)k$, so $\delta^0(R) \geq (1/2 + \alpha/2)k$. By Theorem~\ref{t:ham-cycle-n/2+1} (specifically the case of a directed cycle due to Ghouila-Houri~\cite{Ghouila60}) we find that $R$ contains a directed Hamilton cycle. By relabelling the clusters if necessary, we assume without loss of generality that the edges of this cycle are $1 \to 2$, $2 \to 3$, \dots, $k \to 1$.

  By definition of $(\dplus',\eps')$-regularity,
  for each~\ink\ the set $V'_i$ contains at most
  $\eps' n/k$ vertices which have fewer than $(d'-\eps')|V_\ipp'|$
  outneighbours in $V_\ipp'$ and at most $\eps' n/k$ vertices which
  have fewer than $(d'-\eps')|V_\imm'|$ inneighbours in $V_\imm'$. By
  moving $2\eps' n/k$ vertices (including all vertices with atypical
  degrees) from each $V_i'$ to $V_0'$ to obtain new sets $V_0, V_1, \dots, V_k$, we obtain a partition $V_0, V_1, \cdots, V_k$ of $V(G)$ for which~\ref{rg-i}
  and~\ref{rg-ii} hold.

  Let \rstar\ be the digraph with vertex set $V_0\dcup [k]$ and with edges defined as follows. For each distinct $i, j\in[k]$ we have
  $i\farc j\in E(\rstar)$ if $G[V_i\rarr V_j]$ is $(\dplus,\eps)$-regular, and for
  each $v\in V_0$ and each $i\in[k]$ we have $v\farc i\in E(\rstar)$ if
  $\deg_G^+(v,V_i) \geq (1/2+\eta)m$ and $v\barc i\in E(\rstar)$ if
  $\deg_G^-(v,V_i) \geq (1/2+\eta)m$.
  This definition ensures Properties~\ref{rg-rstar-ij-edge} and~\ref{rg-rstar-V0-edge}.
  Observe also that if $G[V'_i\rarr V'_{j}]$ is $(d'_\geq, \eps')$-regular then $G[V_i\rarr V_{j}]$ is $(\dplus, \eps)$-regular, and so for each $i \in[k]$ we have $\deg^0_{R^\star}(i, [k]) \geq \deg^0_{R}(i) \geq (1/2 + \eta)k$, giving~\ref{rg-rstar-deg-i}. Finally, consider any $v \in V_0$, and let $X \subseteq [k]$ be the set of all $i \in [k]$ for which $\deg_G^+(v,V_i) \geq (1/2+\eta)m$.
  Then $v$ has at most $|V_0| < \eps n$ outneighbours in $V_0$, and
  at most $(1/2+\eta)mk$ outneighbours in clusters $V_i$ with $i \notin X$,
  so $v$ has at least $\delta^0(G) - \eps n - (1/2+\eta)mk \geq \alpha n/2$ outneighbours in clusters $V_i$ with $i \in X$; since each cluster contains $m$ vertices it follows that $|X| \geq \alpha n/2m \geq \alpha k/2$, and so we have $\deg_\rstar^+(v,[k])>\alpha k/2$. An identical argument for inneighbours shows that $\deg_\rstar^-(v,[k])>\alpha k/2$, giving~\ref{rg-rstar-deg-ii}.
\end{proof}

\subsection{Homomorphisms}\label{s:homomorphisms}
As described in Section~\ref{s:outline}, a key idea for the proof of Theorem~\ref{t:treelike}
is to find an allocation of the vertices of a subforest $F$ of $Q$ to the clusters of a reduced
graph $R$ of the host graph $G$. This allocation is a homomorphism from $F \to R$ with properties
that will enable us to later embed each vertex of~$F$ into the cluster to which it is embedded.
In preparation for this, we make the following definitions.

Let $H$ and~$G$ be digraphs. A homomorphism~$\varphi: H\to G$ is an
edge-preserving map from~$V(H)$ to~$V(G)$, meaning that every edge
$u\farc v\in E(H)$ is mapped to an
edge~$\varphi(u)\farc \varphi(v)\in E(G)$. The
\defi{$\varphi$-indegree} $\deg_\varphi^-(v)$ of $v\in V(H)$ in $G$ is
$\bigl|\varphi\bigl(N_H^-(v)\bigr)\bigr|$; the
\defi{$\varphi$-outdegree} $\deg_\varphi^+(v)$ of $v$ is defined
similarly, and the \defi{$\varphi$-degree} of~$v$
is~$\deg_\varphi(v)\deq \deg_\varphi^-(v) + \deg_\varphi^+(v)$. Note
that vertices
in~$\varphi\bigl(N_H^-(v)\bigr)\cap\varphi(N_H^+(v)\bigr)$ are counted
twice towards $\deg_\varphi(v)$. The \defi{maximum degree} of $\varphi$ is
$\Delta(\varphi)\deq\max_{v\in V(H)} \deg_\varphi(v)$.

The embedding algorithm used in the proof of Theorem~\ref{t:treelike}
relies on a property of large dense digraphs stated in Lemma~\ref{l:goodness}. The form
in which it is stated here is a generalisation
of that used by K\"uhn, Mycroft and Osthus~\cite[Lemma~2.5]{KMO11:sumner_approximate}. We begin with a
definition.

\begin{definition}\label{d:good}
  Let $\beta,\gamma,m>0$, let $G$ and $R$ be digraphs and let $S$ be
  an oriented star with centre~$c$. Also let~$\varphi$ be a
  homomorphism from $S$ to~$R$, and let
  $J^-\deq \varphi\bigl(N_S^-(c)\bigr)$ and
  $J^+\deq\varphi\bigl(N_S^+(c)\bigr)$, so $|J^-|+|J^+|=\Delta(\varphi)$.
  Let $\calv=\{\,V_i:i\in V(R)\,\}$ be a
  partition of $V(G)$ such that $\beta m\le |V|\le m$ for each $V\in\calv$, fix a subset $U_j^\circ\subseteq V_j$ for each  $\circ\in\{-,+\}$ and~$j\in J^\circ$, and write $\calu^\circ=\{\,U_j^\circ\subseteq V_j: j\in J^\circ\,\}$ for each $\circ\in\{-,+\}$.
  A~subset $X\subseteq V_{\varphi(c)}$ is
  \defi{$(\calu^+, \calu^-,\beta,\gamma,\varphi,m)$-good} for~$S$ if for every
  possible choice of subsets
  $W_j^\circ\subseteq U_j^\circ$ of size $|W_j^\circ| \geq \beta m$ for each $\circ\in\{-,+\}$ and~$j\in J^\circ$,
  there are at least $\gamma \sqrt{m}$
  vertices $v\in X$ such that
  \begin{itemize}
  \item for all $j\in J^-$ we have $\deg^-(v,W_j^-)\geq \gamma m$, and
  \item for all $j\in J^+$ we have $\deg^+(v,W_j^+)\geq \gamma m$.
  \end{itemize}
\end{definition}

Here is some motivation for Definition~\ref{d:good}. We will embed the
vertices of the tree one by one, and, after embedding a vertex
$x\in V(T)$, we reserve sets of vertices for the children of~$x$. If the
reserved sets are always good, then this greedy embedding strategy
will succeed. The next lemma states one sufficient condition
for this to occur.

\begin{lemma}\label{l:goodness}
  Suppose that $\cramped{\frac{1}{m}}\ll\eps\ll\gamma\ll\cramped{\frac{1}{q}},\, \beta,\,d$. Let $G$
  and $R$ be digraphs and let $S$ be an oriented star with
  centre~$c$.
  Let $\varphi$ be a homomorphism from $S$ to~$R$ with $\Delta(\varphi) \leq q$ and let~$J^\circ\deq \varphi\bigl(N_S^\circ(c)\bigr)$ for each $\circ\in\{-,+\}$.
  Let $\calv=\{\,V_i:i\in V(R)\,\}$ be a
  partition of $V(G)$ such that $\beta m\le |V|\le m$ for all $V\in\calv$, and for
  each $\circ\in\{-,+\}$ and~$j\in J^\circ$ let $U_j^\circ\subseteq V_j$ be such that $G[V_{\varphi(c)}\leftarrow U_{\varphi(j^-)}^-]$ and $G[V_{\varphi(c)}\to U_{\varphi(j^+)}^+]$ are $(\dplus,\eps)$-regular
  for all $j^-\in J^-$ and~$j^+\in J^+$. Finally, write $\calu^\circ=\{\,U_j^\circ\subseteq V_j: j\in J^\circ\,\}$ for each $\circ\in\{-,+\}$.
  Then every subset $Y\subseteq V_{\varphi(c)}$ of size $|Y| \geq \gamma m/2$ contains a subset $V$ with $|V| = \sqrt{m}$
  which is
  $(\calu^+, \calu^-,\beta,\gamma,\varphi,m)$-good for $S$.
\end{lemma}

\begin{proof}
Introduce new constants $\eps'$ and $d'$ with $\eps \ll \eps' \ll \gamma \ll d' \ll \cramped{\frac{1}{q}},\, \beta,\,d$.
  By removing leaves if necessary, we may assume that
  $\varphi(x)\neq \varphi(y)$ whenever $x,y$ are both inleaves or both
  outleaves of~$S$. Indeed, this does not change the statement we are
  trying to prove, since the definition of
  $(\calu^+, \calu^-,\beta,\gamma,\varphi,m)$-good is not affected by whether the
  number of in-leaves (respectively out-leaves) of~$S$ mapped to a
  fixed $x\in V(R)$ is precisely~one or another positive integer.
  So we may proceed assuming that $S$~has precisely $q$~leaves~$x_1,\ldots,x_q$. For each $i \in [q]$ let $U_i \deq U_{\varphi(x_i)}^\bullet$, where $\bullet \in \{+, -\}$ is such that $x_i \in N_S^\bullet(c)$.
  For each $i\in[q]$, let $S_i\deq S\bigl[\{c,x_1,\ldots,x_i\}\bigr]$,
  so~$S_q=S$.

  Fix $Y\subseteq V_{\varphi(c)}$
  with~$|Y| \geq \gamma m/2$.
  For each~$t\in[q]$, each $t$-tuple $\bft=(v_1,\ldots,v_t)$
  of distinct vertices with
  $v_j\in U_j$ for each $j\in[t]$, and each subset
  $Z \subseteq Y$, we write $\commneigh{S_t}{\bft}{Z}$ for the set
  of vertices $v\in Z$ such that mapping $c \mapsto v$ and $x_j\mapsto v_j$ for
  each $j\in[t]$ gives a homomorphism from~$S_t$ to~$G$.
  Also let $T$ be the set of tuples
  $(v_1,\ldots,v_q)$ of distinct vertices with
  $v_i\in U_i$ for each $i\in [q]$; call $\bft\in T$
  \defi{bad} if $\bigl|N^S(\bft,Y)\bigr|< d'|Y|$ and
  \defi{good} otherwise.

Let $B$ be the set of all bad tuples in $T$; our first goal is to show that $B$ is small.
To do this, set $V^0 \deq Y$, and suppose that we choose $v_i \in U_i$ for each $i \in [q]$ in turn, at each step setting $V^i\deq N^{S_i} \bigl((v_1,\ldots,v_i),V^{i-1}\bigr)$.
This yields a tuple ${\bft} = (v_1,\ldots,v_q) \in T$. For each $i \in [q]$ say that the choice of $v_i$ is \defi{bad} if $|V^i| < (d/2)|V^{i-1}|$. So if we do not make any bad choices, then for each $i \in [q]$ we have $|V^i| \geq (d/2)|V^{i-1}|$, and in particular $|V^q| \geq (d/2)^q|V^0| \geq d'|Y|
$, so ${\bft}$ is good.
It follows that if the outcome {\bft} is not good, then for some $i \in [q]$ we must have made a bad choice. By definition of $V^i$ this implies that $\deg^\circ(v_i,V^{i-1}) < |V^{i-1}|d/2$, where $\circ \in \{+, -\}$ is such that $c\in N_S^\circ(x_i)$. Now consider the smallest $i$ for which we made a bad choice of $v_i$; this minimality property implies that $|V^{i-1}| \geq (d/2)^{i-1}|V^{0}| \geq (\gamma d^{i-1}/2^{i}) |V_{\varphi(c)}|$,
and so $G[U_i \rarr V^{i-1}]$ (if $\circ = +$) or $G[V^{i-1} \rarr U_i]$ (if $\circ = -$) is $(d_\geq, \eps')
$-regular by Lemma~\ref{l:slice-pair}.
However, since $d/2 < d-\eps'$,
it follows that at most $\eps'|U_i| \leq  \eps' m$
vertices $b_i\in U_i$ have $\deg^\circ(b_i,V^{i-1}) < |V^{i-1}|d/2$, so there are at most $\eps' m$ bad choices for~$v_i$. We conclude that if ${\bft}$ is bad then we must have made the first bad choice when choosing $v_i$ for some $i \in [k]$, at which point there were at most $\eps' m$ options for the bad choice, so in total we have $|B| \leq k\eps' m^q$ (since we have $m$ options for each other choice).

  Now choose a set $X\subseteq Y$ of size $|X|= \sqrt{m}$ uniformly at random. For each $\bft\in T$ which is good, $\bigl|\commneigh{S}{\bft}{X}\bigr|$ has hypergeometric distribution with expectation $\bigl|\commneigh{S}{\bft}{Y}\bigr| (|X|/|Y|) \geq d'|X| = d' \sqrt{m}$.
    So by Theorem~\ref{t:exp} the
    probability that
    $\bigl|\commneigh{S}{\bft}{X}\bigr| < d'\sqrt{m}/2$
    decreases exponentially with $m$. Since $|T\setminus B| \leq |T| \leq (\beta m)^q$, by taking a union bound we conclude that with high probability every $\bft\in T\setminus B$ has
    $\cramped{\bigl|\commneigh{S}{\bft}{X}\bigr|\ge
      d' \sqrt{m}/2
    }$.
    Fix a choice of $X$ for which this event occurs.

  It remains to show that
  $X$ is $(\calu^+, \calu^-,\beta,\gamma,\varphi,m)$-good for~$S$.
  So fix subsets $W_j^\circ \subseteq U_j^\circ$ of size $|W_j^\circ| \geq \beta m$ for each $\circ \in \{+, -\}$ and each $j \in J^\circ$.
  Let $T'$ be the set of tuples $(v_1,\ldots,v_q)$ of distinct vertices with
  $v_i\in W_{\varphi(x_i)}^-$ for each $i\in [q]$ with $x_i \in N^-_S(c)$ and
  $v_i\in W_{\varphi(x_i)}^+$ for each $i\in [q]$ with $x_i \in N^-_S(c)$.
  So~$|T'|\ge(\beta m)^q - qm^{q-1}$, and therefore at least $|T'| - |B| \geq (\beta m)^q - q^2m^{q-1} - k\eps' m^q \geq \beta^q m^q/2$ tuples $\bft\in T'$ are good and so have $\bigl|\commneigh{S}{\bft}{X}\bigr|\geq d'\sqrt{m}/2$.
  Let $\calp\deq \bigl\{\, (v,\bft) : \bft\in T',\, v\in\commneigh{S}{\bft}{X}\,\bigr\}$, so
  \begin{equation}\label{e:calp}
    |\calp|
    \ge  \frac{\beta^q m^q}{2} \cdot \frac{d' \sqrt{m}}{2}
    \ge  \frac{d'\beta^q}{4} m^{q+1/2}.
  \end{equation}
  In particular, at least
  $\gamma \sqrt{m}$ vertices
  $v^\star\in X$ must have $v^* \in\commneigh{S}{\bft}{X}$ for at least
  $\gamma m^q$ tuples $\bft\in T'$. Indeed, we would otherwise contradict~\eqref{e:calp}, since we would then have
  \[
    |\calp|
    < |T'| \cdot \gamma \sqrt{m}
    +  |X| \cdot \gamma m^q
    \leq m^q \cdot \gamma \sqrt{m} + \sqrt{m} \cdot \gamma m^q \leq 2\gamma m^{q+1/2}.
    \]
    Each such~$v^\star \in X$ must then have
  at least $\gamma m^q/m^{q-1} = \gamma m$ inneighbours in $W_j^-$ for each $j \in J^-$ and at least~$\gamma m$ outneighbours in $W_j^+$ for each $j \in J^+$, as required.
\end{proof}

\subsection{Matching vertices}

We use two simple results about matchings, whose straightforward
proofs we omit.

\begin{lemma}\label{l:greedy-cover}
  Let $G$ be a bipartite graph with vertex classes
  \(V\)~and~\(W,\) and suppose every vertex in $V$ has degree at
  least $\eps|W|$. Then there exists a subgraph $H\subseteq G$ such that
  \begin{enumerate}
  \item $\deg_H(v)=1$ for each $v\in V$,
  \item $\deg_H(w)\leq1+\frac{|V|}{\eps|W|}$ for each $w\in W$.
  \end{enumerate}
\end{lemma}

\begin{fact} \cite[Exercise 16.1.6]{BondyMurty08}\label{f:equitable-matchings}
  Let $M$ and $N$ be edge-disjoint matchings in a graph $G$.
  If $|M|>|N|$, then there exist edge-disjoint matchings $M'$ and $N'$ in $G$
  such that $|M'|=|M|-1$, $|N'|=|N|+1$ and~$M'\cup N'=M\cup N$.
\end{fact}

\subsection{Patterns and diamonds}
\label{s:diamonds}

We define a \defi{pattern} $\hat P$ to be a rooted oriented path whose root is a leaf; this means that for a fixed path length~$\ell$, a pattern is determined by whether each edge along the path is directed towards the root (`ascending') or away from it (`descending'), so there are $2^\ell$ distinct patterns with path length $\ell$. Now let $T$ be an oriented tree with a fixed root $r$. Then we can consistently define the pattern of each bare path $P$ in $T$ which does not contain $r$. Indeed, let $r'$ be the unique vertex of $P$ which is closest in $T$ to $r$. Because $P$ is a bare path, $r'$ must be an endvertex of $P$, and then the pattern of $P$ is simply $P$ (with the orientation inherited from $T$) with root $r'$.

Let $P$ be a pattern with three vertices $a, b$ and $c$, where $a$ is the root and $b$ is adjacent to both $a$ and~$c$. A \defi{$P$-diamond} is an orientation of a cycle on four vertices $u, v, v'\!, w$ in which both $\{u, v,w\}$ and $\{u, v'\!, w\}$ induce patterns which are isomorphic to $P$ (with root $u$ in each case).
For example, if $P$ is $a\farc b\farc c$ then the digraph $H$ with
$V(H)=\{u,v,v'\!,w\}$ and $E(H)=\{u\farc v,\, u\farc v'\!,\, v\farc w,\, v'\farc
w\}$ is a $P$-diamond (see Figure~\ref{fig:diamonds}). We call the paths $uvw$
and $uv'w$ the \defi{branches} of the diamond, and say that
the $P$-diamond $H$ has \defi{prefix} $u$, \defi{middle}
$\{v,v'\}$ and \defi{suffix} $w$; we denote the $P$-diamond
with this prefix, middle and suffix by $\gadget u{\mathrlap{v}\phantom{v'\!}}{v'\!}{w}$.
If $P$ is clear from the context, we write \defi{diamond} instead of
$P$-diamond.  A \defi{$P$-diamond path} in a digraph $D$ is a sequence
of $P$-diamonds
$\bigl(\gadget{u_i}{\mathrlap{v_i}\phantom{v_i'}}{v_i'}{w_i}\bigr)_{i=0}^t$ such
that $v_i=v_{i-1}'$ for each $i\in[t]$; we say that this path
\defi{connects} $v_0$ and~$v_t'$. Finally, let $\cald$ be a set of $P$-diamonds in a graph~$G$. We say that $\cald$ is \defi{connecting} in $G$ if for each pair $u, v \in V(G)$ there exists a $P$-diamond path, only using diamonds in $\cald$, which connects $u$ and $v$.

Due to the nature of Theorem~\ref{t:treelike}, most patterns in our proof will have order two or seven. For notational convenience we make the following definitions, which allows us to usefully speak of $P$-diamond graphs for these patterns also. For a pattern $P$ of order seven, let $v_1, \dots, v_7$ be the vertices of $P$ in order as they appear in the path, with $v_1$ being the root, and let $P'$ be the pattern induced in $P$ by $\{v_3, v_4, v_5\}$, with $v_3$ being the root of $P'$. Then we write $P$-diamond to mean $P'$-diamond. Similarly, for a pattern $P$ of order two, let $u$ be the root and $v$ the other vertex of $P$. Form a pattern $P'$ by adding a third vertex $w$ as an outneighbour of $v$ (we keep $u$ as the root of $P'$). Again we then write $P$-diamond to mean $P'$-diamond.

\begin{figure}[t!]
  \begin{center}
    \includegraphics{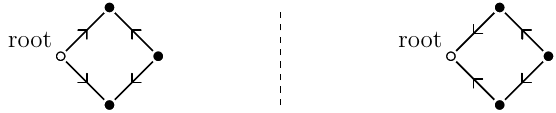}
    \caption[Diamonds]{Left: a
      $(\circ \farc \bullet \barc \bullet)$-diamond.  Right: a
      $(\circ \barc \bullet\barc \bullet)$-diamond ($\circ$~is
      the root of the path).}
    \label{fig:diamonds}
  \end{center}
\end{figure}

\begin{lemma}[$P$-connected subgraphs]\label{l:many-paths/diamonds}
  Let~$G$ be a digraph of order~$n$
  with $\delta^0(G) \geq \bigl(\frac 12 + \alpha\bigr) n$ for some positive $\alpha \leq \cramped{\frac{1}{2}}$.
  If~$P$ is a pattern of order two, three or seven,
  then there exists a connecting set~$\cald$ of $P$-diamonds in~$G$ with $|\cald| = n-1$ and so that no vertex of~$G$ appears in more than $4/\alpha$ diamonds in $\cald$.
\end{lemma}

\begin{proof}
  By the definition of $P$-diamond for patterns of order two or seven, it suffices to prove the lemma in the case when $P$ is a pattern of order three. Also, we may take $V(G)=[n]$.
  For each $i\in[n-1]$, let $\Diamond_i$ be the set of all $P$-diamonds
  with middle $\{i,i+1\}$.
  Let $B_{\mathrm{pref}}$ be
  a bipartite graph with vertex classes
  $\Diamond\deq\{\Diamond_1,\ldots\Diamond_{n-1}\}$ and~$[n]$, with an
  edge between $\Diamond_i$ and $x\in[n]$ if $x$ is a prefix of a
  $P$-diamond in $\Diamond_i$. Since $i$ and $i+1$ each have
  at least $n/2 + \alpha n$ in- and out-neighbours,
  we have $\deg_{B_{\mathrm{pref}}}\bigl(\Diamond_i\bigr)\geq 2\alpha n$. Therefore, by
  Lemma~\ref{l:greedy-cover}, there exists $H_{\mathrm{pref}}\subseteq
  B_{\mathrm{pref}}$ such that each vertex of $\Diamond$ is covered by
  precisely one edge of $H_{\mathrm{pref}}$ and each vertex of~$[n]$
  is covered by at most $1+(n-1)/2\alpha n \leq 2/\alpha$ edges of~$H_{\mathrm{pref}}$. We
  define $B_{\mathrm{suff}}$ similarly for suffixes and obtain the
  corresponding graph $H_{\mathrm{suff}}$. For each $i\in[n-1]$, let $p_i$
  be the neighbour of $\Diamond_i$ in $H_{\mathrm{pref}}$ and let
  $s_i$ be the neighbour of $\Diamond_i$ in $H_{\mathrm{suff}}$, so that the set
  $\cald \deq \{\gadget{p_i}{i}{i+1}{s_i}: i \in [n-1]\}$ is a connecting set of $P$-diamonds in~$G$ with $|\cald| = n-1$. Moreover, each vertex of $G$ appears in at most $2/\alpha + 2 \leq 4/\alpha$ diamonds in $\cald$: at most $1/\alpha$ times as a prefix, $1/\alpha$ times as a suffix, and twice in the middle.
\end{proof}

The next lemma is our main tool for adjusting vertex allocations (see
Figure~\ref{fig:switching}).

\begin{lemma}\label{l:shifting}
  Let $m, k \in \bbn$, let~$R$ be a digraph of order~$k$,
  let~$P$ be a pattern of order two, three or seven, and
  let~$\cald$ be a connecting set of $P$-diamonds in~$R$.
  Also fix integers $\delta_v$ with $|\delta_v| < m$ for each $v\in V(R)$
  such that $\sum_{v\in V(R)} \delta_v = 0$.
  Let $\calp$ be a collection of pairwise vertex-disjoint oriented paths with pattern $P$, and let $Q$ be the oriented graph which is the disjoint union of the paths in $\calp$.
  If $P$ has order two or three, then let $M$ be the set consisting of the second vertex of each path in $\calp$; if instead $P$ has order seven, then let $M$ be the set consisting of the fourth vertex of each path in $\calp$.
  If there exists a
  homomorphism $\varphi: Q \to R$ such
  that for each diamond $D = \gadget{x_i}{y_i}{w_i}{z_i}$ in $\cald$ at least
  $km$ paths in $\calp$ are mapped to
  $x_iy_iz_i$ and at least $km$ paths are mapped to
  $x_iw_iz_i$, then there exists a homomorphism $\varrho:Q\to R$ such
  that $|\varrho^{-1}(v)|=|\varphi^{-1}(v)|+\delta_v$ for all $v\in V(R)$,
  and with the property that $\varrho(v) = \varphi(v)$ for every $v \in V(Q) \setminus M$.
\end{lemma}

\begin{proof}
As in the previous lemma, it suffices to prove the lemma in the case when $P$ is a pattern of order three.
  We proceed iteratively as follows.  Let $u,\,v\in V(R)$ be such that
  $\delta_v<0<\delta_u$, and consider the $P$-diamond path from $u$
  to~$v$. Let $\bigr(\gadget{x_i}{y_i}{w_i}{z_i}\bigr)_{i=1}^t$ be the
  sequence of diamonds in this path, so $u=y_1$ and $v=w_t$. For each
  $i\in [t]$, select a path in $Q$ which is mapped to $x_iy_iw_i$, and
  modify the mapping so that it is now mapped to
  $x_iw_iz_i$ (see Figure~\ref{fig:switching}).  These changes result in a
  homomorphism $\xi : Q \to R$ such that $|\xi^{-1}(u)|=|\varphi^{-1}(u)|-1$ and
  $|\xi^{-1}(v)|=|\varphi^{-1}(v)|+1$, whereas
  $|\xi^{-1}(x)|=|\varphi^{-1}(x)|$ for all $x\in V(R)\setminus \{u,v\}$. So the number of paths mapped to any diamond branch is reduced by at most one; moreover we have $\xi(v) = \varphi(v)$ for every $v \in V(Q) \setminus M$. By iterating this procedure at
  most $\sum_v |\delta_v|\leq km$ times we achieve the desired mapping~$\varrho$; since each diamond initially had at least $km$ paths mapped to each branch, we never run out of paths to edit.
\end{proof}

\begin{figure}[b]
  \begin{center}
    \includegraphics{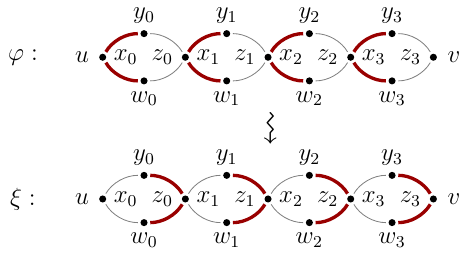}
    \caption[Correcting allocation with $P$-diamonds]{Remapping paths to the other branch of each diamond reduces the number of vertices mapped to $u$ by one and increases the number of vertices mapped to $v$ by one, with the number for each other vertex unchanged.}
    \label{fig:switching}
  \end{center}
\end{figure}

\subsection{Regular expander subdigraphs}
\label{s:regular-expander}

We call a digraph~$D$ an \defi{expander} if
$\bigl|N^-(S)\bigr| > |S|$ and $\bigr|N^+(S)\bigr|> |S|$ for all
nonempty proper subsets $S\subsetneq V(D)$. The next lemma shows that every digraph $G$ which satisfies our minimum degree condition contains a spanning subgraph $H$ which is a regular expander and which is not too dense. Moreover we may insist that $H$ contains a given subgraph of $G$ of bounded maximum degree.

\begin{lemma}\label{l:reg-expander}
  Suppose that $\cramped{\frac{1}{n}}\ll \cramped{\frac{1}{f}},\alpha \leq 1$. Let $G$ be a digraph of order $n$ with
  $\delta^0(G)\geq (\frac 12 + \alpha) n$.
  If $F\subseteq G$ and~$\Delta^0(F)\leq f$, then $G$ contains a
  spanning $d$-regular subdigraph $H$ such that
  \begin{enumerate}

  \item\label{l:reg-expander-F}
    $F \subseteq H$,

  \item $d\leq 25n^{2/3}/\alpha$, and
    \label{l:reg-expander-i}

  \item $H$ is an expander.
    \label{l:reg-expander-ii}
  \end{enumerate}
\end{lemma}

\begin{proof}
Form a subgraph $H_p \subseteq G$ with $V(H_p) = V(G)$ by selecting each edge of $G$ for inclusion
in $H_p$ with probability $p\deq n^{-1/3}$, with the choices for each edge being independent of all other choices.
  \begin{claim}\label{cl:reg-expander-aux}
    With high probability,
    \begin{enumerate}[label={(\alph*)}]

    \item\label{cl:degree}
      every vertex of $H_p$ has in- and outdegree at most $2n^{2/3}$, and

    \item\label{cl:expander}
      if $S$ is a nonempty proper subset of $V(G)$, then
      $\bigl|N_{H_p}^-(S)\bigr|,\bigl|N_{H_p}^+(S)\bigr|>|S|$.
    \end{enumerate}
  \end{claim}

  \begin{claimproof}[Proof of claim]
    Let $x\in V(H_p)$.  Note that $\deg_{H_p}^-(x)$ and $\deg_{H_p}^+(x)$
    are binomial random variables with expectation between
    $(\frac 12 + \alpha)n^{2/3}$ and $n^{2/3}$. By Theorem~\ref{t:exp} (applied with $a=1$) we have
    \begin{equation}\label{e:deg-chernoff}
      \begin{split}
        \bbp\bigl(\deg_{H_p}^-(x) > 2n^{2/3}\bigr) &\leq 2\exp\bigl(-n^{2/3}/6\bigr)\quad\tand\\
        \bbp\bigl(\deg_{H_p}^+(x) > 2n^{2/3}\bigr)
        &\leq 2\exp\bigl(-n^{2/3}/6\bigr).
      \end{split}
    \end{equation}
    Taking a union bound over all $n$ vertices we find that \ref{cl:degree}~holds with high
    probability.  To prove \ref{cl:expander}, we will show that with high probability each nonempty proper subset $S\subseteq V(G)$ has
    $|N_{H_p}^+(S)| >|S|$; the same argument with directions reversed shows that with high probability each nonempty proper subset $S\subseteq V(G)$ has
    $|N_{H_p}^-(S)| >|S|$, giving the desired conclusion. So fix a nonempty proper subset $S\subseteq H_p$, and let $Z_S$ denote the event that $|N_{H_p}^+(S)|\leq|S|$.
    We consider four cases.

    If
    $|S|<n^{1/2}$, then choose any $x \in S$ and recall that $\deg^+_{H_p}(x)$ is a
    binomial random variable
    with expectation at least $(\frac 12 +\alpha)n^{2/3}$.
    Applying Theorem~\ref{t:exp} with $a=1/2$ we find that
    \begin{equation}\label{e:S-small}
      \bbp(Z_S)
      \leq \bbp\left(\deg_{H_p}^+(x) < n^{1/2}\right)
      \leq 2\exp\Bigl(-\frac{n^{2/3}}{24}\Bigr) \leq \exp\Bigl(-\frac{n^{2/3}}{25}\Bigr).
    \end{equation}
    If $n^{1/2}\leq |S|<n/2$, then for $Z_S$ to occur there must exist a set $T \subseteq V(G)$
    with $|T| \geq n/2$ such that $H_p[S \farc T]$ is empty. So fix $T \subseteq V(G)$
    with $|T| \geq n/2$, and observe that each $x \in S$ has $\deg_G^+(x, T) \geq \alpha n$.
    It follows that $e(G[S \farc T]) \geq |S| \alpha n/2 \geq \alpha n^{3/2}/2$, and so
    $e(H_p[S \farc T])$ has a binomial distribution with expectation at least $\alpha n^{7/6}/2$; by
    Theorem~\ref{t:exp} applied with $a=1$ we find the probability that $H_p[S \farc T]$
    is empty to be at most $2\exp(-\alpha n^{7/6}/6)$. Taking a union bound over all $T$ we obtain
    \begin{equation}\label{e:S-medium}
      \bbp(Z_S)
      \leq 2^{n+1} \exp\Bigl(-\frac{\alpha n^{7/6}}{6}\Bigr)
      \leq  \exp\Bigl(-\frac{\alpha n^{7/6}}{10}\Bigr).
    \end{equation}

    Similarly, if $n/2\leq |S| \leq n - n^{1/2}$, then for $Z_S$ to occur there must exist a
    set $T \subseteq V(G)$ with $|T| \geq n^{1/2}$ such that $H_p[S \farc T]$ is empty. So fix
    $T \subseteq V(G)$ with $T \geq n^{1/2}$, and observe that each $x \in T$ has $\deg_G^-(x, S) \geq \alpha n$. So $e(G[S \farc T]) \geq |T| \alpha n/2 \geq \alpha n^{3/2}/2$, and so we obtain~\eqref{e:S-medium} exactly as in the previous case.

    Finally, if $|S| > n-n^{1/2}$, then for each $x \in V(G)$ we have
    $\deg_G^-(x,S)\geq \deg_G^-(x) - n^{1/2} \geq n/2$, so $\deg_{H_p}^-(x,S)$
    is a~binomial random variable with expectation at least~$n^{2/3}/2$.
    Applying Theorem~\ref{t:exp} with $a=1$ we deduce that $\bbp(\deg_{H_p}^-(x,S) = 0) \leq 2\exp(-n^{2/3}/6)$, and so taking a union bound over all vertices we obtain
    \begin{equation}\label{e:S-large}
      \bbp(Z_S)
      \leq \sum_{x\in V(G)} \bbp\Bigl(\deg_{H_p}^-(x, S) = 0 \Bigr)
      \leq 2n \exp\Bigl(-\frac{n^{2/3}}{6}\Bigr) \leq \exp\Bigl(-\frac{n^{2/3}}{25}\Bigr).
    \end{equation}
    Taking a union bound, we obtain the desired bound of
    \begin{align*}
     \bbp\bigl(\bigcup_S Z_S\bigr)
      &\leq \sum_{|S|<\sqrt{n}} \,\,\bbp(Z_S)\quad\,\,
      +\,\,\quad \sum_{\mathclap{\sqrt{n}\leq |S| \leq n-\sqrt{n}}} \,\,\bbp(Z_S) \quad\,\,
      +\,\,\quad \sum_{\mathclap{n-\sqrt{n} < |S|}} \,\,\bbp(Z_S)\\
      &\leq 2\sqrt{n}\binom{\phantom{\sqrt{n}}\mathllap{n}\,\,}{\sqrt{n}\,\,}
      \underbracket{\exp\left(-\frac{n^{2/3}}{25}\right)}_{\text{\eqref{e:S-small} and \eqref{e:S-large}}} +
      2^n\underbracket{\exp\left(-\frac{\alpha n^{7/6}}{10}\right)}_{\text{\eqref{e:S-medium}}}
      = \littleo(1),
    \end{align*}
    where the sums and union each range over all proper nonempty subsets $S\subseteq V(G)$ with the specified sizes, and we use the bounds
    $\binom{n}{\sqrt{n}}\leq
    (\sqrt{n}\,e)^{\sqrt{n}}\leq\exp\bigl(\sqrt{n}\,\ln
    n\bigr)$.
  \end{claimproof}
  Returning to the proof of the lemma,
  fix an outcome of~$H_p$ such that
  \ref{cl:degree}~and~\ref{cl:expander}~hold
  and let~$H'\deq H_p\cup F$.
  Clearly each $H$ with $H'\subseteq H \subseteq G$
  satisfies both~\ref{l:reg-expander-F} and~\ref{l:reg-expander-ii}.
  So to conclude the proof
  it suffices to find such an~$H$
  satisfying~\ref{l:reg-expander-i}.
  By~\ref{cl:degree}, we have
  $\Delta^0(H')\leq 2n^{2/3}+f\leq 3n^{2/3}-1$.
  Hence, by Vizing's theorem, the edges of~$H'$
  can be partitioned into at most $3n^{2/3}$~matchings.
  It follows by repeatedly applying Fact~\ref{f:equitable-matchings} that
  $E(H')$ admits a partition~$\calm$ into matchings
  such that $|\calm|\leq 3n^{2/3}$
  and $|M| - |N| \leq 1$ for all~$M,\,N\in\calm$
  (more precisely, one may repeatedly apply Fact~\ref{f:equitable-matchings} to the largest
  and smallest matching in the partition, until all pairs of matchings have sizes differing by at most~1). By partitioning each matching in $\calm$ into $\lceil\frac{7}{\alpha}\rceil$ pairwise disjoint submatchings with sizes as close as possible, we obtain a partition
  $M_1,\ldots,M_d$ of~$E(H')$ into matchings with
  $d \leq 3n^{2/3}\lceil\frac{7}{\alpha}\rceil \leq   25n^{2/3}/\alpha$,
  such that $|M_i|\le \alpha n/7 + 1 \leq \alpha n/6$ for each $i\in[d]$,
  and $\bigl||M_i|-|M_j|\bigr|\le 1$ for all $i,\,j\in[d]$.
  The following procedure builds the desired~$H$.

  \begin{description}
  \item[Procedure]
    Let $G_0\deq G - E(H')$.  For each $i\in[d]$, in order,
    greedily choose a directed cycle $C_i$ in $G_{i-1}\cup M_i$, such that
    $C_i$ has length $3|M_i|$ and covers all edges in $M_i$; let $C_i'$
    be a directed Hamilton cycle in $G_{i-1}\setminus V(C_i)$ and let
    $G_i\deq G_{i-1}\setminus \bigl(E(C_i)\cup E(C_i')\bigr)$. We set
    $H\deq \bigcup_{i\in[d]} (C_i \cup C_i')$.
  \end{description}

  Let us check that these steps may be carried out. Fix~$i\in[d]$; it
  suffices to show that $C_i$ and $C_i'$ exist. Let
  $M_i=\{u_1\farc v_1,\ldots,u_r\farc v_r\}$. Since
  $\delta^0(G_i)\geq \delta^0(G_0)-d\geq (1/2+3\alpha/4)n$, for each
  pair of vertices $x,y\in V(G_i)$ there exists at least
  $3\alpha n/4 - 2|M_i| \geq |M_i|$ vertices $z$ such that
  $z\in (N_{G_i}^+(x)\cap N_{G_i}^-(y))\setminus V(M_i)$. We may therefore choose
  distinct vertices~$z_1, \dots, z_{|M_i|}$
  with~$z_j\in (N_{G_i}^+(v_j)\cap N_{G_i}^-(u_{j+1}))\setminus V(M_i)$
  for each~$j\in[r]$ (with  addition taken modulo~$r$). Let $C_i$ be the
  cycle~$u_1\farc v_1\farc z_1\farc u_2\farc\cdots\farc u_r\farc
  v_r\farc z_r\farc u_1$. Since $|C_i|= 3|M_i| \leq \alpha n/2$, it
  follows that $\delta^0\bigl(G_i\setminus V(C_i)\bigr)\geq (1/2+\alpha/4)n$, so
  $G_i\setminus V(C_i)$ contains a directed Hamilton cycle~$C_i'$.

  To complete the proof, observe that $H'\subseteq H \subseteq G$ and that $H$ is the edge-disjoint
  union of spanning subdigraphs $C_i\cup C_i'$ of $G$. Since each vertex of $G$ has precisely one
  inneighbour and one outneighbour in $C_i\cup C_i'$ for each $i \in [d]$,
  we conclude that $H$ is a spanning subgraph of~$G$ with
  $\deg_H^-(x)=\deg_H^+(x)=d\leq 25n^{2/3}/\alpha$ for each~$x\in V(H)$,
  so~\ref{l:reg-expander-i} holds.
\end{proof}

We next establish a property of regular expander digraphs $D$ which is crucial for our random allocation strategy, namely that if $X$ is a random vertex of $D$ (with some unspecified probability distribution on the vertices of $D$), and $Y$ is a uniformly random outneighbour (or inneighbour) of $X$, then the distribution of $Y$ is more uniform than the distribution of $X$, except in the case where $X$ is uniformly-distributed, in which case the same is true of $Y$. This property is established in Lemma~\ref{l:walk-on-expander}, using the notation we now introduce.

Let $D$ be a digraph of order $k$, and let $X$ be a random vertex of $D$. Formally speaking this means that $X$ is a random variable with codomain $V(D)$ in some suitable probability space, so, writing $V(D) = \{x_1, \dots, x_k\}$ we have a probability distribution $\bbp(X = x_i) = p_i$ on the vertices of $D$. We then define the variation of $X$ by
\[
  \Var(X) \deq \sum_{x\in V(D)} \left(\bbp(X=x)-\frac{1}{k}\right)^{\!2}.
\]
So $\Var(X)$ is a measure of how uniform the distribution of $X$ is, and in particular $\Var(X) = 0$ if and only if~$X$ is a uniformly-random vertex of $D$.

Consider the set $S \subseteq \bbr^k$ defined by $S = \{(p_1, \dots, p_k) : p_i \geq 0, \sum_{i=1}^k p_i = 1\}$, so $S$ is the set of all possible probability distributions for a random vertex $X$ of $D$. Observe that $S$ is a convex polytope whose vertices are the unit vectors along each axis of the coordinate system. Moreover, the function $\Var: \bbr^k \to \bbr$ given by $\Var((p_1, \dots, p_k)) = \sum_{i=1}^k (p_i-1/k)^2$ is convex, since it is a constant translation of the function which squares each coordinate. It follows that $\Var(\cdot)$ obtains its maximum value on $S$ at a vertex of $S$, and therefore that for any distribution of a random vertex $X$ in $D$ we have
\begin{equation}\label{e:m(.)}
    0 \leq \Var(X) \leq \left(1-\frac{1}{k}\right)^2 + \frac{k-1}{k^2} = 1 - \frac{1}{k} < 1.
\end{equation}

Our next lemma states that for every expander digraph and every assignment of weights to vertices,
there exist vertices with somewhat distinct weights which share
a common inneighbour (and the same is true for outneighbours).

\begin{lemma}\label{l:mix-neighbours}
  Let $D$ be an expander digraph of order~$n$, let $f:V(D)\to \bbr$ and let
  $M \deq \max_{x,\,y\in V(D)} f(x)-f(y)$. If $n\deq |D|\ge 3$ and~$M>0$,
  then there exist $u,\,x,\,y\in V(D)$ such that $x,\,y\in N^-(u)$ and
  $f(y)-f(x)\geq M/(n-1)$ and, similarly, there exist $v,\,w,\,z\in V(D)$
  such that $w,\,z\in N^+(v)$ and $f(w)-f(z)\geq M/(n-1)$.
\end{lemma}

\begin{proof}
  It suffices to prove the existence of~$u,x,y$; the statement for
  $v,w,z$ follows by an identical argument with the roles of inneighbours
  and outneighbours switched. Let $S_1,\ldots,S_r$ be a partition of
  $V(D)$ such that for all $x,\,y\in V(D)$ we have $f(x)=f(y)$ if and only
  if $x,\,y\in S_i$ for some $i\in[r]$. Clearly, $1<r\leq n$.
  Since $f$~is constant in each set of this partition, we write~$f(i)$ for the
  common value of $f$ over all $x\in S_i$. We can assume that the sets
  are labelled so that $f(i)<f(j)$ whenever $i<j$. Note
  that~$M=f(r)-f(1)$, and therefore
  $f(j+1)-f(j)\geq M/(r-1)\geq M/(n-1)$ for some $j\in[r-1]$. Let
  $X\deq S_1\cup\cdots\cup S_j$ and
  $Y\deq S_{j+1}\cup\cdots\cup S_n$. Since $D$ is an expander,
  $\bigl|N^+(X)\bigr| > |X|$ and
  $\bigl|N^+(Y)\bigr|>|Y|$. Because $|X|+|Y|=n$ there must be a vertex
  $u\in N^+(X)\cap N^+(Y)$. Let $x\in X$ and $y\in Y$ be inneighbours
  of $u$. Then $f(y)-f(x)\geq f(j+1)-f(j)\geq M/(n-1)$ as desired.
\end{proof}

We are now ready to establish the key property of regular expanders we need for our allocation strategy. Note here that when we say that $Y$ is a uniformly-random outneighbour of $X$, we mean that the distribution of the random vertex $Y$ can be obtained by first choosing a vertex according to the distribution of $X$, then selecting a uniformly-random outneighbour of the chosen vertex.

\begin{lemma}\label{l:walk-on-expander}
  Let $D$ be an $d$-regular expander digraph of order~$k$. Let $X$ be a random vertex of $D$, and let $Y$ be a uniformly-random outneighbour of $X$. Then
  \begin{equation}\label{e:dist-change}
    \Var(Y) \leq \left(1-\frac{1}{2k^5}\right) \Var(X).
  \end{equation}
The same bound holds if $y$ is a uniformly-random inneighbour of $x$.
\end{lemma}

\begin{proof}
For each $x \in V=V(D)$ write~$f(x)\deq\tfrac{1}{d} \bigl(\bbp(X = x) - \tfrac{1}{k}\bigr)$. We then have
  \begin{equation} \nonumber
    \Var(X)
    = \sum_{x\in V} \left(\bbp(X=x)-\frac{1}{k}\right)^{\!2}
    = \sum_{x\in V} (df(x))^2
    = \sum_{y\in V}\, \sum_{x\in N_D^-(y)} \!\! df(x)^2,
  \end{equation}
where the final equality holds since $D$ is $d$-regular so each vertex $x$ appears in $N_D^-(y)$ for precisely $d$~vertices~$y$. Similarly, since each $x \in V$ has precisely $d$~outneighbours, for each $y \in V$ we have $\bbp(Y=y) = \sum_{x\in N_D^-(y)}  \tfrac{1}{d} \cdot \bbp(X=x)$, so
  \begin{equation}\label{e:uniformity_increases} \nonumber
    \Var(Y)
    = \sum_{y\in V} \left(\bbp(Y=y)-\frac{1}{k}\right)^{\!2}
    = \sum_{y\in V} \Biggl(\sum_{x\in N_D^-(y)}\frac{1}{d} \biggl(\bbp(X=x)-\frac{1}{k}\biggr)\Biggr)^{\!2}
    = \sum_{y\in V} \biggl(\sum_{x\in N_D^-(y)} \!\!\!\!f(x)\biggr)^{\!2}.
  \end{equation}
Combining these expressions we obtain
  \begin{align}
    \Var(X) - \Var(Y)
    &  =  \sum_{y\in V} \Biggl(\, \sum_{x\in N_D^-(y)} \!\! df(x)^2 - \biggl(\sum_{x\in N_D^-(y)} f(x)\biggr)^{\!2}\, \Biggr) \\
    &  =  \sum_{y\in V} \Biggl( \biggr( \sum_{z, w \in N_D^-(y)} \frac{1}{2}f(z)^2 + \frac{1}{2}f(w)^2\biggr) - \biggl(\sum_{z, w \in N_D^-(y)} f(z)f(w)\biggr) \Biggr),\nonumber\\
    &  =  \sum_{y\in V} \biggl( \sum_{z, w \in N_D^-(y)} \frac{1}{2}\bigl(f(z)-f(w)\bigr)^2 \biggr) \geq \frac{1}{2}\, \max_{\mathclap{\substack{y\in V\\z,w\in N_D^-(y)}}}\, \bigl(f(z) - f(w)\bigr)^2, \nonumber
  \end{align}
where the second equality holds since $|N^-_D(y)| = d$ for each $y \in V$, meaning that the term $\tfrac{1}{2} f(x)^2$ is counted precisely $2d$ times in the latter expression, whilst the final inequality holds simply because a sum of non-negative terms is at least as large as its maximum term.
Write $L \deq \max_{x \in V} \bigl|f(x)\bigr|$. We then have
\[
  \Var(X) - \Var(Y) \geq \frac{1}{2} \max_{\mathclap{\substack{y\in V\\z,w\in N_D^-(y)}}} \bigl(f(z) - f(w)\bigr)^2 \geq \frac{\max_{u, v \in V} \bigl(f(u) - f(v)\bigr)^2}{2(k-1)^2}  \geq \frac{L^2}{2(k-1)^2} \geq \frac{\Var(X)}{2k(k-1)^2d^2} \geq \frac{\Var(X)}{2k^5}.
\]
Indeed, the first inequality is~\eqref{e:uniformity_increases}, whilst the second holds by Lemma~\ref{l:mix-neighbours}. The third holds since $\sum_{x\in V} f(x) = 0$ and so $\max_{u, v \in V} (f(u) - f(v)) \geq L$, whilst the fourth holds since $\Var(X) = \sum_{x \in X} \bigl(df(x)\bigr)^2 \leq k d^2L^2$.

This completes the proof in the case where $Y$ is a uniformly-random outneighbour of $X$; the argument for~$Y$ being a uniformly-random inneighbour of $X$ is identical with the roles of inneighbours and outneighbours switched.
\end{proof}

\subsection{General allocation algorithm}
\label{s:gen-allocation-algo}
For the following algorithm, recall that if a tree $T$ is a component of a forest $F$\!, then the root of~$T$ appears prior to each other vertex of~$T$ in any ancestral order on the vertices of~$F$; together with the stipulation that the root of each component of~$F$ lies in $Z$, this ensures that we may always take $t_\sigma$ to be the parent of~$\ttau$ when required to do so. Note also that the algorithm makes arbitrary choices of~$\varphi(\ttau)$ for vertices $\ttau$ in $Z$. When we apply the algorithm later we will specify how these arbitrary choices should be made; the point of not specifying this here is that the results we prove in this section about Algorithm~\ref{a:alloc} hold no matter how we subsequently do this.

\smallskip
\begin{algorithm}[H]
  \SetAlgoVlined
  \caption{The Vertex Allocation Algorithm}\label{a:alloc}
  \DontPrintSemicolon \SetKwInOut{Input}{Input} \Input{an oriented
    forest~$F$\!, an ancestral
    order $t_1,\ldots,t_n$ of~$V(F)$, a digraph $D$ with
    $\delta^0(D)\geq 1$, a set $Z \subseteq V(F)$ which contains the root of each component of~$F$\!, a set $\cale \subseteq E(F)$, and for each $e \in \cale$ a permutation $\pi_e$ of~$V(F)$.}
  \DontPrintSemicolon \SetKwInOut{Output}{Output} \Output{a map $\varphi: V(F) \to V(D)$.}
  \For{$\tau = 1$ \KwTo $n$}{
    \lIf{$\ttau \in Z$}
    {choose $\varphi(\ttau)$ arbitrarily.}
    \Else
    { Let~$t_\sigma$ be the parent of~$\ttau$, let $x_\sigma=\varphi(t_\sigma)$ and let $e$ be the edge of~$F$ between $t_\sigma$ and $\ttau$.\;
      \lIf{$e \notin \cale$}{set $\varphi(\ttau)\deq \begin{cases} c_\sigma^+ & \mbox{ if $\ttau \in N_F^+(t_\sigma)$,} \\  c_\sigma^- & \mbox{ if $\ttau \in N_F^-(t_\sigma)$}\end{cases}$.}
      \lIf{$e \in \cale$}{set $\varphi(\ttau)\deq \begin{cases} \pi_e(x_\sigma) & \mbox{ if $\ttau \in N_F^+(t_\sigma)$,} \\  \pi_e^{-1}(x_\sigma) & \mbox{ if $\ttau \in N_F^-(t_\sigma)$.}\end{cases}$.}}
      Pick $c_\tau^+\in N_D^+(x_\tau)$ and $c_\tau^-\in N_D^-(x_\tau)$ uniformly at random independently of all previous choices.
      }
\end{algorithm}
\smallskip

Recall the proof outline of Theorem~\ref{t:treelike} in Section~\ref{s:outline}: we aim to allocate the vertices of a forest $F$ to the vertices of a reduced graph $R$ so that vertices in $F$ are allocated to a uniformly-random in- or out-neighbour of their parent (according to the direction of the corresponding edge), whilst some edges (those in chosen bare paths) should be allocated along a Hamilton cycle $1 \farc 2 \farc \cdots \farc k \farc 1$ in $R$. These edges will be contained in the set $\cale$ given as input to Algorithm~\ref{a:alloc}, and indeed moving one step on the Hamilton cycle is a permutation on the vertex set, as required. Also, the allocation of some vertices of~$F$ will be constrained by the existence of neighbours outside $F$ which have already been allocated; these vertices form the set $Z$ given as input to Algorithm~\ref{a:alloc}, and can be allocated appropriately at the point the algorithm arrives at them. Our aim at this point, then, is to show that so long as $Z$ is not too large, and there are no long paths entirely in~$\cale$, Algorithm~\ref{a:alloc} will yield an approximately-uniform allocation of vertices of~$F$ among the vertices of~$R$ (corresponding to clusters of~$G$). This is asserted by the main result of this section, Lemma~\ref{l:vtx-distribution}. Before that, we give a preliminary result describing how, for a long path $P$ in $F$ with no vertex from $Z$ (except possibly the initial vertex of~$P$) which contains many edges not in $\cale$, the allocation of the final vertex in $P$ by~Algorithm~\ref{a:alloc} is essentially independent of the allocation of the initial vertex of~$P$.

Let $F$ be a forest and let $\cale$ be a set of edges of~$F$. For vertices $u, v$ in the same component of~$F$\!, we define the $\cale$-distance from $u$ to $v$, denoted $\dist_\cale(u, v)$, to be the number of edges of~$\cale$ in the (unique) path in $F$ from $u$ to $v$ (actually we will mainly work with $\dist_{\overline{\cale}}(u,v)$ where $\overline{\cale}$ denotes the complement $E(F) \setminus \cale$).

\begin{proposition} \label{p:singlestep}
Let $R$ be a digraph with vertex set $[k]$ which is a $d$-regular expander. Let $F$ be an oriented forest with a fixed ancestral order, let $Z \subseteq V(F)$ be a set which contains the root of each component of~$F$\!, let $\cale$ be a set of edges of~$F$\!, and for each edge $e \in \cale$ let $\pi_e: [k] \to [k]$ be a permutation. Let $u_0, u_1, \dots, u_\ell$ be the vertices of a path in $F$ for which $u_0$ is an ancestor of~$u_\ell$, $\dist_{\overline{\cale}}(u_0, u_\ell) \geq 4k^5\log m$, and $u_i \notin Z$ for each $1 \leq i \leq \ell$. If we apply Algorithm~\ref{a:alloc} to obtain a map $\varphi : V(F) \to V(R)$
    then for all~$x,y\in V(D)$,
    \[
      \bbp\left(\,\varphi(u_\ell) = y \,\mid\, \varphi(u_0) = x\,\right)
      =\frac{1}{k}\pm \frac{1}{m}.
    \]
\end{proposition}

\begin{proof}
For each $0 \leq i \leq \ell$ let $X_i$ be the value of~$\varphi(u_i)$ conditioned on the event that $\varphi(u_0) = x$, so $X_i$ is a random vertex of~$D$. In particular, $X_0$ then takes value $x$ with probability $1$, so $\Var(X_0) = 1-1/k < 1$. Now consider some $i \in [\ell]$. If the edge $e$ of~$F$ between $u_{i-1}$ and $u_i$ is in $\cale$, then $X_i = \pi_e(X_{i-1})$ or $X_i = \pi^{-1}_e(X_{i-1})$ according to the direction of this edge; in either case we have $\Var(X_i) = \Var(X_{i-1})$. On the other hand, if the edge of~$F$ between $u_{i-1}$ and $u_i$ is not in $\cale$, then $X_i$ is a uniformly-random outneighbour or inneighbour of~$X_{i-1}$, again according to the direction of this edge. In either case we have $\Var(X_i) \leq (1-1/2k^5)\Var(X_{i-1})$ by Lemma~\ref{l:walk-on-expander}. We conclude that
\[
  \Var(X_\ell) \leq \left(1-\frac{1}{2k^5}\right)^{\dist_{\overline{\cale}}(u_0, u_\ell)} \Var(X_0) < \left(1-\frac{1}{2k^5}\right)^{\dist_{\overline{\cale}}(u_0, u_\ell)}.
\]
By definition of~$\Var$ we have $\bbp(X_\ell = y \mid \varphi(u_0) = x) = \tfrac{1}{k} \pm \sqrt{\Var(X_\ell)}$ for each $y \in V(D)$, so
\begin{align*}
      \bbp \left( \,\varphi(u_\ell) = y \, \mid \, \varphi(u_0) = x \, \right)
      &= \frac{1}{k} \pm \left(1-\frac{1}{2k^5}\right)^{\dist_{\overline{\cale}}(u_0, u_\ell)/2}
      = \frac 1k \pm \left(1-\frac{1}{2k^5}\right)^{2k^5 \log m}
      = \frac 1k \pm \frac 1{m},
\end{align*}
where the final inequality holds by the standard inequality $(1-x) \leq \ee^{-x}$.
\end{proof}

\begin{lemma}\label{l:vtx-distribution}
Suppose that $\cramped{\frac{1}{n}} \ll \cramped{\frac{1}{k}}$. Let $R$ be a digraph with vertex set $[k]$ which is a $d$-regular expander. Let $F$ be an oriented forest on $n$ vertices with a fixed ancestral order and let $Z \subseteq V(F)$ be a set which contains the root of each component of~$F$. Let $\cale$ be a set of edges of~$F$ which does not contain a path of length 7, and for each edge $e \in \cale$ let $\pi_e: [k] \to [k]$ be a permutation. If we apply Algorithm~\ref{a:alloc} to obtain a map $\varphi : V(F) \to V(R)$, then for each set $S \subseteq V(F)$,
 with high probability we have for each $i \in [k]$ that
 \[
   |\varphi^{-1}(i) \cap S| = \frac{|S|}{k}  \pm \left(\frac{n}{\log n} + (|Z|+3n^{1/3})\Delta(F)^{56k^5\log \log n}\right).
 \]
\end{lemma}

\begin{proof}
We may assume that $|S| \geq n/\log n$ as otherwise there is nothing to prove.

By Corollary~\ref{c:treepieces} we may choose a set $Y \subseteq V(F)$ with $|Y| \leq 3n^{1/3}$ such that every component of~$F - Y$ has size at most $n^{2/3}$.
Write $Z^* \deq Y \cup Z$, so $|Z^*| \leq |Z| + 3n^{1/3}$. Also let $T_1, \dots, T_s$ be the components of~$F - Z^*$, so $|T_i| \leq n^{2/3}$ for each $i \in [s]$, and for each $i \in [s]$ let $z_i$ be the nearest ancestor in~$Z^*$ of~vertices in~$T_i$; note that $z_i$ exists since the root of each component of~$F$ is in $Z$, and $z_i$ is well-defined since each vertex in $T_i$ has the same nearest ancestor in~$Z^*$\!. Moreover, choose the indices of the $T_i$ and $z_i$ so that if $z_i$ is an ancestor of~$z_j$ then $i < j$ (this can be achieved by having the order $z_1, z_2, \dots, z_s$ be the restriction of the ancestral order on $F$ to the vertices $z_i$). Let $B$ be the set of all vertices $z \in V(F)$ for which there exists $x \in Z^*$ which is an ancestor of~$z$ with $\dist_F(x, z) \leq 56k^5 \log \log n$, so $|B| \leq |Z^*|\Delta(F)^{56k^5 \log \log n}$. For each $i \in [s]$ set $F_i \deq V(T_i) \setminus B$ for each $i \in [s]$. So the sets $F_1, \dots, F_s$ are pairwise disjoint subsets of~$V(F)$ with the following properties.
  \begin{enumerate}
  \item \label{i:Fi-i}
    $\bigl|\bigcup_{i\in[s]} F_i\bigr|\geq n-(|Z|+3n^{1/3})\Delta(F)^{56k^5\log \log n}$, since every vertex of~$F$ is in some $F_i$ except for the vertices in $B$.
  \item \label{i:Fi-ii}
    $|F_i| \leq |T_i| \leq n^{2/3}$ for each~$i\in[s]$.
  \item \label{i:Fi-iii}
    For each~$i\in[s]$,
    each~$x \in \bigcup_{j<i}V(F_j)$, and each~$y\in F_i$, either there is no path from~$x$ to~$y$ in~$F$\!, or the path from $x$ to $y$ in $F$ includes~$z_i$.
  \item \label{i:Fi-iv}
    For any~$i \in [s]$ and~$y\in F_i$ we
    have~$\dist_F(z_i,y)\geq 56k^5\log \log n$.
  \end{enumerate}
Define random variables~$X_i^j$ for each~$i\in[s]$ and~$j\in[k]$ by
\[
  X_i^j \deq \frac{|\varphi^{-1}(j) \cap F_i \cap S|}{n^{2/3}}
\]
so $X_i^j$ is the number of vertices of~$F_i \cap S$ allocated to cluster $j$, normalised by $n^{2/3}$ so that, by~\ref{i:Fi-ii},
each~$X_i^j$ lies in the range~$[0,1]$.
Observe, crucially, that~\ref{i:Fi-iii} implies that for each $i \in [s]$ the allocation of vertices in~$F_i$
  conditioned on the value of~$\varphi(z_i)$ is independent of $\{\varphi(w) : w \in \bigcup_{j<i}F_j\}$. Indeed, for any fixed value of $\varphi(z_i)$ the value of $\varphi(x)$ for $x \in F_i$ depends only on the random choices made by Algorithm~\ref{a:alloc} for descendants of $z_i$ on the path between $z_i$ and $x$, none of which are in $Z^*$ by choice of $z_i$, and each of these choices are independent of all other choices made by the algorithm.
  Hence, for each~$q \in [k]$,
  we have~$\bbe(\,X_i^j \mid X_\imm^j,\ldots,X_1^j, \varphi(z_i) = q\,) = \bbe(\,X_i^j\mid \varphi(z_i) = q\,)$.
It follows that for every~$i \in [s]$ and~$\jnk$ we have
  \begin{align*}
    \bbe(\,X_i^j \mid X_\imm^j,\ldots,X_1^j\,)
    &\leq
    \max_{q\in[k]} \,\bbe(\,X_i^j \mid X_\imm^j,\ldots,X_1^j, \varphi(z_i) = q\,)
    =
    \max_{q\in[k]} \,\bbe(\,X_i^j \mid \varphi(z_i) = q\,)\\
    &=
    \max_{q\in[k]} \frac{\sum_{x\in F_i \cap S} \bbp(\,x\in V_j\mid \varphi(z_i) = q\,)}{n^{2/3}}
    \leq \frac{1}{k} \left(1 +\frac{1}{2 \log n}\right)\frac{|F_i\cap S|}{n^{2/3}}.
  \end{align*}
  To see that the final inequality holds, note that by~\ref{i:Fi-iv} we have $\dist_F(z_i, y) \geq 56k^5 \log \log n$; since $\cale$ does not contain a path of length seven, at least a seventh of the edges on the path between $z_i$ and $y$ are not in~$\cale$, and it follows that $\dist_{\overline{\cale}}(z_i, y) \geq 8k^5 \log \log n \geq 4k^5 \log (2k \log n)$. So we may apply Proposition~\ref{p:singlestep} with~$2k \log n$ in place of $m$, giving the desired inequality.

  We apply Lemma~\ref{l:martingale} with
  \begin{align*}
    \mu
    \deq \frac{1}{k} \left(1+\frac{1}{2 \log n}\right) \frac{|S|}{n^{2/3}}
    & \geq \frac{1}{k} \left(1+\frac{1}{2 \log n}\right) \sum_{i\in[s]}
      \frac{|F_i\cap S|}{n^{2/3}},
\end{align*}
which (since~$\cramped{\frac{1}{n}}\ll \cramped{\frac{1}{k}}$) yields
    \begin{align*}
    \bbp\Bigl(\sum_{i\in[s]} X_i^j > \Bigl(1+\frac{1}{3\log n}\Bigr)\mu\Bigr)
    & \leq   \exp\left(\frac{-\mu}{27 (\log n)^2}\right)
      \leq      \exp \left(-\frac{|S|}{27kn^{2/3}(\log n)^2}\right)
      \leq   \exp\bigl(-n^{1/4}\bigr).
  \end{align*}
Taking a union bound we find that with high probability, for every $i \in [s]$ and $j \in [k]$ the event described does not occur, implying that for every $j \in [k]$ we have
  \begin{align*}
    n^{2/3}\sum_{i\in[s]} X_i^j
    \leq n^{2/3}\left(1+\frac{1}{3 \log n}\right)\mu
    \leq \frac{|S|}{k} \left(1 + \frac{1}{\log n}\right).
  \end{align*}
In other words, for each $j \in [k]$ we have $\bigl|\varphi^{-1}(j) \cap S \cap \bigcup_{i \in [s]} F_i\bigr| \leq \frac{|S|}{k} \left(1 + \frac{1}{\log n}\right)$. Since every vertex is in $\varphi^{-1}(j)$ for precisely one $j \in [k]$, it follows that for each $j \in [k]$ we have
\begin{align*}
  \frac{|S|}{k} \left(1 + \frac{1}{\log n}\right) + \Bigl|S \setminus \bigcup_{i \in [s]} F_i\Bigr| \geq \bigl|\varphi^{-1}(j) \cap S\bigr|
  & \geq \Bigl|S \cap \bigcup_{i \in [s]} F_i\Bigr| - \frac{(k-1)|S|}{k} \left(1 + \frac{1}{\log n}\right)\\
  & \geq \frac{|S|}{k} - \frac{|S|}{\log n} - \left|S \setminus \bigcup_{i \in [s]} F_i\right|,
\end{align*}
and together with~\ref{i:Fi-i} this gives the desired conclusion.
\end{proof}

\section{Proof of Theorem~\ref{t:treelike}}
\label{s:treelike}

Note that Theorem~\ref{t:treelike}\,\ref{i:treelike-quasi-spanning}
follows from Theorem~\ref{t:treelike}\,\ref{i:treelike-spanning}
by appending to $Q$ a directed path of order $|G| - n \geq \alpha n$
joined to $Q$ by a single edge (yielding a graph~$Q'$ of order~$|V(G)|$),
and adjusting the remaining constants accordingly.
Hence, it suffices to prove Theorem~\ref{t:treelike}\,\ref{i:treelike-spanning}. We do this by establishing the following more general result (to see that this implies Theorem~\ref{t:treelike}\,\ref{i:treelike-spanning}, observe that the fact that every edge of $Q_0$ is subdivided at least once implies that $Q$ is 2-degenerate, and deleting the vertices of $Q_0$ from $Q$ yields a $1$-degenerate subgraph).

\begin{theorem}\label{t:main-degen}
Suppose $\cramped{\frac{1}{n}}\ll\lambda\ll\alpha$.
Let $Q$ be a $2$-degenerate graph of order~$n$
with maximum degree $\Delta(Q)\le \exp(\sqrt{\log n})$,
and suppose that $Q$ can be made $1$-degenerate
by deleting from it at most~$n^{0.99}$~vertices.
If $Q$~contains either $\lambda n$~pairwise vertex-disjoint bare paths of order 7 or
$\lambda n$~pairwise disjoint edges incident to~leaves,
then every orientation of
$Q$ is contained in every directed graph $G$ of order $n$
with $\delta^0(G)\ge (1/2+\alpha)n$.
\end{theorem}

To prove Theorem~\ref{t:main-degen}, introduce new constants~$K, K', \eps,\eps',\gamma,\beta,d,\eta$
with
\[
  \frac{1}{n}    \ll  \frac{1}{K} \ll \frac{1}{K'}     \ll
  \eps         \ll  \eps'    \ll
  \gamma       \ll   \beta         \ll  d
  \ll  \lambda  \ll  \eta          \ll  \alpha .
\]
Let $\calp_{\mathrm{undir}}$ be a collection of $\lambda n$ pairwise vertex-disjoint paths in $Q$, where either each $P \in \calp_{\mathrm{undir}}$ is a bare path of order 7, or each $P \in \calp_{\mathrm{undir}}$ is a single edge (i.e., a path of length 1) one of whose endvertices is a leaf. Fix an arbitrary orientation of $Q$; our goal is then to construct an embedding~$\varrho$ of~$Q$ into~$G$.

\subsection{Anatomy of the treelike structure}
\label{s:main-proof/anatomy-of-Q}
We assume without loss of generality that at most one component of $Q$ is a tree. We may do this because if more than one component of $Q$ is a tree, then by Proposition~\ref{p:addedges} we may add edges to $Q$ to connect all these trees into a single tree without affecting the conditions on $Q$ in the statement of the theorem.

Let $Q_0$ be a set of at most $n^{0.99}$ vertices of $Q$
whose deletion turns $Q$ into a $1$-degenerate graph.
In other words, every component of $Q-Q_0$ is a tree;
for each such tree~$T$, let $A_T$ be the set of attachments of~$T$ in~$Q$
(these are vertices of~$T$ with neighbours in~$Q_0$).
Observe that the total number of attachments, over all components, is then $|\bigcup_T A_T| \leq |Q_0| \Delta(Q)$, where the union is taken over all
components $T$ of~$Q-Q_0$. Also, our previous assumption implies that at most one tree has no attachments.

We apply Proposition~\ref{p:treeset} to each component~$T$ of $Q-Q_0$, with $A_T$ in place of $X$;
this yields a set $Y_T \subseteq V(T)$ with~$|Y_T|\le \max(6\bigl|A_T\bigr|, 1)$ such that $A_T \subseteq Y_T$, so that each component of $T - Y_T$ contains at most $|T|/2 \leq n/2$ vertices, and such that each component $T'$ of~$T - Y_T$ has either one or two attachments in $T$ (these are vertices with neighbours in~$Y_T$), each of which has only one neighbour in $Y_T$ (and hence in $V(Q)\setminus V(T)$), with the additional property that if $T'$ has two attachments in $T$ then these are not adjacent.

We set $\Vground \deq Q_0 \cup \bigcup_{T} Y_T$, where the union is taken over all
components $T$ of~$Q-Q_0$. So
\begin{equation}\label{e:|A_0|}
  |\Vground| \le |Q_0| + \Bigl|\bigcup_T Y_T\Bigr| \leq |Q_0| + 6\Bigl|\bigcup_T A_T\Bigr| + 1 \leq |Q_0| (1 + 6\Delta(Q)) + 1
  < 7 n^{0.99}\exp(\sqrt{\log n}) \leq n^{0.995}.
\end{equation}

We also set $F = Q - \Vground$, so $F$ is a forest, and let $\calt$ be the set of components of $F$. So each $T \in \calt$ is an oriented tree with $|T| \leq n/2$, and our choice of $\Vground$ ensures that each $T \in \calt$ has either one or two attachments in $Q$ (these are vertices with a neighbour in $\Vground$), that each attachment has at most one neighbour in $\Vground$, and that if $T$ has two attachments, then these vertices are not adjacent in~$T$.
Moreover, by~\eqref{e:|A_0|} we have
\begin{equation}\label{e:|calt|}
  |\calt|
  \le |\Vground|\Delta(Q) < n^{0.995}\exp(\sqrt{\log n}) \leq n^{0.999}.
\end{equation}

For each $T \in \calt$ fix $r^T$ as the root of $T$.
We may then fix a tidy ancestral order $\prec$ of the forest $F$ by Lemma~\ref{l:tidy}.
So in particular, $r^T$ appears before any other vertex of $T$ in the order $\prec$.
If $T$ has two attachments in $Q$ then we refer to the other attachment
as the \defi{secondary attachment}~$s^T$ of~$T$.
Let ${\hat r}^T$ (respectively, ${\hat s}^T$)
denote the sole neighbour of~$r^T$ (respectively, $s^T$)
which lies outside of~$T$, and
let~$p^T$ denote the parent of~$s^T$ in~$T$
(where we consider $T$~rooted at~$r^T$). So the vertices $r^T, p^T, s^T, {\hat r}^T$
and ${\hat s}^T$ are all distinct.
We partition~$\calt$ into sets $\calt_1$ and $\calt_2$ such that
\[n/3 \leq \Bigl|\bigcup_{T \in \calt_1} V(T)\Bigr|, \Bigl|\bigcup_{T \in \calt_2} V(T)\Bigr| \leq 2n/3,\]
which is possible since $|T| \leq n/2$ for every $T \in \calt$.
Call a path $P\in\calp_{\mathrm{undir}}$ \defi{unfit} if $P$ either
\begin{enumerate}
\item contains a vertex of $\Vground$, or
\item contains either $p^T,\,r^T$ or $s^T$ for some $T\in\calt$.
\end{enumerate}
We call the remaining paths in~$\calp_{\mathrm{undir}}$~\defi{fit}.
Observe that~$|\calp_{\mathrm{undir}}|= \lambda n$, and $\calp_{\mathrm{undir}}$ contains
at most $|\Vground| + 3|\calt|\leq \lambda n/3$ unfit paths.
 So by relabelling $\calt_1$ and $\calt_2$ if necessary, we may assume that there are at least $\lambda n/3$ fit paths in $\calp_{\mathrm{undir}}$ which are entirely contained in trees in~$\calt_1$.
Furthermore, considering the fixed roots $r^T$ chosen for each $T \in \calt$, each such path $P$ has a well-defined pattern (as defined in Section~\ref{s:diamonds}, with respect to the root $r^T$ of the tree $T$ which contains $P$). Since there are at most $2^6$ possible patterns for the orientation of a rooted path $P \in \calp_{\mathrm{undir}}$,
we may choose a set $\calp \subseteq \calp_{\mathrm{undir}}$ with $|\calp| =
\lambda n/2^8$ such that every $P \in \calp$ is fit, lies in some $T \in \calt_1$, and has the same pattern $\hat P$. For each $P \in \calp$ let $v_1^P, \dots, v_7^P$ be the vertices of $P$, labelled with $v_1^P \prec v_2^P \prec \dots \prec v_7^P$.

\subsection{Reduced graph}
\label{s:reduced-graph}

We next construct a regular partition of~$G$,
which plays a crucial role in the allocation and embedding phases.
We apply Lemma~\ref{l:reduced-graph} to $G$ to obtain an integer $k$ with $K' \leq k \leq K$, a partition
$V_0\dcup V_1\dcup\cdots\dcup V_k$ of $V(G)$ and a digraph~\rstar\
with~$V(\rstar)=V_0\dcup[k]$
satisfying properties~\ref{rg-i}--\ref{rg-rstar-deg-ii} (in particular, $m$ is defined to be the common size of the clusters $V_1, \dots, V_k$, and satisfies $(1-\eps) n/k \leq m \leq n/k$).
Note that since $|V_0| \leq \eps n$ we have
\begin{equation}\label{e:semideg-G-V_0}
  \delta^0(G-V_0)
  \ge  (1/2+\alpha)n - |V_0|
  \ge  (1/2+\alpha/2)n.
\end{equation}
By Lemma~\ref{l:many-paths/diamonds} applied to $\rstar[[k]]$,
there is a connecting set $\cald$ of $\hat P$-diamonds
in $\rstar[[k]]$ with $|\cald| = k-1$ such that each $\ink$
lies in at most $4/\eta$ diamonds in $\cald$.
Let $H^\diamond=\bigcup_{\Diamond\in\cald} \Diamond$, and
let $H\subseteq \rstar[[k]]$ be the Hamilton cycle $1\rarr 2\rarr\cdots\rarr k\rarr 1$ (recall that Lemma~\ref{l:reduced-graph} guarantees that this is indeed a cycle in $\rstar[[k]]$).
By~Lemma~\ref{l:reg-expander},
$\rstar[[k]]$ contains a spanning $d_J$-regular expander~$J$,
with~$H \subseteq J$.

\subsection{Allocation}
\label{s:main-proof/allocation}

The next step in the proof is to allocate vertices to clusters
(vertices of $\Vground$ will in fact be embedded at this step).
To do this, we fix an embedding~$\varrho$ of the vertices in~$\Vground$ and let $\varphi_0$ be the associated allocation.
We then choose an allocation~$\varphi_{\mathrm{root}}$ for the roots of each $T\in\calt$,
and fix a set~$I(T)\subseteq [k]$ of candidate allocations
for each secondary attachment~$s^T$. Finally
the allocation of~$Q$ is completed using Claim~\ref{cl:allocate}
(proved in Section~\ref{s:main-proof/proof-cl:allocate}).

Since $Q$ is $2$-degenerate, the same is true of $Q[\Vground]$, so the small size of $\Vground$
allows us to greedily embed its vertices to~$G-V_0$
as follows.
Fix an ordering $r_1,r_2,\dots, r_t$ of the vertices in~$\Vground$
such that for each $i \in [t]$ the vertex $r_i$ has at most two neighbours in $\{r_1,\ldots,r_{i-1}\}$, and fix the image $\varrho(r_1)$ as some arbitrary vertex in~$V(G)\setminus V_0$.
By~\eqref{e:|A_0|} and~\eqref{e:semideg-G-V_0}, for all $\bullet,\diamond\in\{-,+\}$ and all $x,y\in V(G)$,
we have
\[
  \bigl|(N_G^\bullet(x)\cap N^\diamond(y))\setminus V_0\bigr|\ge \alpha n - \eps n > |\Vground|,
\]
and hence for each $1< i \le t$ there is an appropriate choice for the image $\varrho(r_i)$
among the vertices of $V(G)\setminus V_0$
which are not yet in the image of $\varrho$. For each $x \in \Vground$, having fixed $\varrho(x)$, set $\varphi_0(x)$ so that $\varrho(x) \in V_{\varphi_0(x)}$; this means that each vertex of $\Vground$ is allocated to the cluster to which it is embedded.

Let us now allocate the roots $r^T$\! of each $T \in \calt$.
For each $v\in \Vground$, and each $\bullet\in\{-,+\}$, by~\eqref{e:semideg-G-V_0} we may choose $i(v,\bullet) \in[k]$ such that
\begin{equation}\label{e:allocation-r^T}
  \deg_G^\bullet(\varrho(v), V_{i(v,\bullet)})
     \ge (1 + \alpha)m/2.
\end{equation}
For each $T\in\calt$,
set $\varphi_{\mathrm{root}}(r^T)=i(\hat r^T, \bullet)$ where $r^T\in N^\bullet(\hat r^T)$.
Moreover,
for each~$T\in\calt$ with a secondary attachment~$s^T\in N^\bullet(\hat s^T)$, recall that ${\hat s}^T$ is the unique neighbour of~$s^T$ in~$\Vground$, so $\varrho({\hat s}^T)$ has been defined. Let~$I(T)$ be the set of $j\in [k]$ such that
\begin{align}\label{e:allocation-s^T}
  \deg_G^\bullet(\varrho({\hat s}^T), V_j) \ge \eta m.
\end{align}
Since each cluster has at most~$m \leq n/k$ vertices, we then have
\begin{equation}\label{e:|I(T)|}
  |I(T)|\ge (1/2+\eta)k,
\end{equation}
since otherwise $\deg_{G-V_0}^\bullet(\varrho({\hat s}^T)) < \eta m \cdot k + m \cdot (1/2+\eta)k \leq (1/2 + 2 \eta)n$, contradicting~\eqref{e:semideg-G-V_0}.
We now allocate all remaining vertices of $Q$ through the following claim (proved in Section~\ref{s:main-proof/proof-cl:allocate}), for which we define

\[
  g \deq \left\lceil \frac{\lambda m}{2^{10}} \right\rceil
\]

\begin{claim}\label{cl:allocate}
  There exist disjoint~$\calp^0,\calp^H\subseteq\calp$
  and an extension of~$\varphi_0$
  to a map~$\varphi:V(Q)\to V(\rstar)$ such that
  \begin{enumerate}
  \item \label{l:tl/phi-deg}
  for each $T \in \calt$ the restriction $\varphi^T$ of $\varphi$ to $T$
    is a homomorphism from~$T$ to~$\rstar$
    with~$\Delta(\varphi^T)\leq 5$;

 \item \label{l:tl/a-V0-paths}
   $|\calp^0|=|V_0|$, and the restriction of $\varphi$ to either the centres $v_4^P$ of paths $P\in\calp^0$ (if these paths have order~7) or the non-root vertices $v_2^P$ of paths $P\in\calp^0$ (if these paths have order~$2$, in which case $v_2^P$ is a leaf of $Q$) is a bijection from that set to~$V_0$;

  \item \label{l:tl/a-V0-neigh}
    for each $\ink$ we have $|\varphi^{-1}(i)\cap N|\le 6\eps m/\alpha$, where $N \deq \bigcup_{x \in V(Q) : \varphi(x) \in V_0} N_Q^-(x)\cup N_Q^+(x)$;

  \item \label{l:tl/a-i}
    $\varphi$ maps precisely $m$~vertices to each $\ink$;

  \item \label{l:tl/a-ii}
    for each $P\in\calp^H$, the restriction of $\varphi$ to $P$
    is a homomorphism from $P$ to~$H$;

  \item \label{l:tl/a-iii}
    $\varphi$ maps
    precisely $g$ roots $v_1^P$ of paths in~$\calp^H$
    to each $i\in[k]$;

  \item\label{l:tl/a-st}
  For each $T \in \calt$ we have $\varphi(r^T) = \varphi_{\mathrm{root}}(r^T)$ and, if $T$ has a secondary attachment $s^T$\!, then $\varphi(s^T)\in I(T)$.
  \end{enumerate}
\end{claim}

\subsection{Embedding}
\label{s:embedding}
Fix $\calp^0\!$,\, $\calp^H$ and an allocation $\varphi$ extending $\varphi_0$ as obtained from Claim~\ref{cl:allocate}, and
let $M$ be the set of pendant vertices of paths in~$\calp^H$, so $M \deq \{v_2^P : P \in \calp^H\}$ if paths in $\calp$ have order 2, and $M \deq \{v_2^P, v_3^P, v_4^P, v_5^P, v_6^P: P \in \calp^H\}$ if the paths in $\calp$ have order 7.
Also say that $v\in V(Q)$ is a \defi{distinguished vertex} if $\varphi(v)\in V_0$.
For each distinguished vertex $v \in V(Q)$ set $\varrho(v) = \varphi(v)$.
So $\varrho$ now embeds all vertices in $\Vground$ and all distinguished vertices.
The bulk of the embedding is achieved by the following claim
(proved in Section~\ref{s:main-proof/proof-cl:embedding}).

\begin{claim}\label{cl:embed}
  There exists an extension of $\varrho$
  to an embedding of~$Q-M$ in $G$
  such that properties~\ref{i:respects-alloc}--\ref{i:sreg} below hold. For each \ink, write
  \begin{align*}
    U_i
    &  \deq \{\,\varrho(u)\in V_i:\text{$u = v_1^P$ for some $P\in\calp^H$}\,\},\\
    W_i
    &  \deq \{\,\varrho(w)\in V_i:\text{$w = v^7_P$ for some $P\in\calp^H$}\,\}\text{, and}\\
    V_i^\star
    & \deq V_i\setminus\varrho(Q-M).
  \end{align*}
  \begin{enumerate}
  \item\label{i:respects-alloc}
    The embedding $\varrho$ respects the allocation, meaning that
    $\varrho(x)=\varphi(x)$ if $\varphi(x)\in V_0$ and $\varrho(x)\in V_{\varphi(x)}$
    for all other $x\in V(Q)\setminus M$.
  \item \label{cl:embed_2path} If the paths in $\calp$ have order two then for each $i \in [k]$ we have $|U_i| = |V_i^*| = g$ and $W_i = \emptyset$.
  \item \label{cl:embed_7path} If the paths in $\calp$ have order seven then for each $i \in [k]$ we have $|U_i| = |W_i| = g$ and $|V_i^*| = 5g$.
  \item \label{i:sreg}
    For each \ink\
    the graphs $G[V_{\imm}^\star\rarr U_i]$,  $G[U_i \rarr V_{\ipp}^\star]$, $G[V_{\imm}^\star\rarr W_i]$,  $G[W_i\rarr V_{\ipp}^\star]$
    and $G[V_{\imm}^\star\rarr V_i^\star]$
    are each $(\beta\!,\eps')$-superregular.
  \end{enumerate}
\end{claim}

Fix such an (extended) embedding~$\varrho$ of $Q-M$ in $G$.
We complete~$\varrho$ to an embedding of $Q$ in $G$ by
defining the images of vertices in $M$, which are the pendant vertices of paths in~$\calp^H$,
as follows.

Suppose first that the paths in $\calp$ each have order~2, in which case they are pairwise-disjoint edges incident to leaves. So each path $P \in \calp^H$ has a root vertex~$v_1^P$ and another vertex $v_2^P$ which is a leaf vertex of~$Q$, and either $v_2^P$ is an outneighbour of $v_1^P$ for every $P \in \calp^H$ or $v_2^P$ is an inneighbour of $v_1^P$ for every $P \in \calp^H$.
We assume the former; the argument for the latter case is similar.
For each $i \in [k]$ we have $|U_i|=|V_\ipp^\star| = g$ by Claim~\ref{cl:embed}\ref{cl:embed_2path}; since $G[U_i\rarr V_\ipp^\star]$ is
$(\beta\!,\eps')$-superregular by Claim~\ref{cl:embed}\,\ref{i:sreg}, it follows that $G[U_i\rarr V_\ipp^\star]$ contains a perfect matching $M_i$ by Lemma~\ref{l:matching}. For each $P \in \calp^H$ with $v_1(P) \in U_i$ set $\varrho(v_2^P)$ to be the vertex in $V_\ipp^\star$ which $M_i$ matches to $\varrho(v_1^P)$; doing this for each $\ink$ gives the desired embedding $\varrho$ of $Q$ in $G$.

Now suppose instead that the paths in~$\calp$ each have order~$7$. For each \ink, let $Z_i^1=U_i$ and $Z_i^7= W_i$, and choose uniformly at random an equipartition of $V_i^\star$ into five sets $Z_i^j$ for $j \in [2,3,4,5,6]$; the choice for each $i \in [k]$ is independent of all others. So by Claim~\ref{cl:embed}\ref{cl:embed_7path} we have $|Z_i^j| = g$ for each $i \in [k]$ and $j \in [7]$. Moreover, for each $j \in \{2,3,4,5,6\}$ the set $Z_i^j$ is a uniformly-random subset of $V^\star_i$ of size $|V^\star_i|/5$, whilst tautologically $Z_i^1$ is a uniformly-random subset of $U_i$ of size $|U_i|$ and $Z_i^7$ is a uniformly-random subset of $W_i$ of size $|W_i|$. Using Claim~\ref{cl:embed}\ref{i:sreg} we may apply Lemma~\ref{l:sliceSRpair} and take a union bound over both events for each $i \in [k]$ and $\ell \in [6]$ to find that, with positive probability, for each \ink\ and each $\ell\in[6]$
the graphs
\begin{equation}\label{e:long-path-superreg}
  G[ Z_\imm^{\ell+1} \rarr Z_i^\ell]
  \tand
  G[ Z_i^\ell \rarr Z_\ipp^{\ell+1}]
  \text{ are each $(\beta,5\eps')$-superregular,}
\end{equation}
with addition on the indices taken modulo $k$. Fix an outcome for which each of these events occurs.

Recall that all paths in $\calp$ have the same pattern $\hat{P}$.
Let $u_1,\ldots,u_7$ be the vertices of~$\hat{P}$,
ordered as they appear in $\hat{P}$ with root $u_1$.
For each $\ell\in[7]$, let $\Sigma(\ell)$ denote the difference
between the number of forward and backward edges in the subpath of~$P$
from $u_1^P$ up to~$u_\ell^P$. In other words, we define
\[
  \Sigma(\ell) \deq
  \bigl|\{j\in[\ell -1] : u_j\rarr u_{j+1} \in E(\hat{P})\}\bigr|
  -  \bigl|\{j\in[\ell -1] : u_j\larr x_{j+1} \in E(\hat{P})\}\bigr|,
\]
and remark that $\Sigma(1)=0$.
With this definition,~\eqref{e:long-path-superreg} implies
that for each $\ell\in[6]$ and $\ink$ the sets $Z_{i+\Sigma(\ell)}^\ell$ and $Z_{i+\Sigma(\ell+1)}^{\ell+1}$
form a $(\beta\!,5\eps')$-superregular  pair in the direction of the edge between
$u_\ell$ and $u_{\ell+1}$.

Fix \ink\ and let $\calp_i^H \deq\{P\in\calp^H:\varrho(v_1^P)\in V_i\}$.
So each $P \in \calp_i^H$ has $\varrho(v_1^P)\in U_i = Z^1_i$ and also
$\varrho(v_7^P) \in W_{i+\Sigma(7)} = Z^7_{i+\Sigma(7)}$ by
Claim~\ref{cl:allocate}\ref{l:tl/a-ii} and Claim~\ref{cl:embed}\ref{i:respects-alloc},
while $v_2^P, \dots, v_6^P$ remain to be embedded. Let $\pi: Z_i^1 \to Z_{i+\Sigma(7)}^7$
be the bijection with $\pi(\varrho(v_1^P)) = \varrho(v_7^P)$ for each $P \in \calp_i^H$, and let
$L_i$ be the $7$-layer graph with vertex classes $Z_{i+\Sigma(j)}^j$ for $j \in [7]$. By applying
Lemma~\ref{l:four-layered} to $L_i$, we obtain a collection $\calp'_i$ of $g$ pairwise
vertex-disjoint paths of order 7 in $L_i$ such that for each $P \in \calp_i^H$ there is
a path $P' \in \calp'_i$ with ends $\varrho(v_1^P)$ and $\varrho(v_7^P)$. Let the
vertices of this path be $x_1^P, x_2^P, \dots, x_7^P$ in that order, so $x_1^P = \varrho(v_1^P)$, $x_7^P = \varrho(v_7^P)$ and $x_j^P \in Z_{i+\Sigma(j)}^j$ for each $j \in [7]$. Set $\varrho(v_j^P) = x_j^P$ for
each $P \in \calp_i^H$ and $j \in \{2,3,4,5,6\}$; doing this for each \ink~gives the desired embedding $\varrho$ of $Q$ in $G$.

This completes the proof of Theorem~\ref{t:treelike} except for the proofs
of Claims~\ref{cl:allocate} and~\ref{cl:embed}, which are contained in the next two sections.

\subsection{Proof of Claim~\ref{cl:allocate}}
\label{s:main-proof/proof-cl:allocate}
Let $F_1$ and $F_2$ be the forests formed by the trees in $\calt_1$
and~$\calt_2$ respectively. The allocation proceeds in two phases.
Roughly speaking, in the first phase
we allocate~$F_1$ using Algorithm~\ref{a:alloc},
enforcing that the root of each $T \in \calt$ is mapped in accordance with $\varphi_{\mathrm{root}}$ and
that each secondary attachment is mapped within a set $V_i$ with~$i\in I(T)$;
we also ensure that paths in $\calp$ are mapped along~$H$.
To conclude this phase, we modify the allocation of some
paths in~$\calp$, re-routing them so as to go through
every vertex in~$V_0$ and
through the connecting set~$\cald$ of diamonds in~$\rstar[[k]]$.
The resulting allocation~$\varphi_1$ will map precisely one vertex of $F_1$ to each vertex of $V_0$,
and will map the remaining vertices of $F_1$
approximately uniformly among $\ink$,
so that
$\bigl|\,\bigl|\varphi_1^{-1}(i)\bigr|-\bigl|\varphi_1^{-1}(j)\bigr|\,\bigr| \leq 10 \eps n/(\eta k)$
for all $i,j\in[k]$.

In the second phase, we build an allocation $\varphi_2$ of~$F_2$
using a biased allocation algorithm (Algorithm~\ref{a:alloc}
is not really modified, but we apply it to an auxiliary digraph,
which will produce the desired bias). This ensures that
the combination $\varphi_{\mathrm{join}}$ of $\varphi_0$, $\varphi_1$ and $\varphi_2$
will map~$V(Q)$ much more uniformly over~$[k]$.
In particular, for all $i,j\in[k]$ we shall have $\bigl|\,\bigl|\varphi_{\mathrm{join}}^{-1}(i)\bigr|-\bigl|\varphi_{\mathrm{join}}^{-1}(j)\bigr|\,\bigr| < 4 n \log \log n / \log n$.
Finally we complete the proof by modifying the allocation of paths routed through
diamonds to obtain a perfectly uniform allocation $\varphi$ with the desired properties.

For each path $P\in\calp$, given an allocation~$\pi$ of the root~$v_1^P$ of $P$,
a \defi{canonical} allocation of~$P$ is a homomorphism $\pi$ from $P$
into $H$ which extends~$\pi$.

\smallskip

\noindent\textbf{The first phase.}
Define permutations $\pi^+\colon[k]\to[k]$ and $\pi^-\colon[k]\to[k]$ such that $\pi^+(i) = i+1$ and $\pi^-(i) = i-1$ for each $\ink$, with addition and subtraction taken modulo $k$. So $\pi^+$ maps each $\ink$ to its outneighbour in $H$, and likewise $\pi^-$ maps each $\ink$ to its inneighbour in $H$.
Let $\cale \deq \bigcup_{\,P\in\calp\,} E(P)$, and for each edge $e \in \cale$ joining
$v_\ell^P$ and $v_{\ell +1}^P$,
define $\pi_e$ by
\[
  \pi_e \deq
  \begin{cases}
    \pi^- & \text{if $v_{\ell+1}^P \in N_{F_1}^-(v_\ell^P)$,}\\
    \pi^+ & \text{otherwise}.
  \end{cases}
\]
We apply Algorithm~\ref{a:alloc} to $F_1$ and $J$ to obtain a map $\psi: V(F_1) \to [k]$.
We do with~$\pi_e$ as defined above for each edge in $\cale$,
and with $Z=\bigcup_{\,T \in \calt_1\,}\{r^T,s^T\}$.
Our choice of allocation for vertices of~$Z$ is as follows: for each $T \in \calt_1$ we set $\psi(r^T) = \varphi_{\mathrm{root}}(r^T)$, whilst for each $T \in \calt_1$ with a secondary attachment $s^T$ we set $\psi(s^T)$ to be some $j\in[k]$ which lies in both $I(T)$ and the
appropriate neighbourhood of the image of $p^T$
(recall that $p^T$ is the parent of $s^T$, so $\psi(p^T)$ will already have been defined when the algorithm comes to choose~$\psi(s^T)$).
More precisely, fix $\bullet\in\{-,+\}$ so that $s^T\in N_{T}^\bullet(p^T)$,
and choose $\psi(s^T)\in N_{R^\star[[k]]}^\bullet\bigl(\psi(p^T)\bigr)\cap I(T)$
(this is possible since $|I(T)| > k/2$ by~\eqref{e:|I(T)|} and $\delta^0\bigl(\rstar[[k]]\bigr) > k/2$ by Lemma~\ref{l:reduced-graph}\ref{rg-rstar-deg-i}).
Observe that, with these definitions, Algorithm~\ref{a:alloc} ensures that for each edge $u \to v$ of $F_1$ we have that $\psi(u) \to \psi(v)$ is an edge of $J$, \emph{except} possibly when $\{u,v\} = \{p^T, s^T\}$, in which case the fact that $\psi(s^T)$ was chosen in~$N_{R^\star[[k]]}^\bullet\bigl(\psi(p^T)\bigr)$ ensures that $\psi(u) \to \psi(v)$ is an edge of~$R^\star[[k]]$. Since $J \subseteq R^\star[[k]]$, it follows that $\psi$ is an allocation of~$F_1$ to~$R^\star[[k]]$ (that is, a homomorphism from $F_1$ to $R^\star[[k]]$). Moreover, $\psi$ has the following crucial properties. First, our choices of $\cale$ and $\pi_e$ ensure that the allocation of each path in $\calp$ by $\psi$ is canonical. Second, for each $T \in \calt_1$ we have $\psi(r^T) = \varphi_{\mathrm{root}}(r^T)$ and, if $T$ has a secondary attachment, then our choice of $\psi(s^T)$ ensures that $\psi(s^T) \in I(T)$. Third, by Lemma~\ref{l:vtx-distribution} (with $S = V(F_1)$), with high probability we have for each $i \in [k]$ that
\begin{equation}\label{e:f1-b4-rerouting}
    |\psi^{-1}(i)| = \frac{v(F_1)}{k}  \pm \left(\frac{n}{\log n} + (|Z|+3n^{1/3})\Delta(F_1)^{56k^5\log \log n}\right) = \frac{v(F_1)}{k}  \pm \frac{2n}{\log n},
\end{equation}
using that $|Z| \le 2n^{0.999}$ by~\eqref{e:|calt|} and that $3n^{0.999}\Delta(F_1)^{56k^5\log \log n} \le n/\log n$.
Fourth, set $\Lambda \deq\{v_1^P:P\in\calp\}$, so $\Lambda$ is the set of roots of paths in $\calp$, and for each $\ink$ let $\calp_i \subseteq \calp$ consist of all paths $P \in \calp$ with $\psi(v_1^P) = i$, that is, whose root is allocated to $V_i$. Then $|\Lambda|= |\calp| = \lambda n/2^8$, and a similar application of Lemma~\ref{l:vtx-distribution} (with $S=\Lambda$) shows that with high probability we have for each $i \in [k]$ that
\begin{equation}\label{e:l1-b4-rerouting}
 |\calp_i| = |\psi^{-1}(i)\cap \Lambda|
  = \frac{\lambda n}{2^8k}  \pm  \frac{2n}{\log n} \geq \frac{\lambda m}{2^{9}}
\end{equation}
Fix an outcome of Algorithm~\ref{a:alloc} for which $\psi$ has each of these properties. We now modify $\psi$ to obtain the allocation $\varphi_1$ which is our goal in this phase; the procedure for this varies according to whether the paths in~$\calp$ have order~$2$ or~$7$.

Let us handle the former case first, that is, when each path in $\calp$ is an edge incident to a leaf.
Let $\bullet \in\{-,+\}$ be the sign such that $v_1^P\in N^\bullet(v_2^P)$ for all~$P\in\calp$.
By Lemma~\ref{l:reduced-graph}\,\ref{rg-rstar-deg-ii},
we have $\bigl|N_{R^\star}^\bullet(x)\cap[k]\bigr|\ge \alpha k/2$ for each $x\in V_0$,
so by Lemma~\ref{l:greedy-cover} there exists a function $\mathfrak{g}:V_0\to [k]$
such that $\mathfrak{g}(x) \in N_{R^\star}^\bullet(x)$ for each $x \in V_0$ and $\bigl|\mathfrak{g}^{-1}(i)\bigr|\le 1 + 2|V_0|/(\alpha k)\le 3\eps m/\alpha$
for each $\ink$.
Next, for each $\Diamond \in \cald$ let $\cald_i$ be the set of diamonds in $\Diamond$ with prefix $i$, so for each $i \in [k]$
we have $|\cald_i| \leq 4/\eta$ by our choice of $\cald$.
For each $i \in [k]$ choose pairwise disjoint subsets $\calp_i^0,\, \calp_i^\diamond,\, \calp_i^H \subseteq \calp_i$
with sizes $|\calp_i^0| = \bigl|\mathfrak{g}^{-1}(i)\bigr|$,\,
$|\calp_i^\diamond| = 2n|\cald_i|/k^2$ and $|\calp_i^H| = g$.
This is possible by~\eqref{e:l1-b4-rerouting} since
\[
  |\calp_i^0| + |\calp_i^\diamond| + |\calp_i^H|  = \bigl|\mathfrak{g}^{-1}(i)\bigr| + \frac{2n|\cald_i|}{k^2}+ g \leq \frac{3\eps m}{\alpha} + \frac{2n}{k^2} \cdot \frac{4}{\eta}+ \left\lceil \frac{\lambda m}{2^{10}} \right\rceil \leq \frac{\lambda m}{2^9} \leq |\calp_i|.
\]
Let $\calp^0 := \bigcup_\ink \calp^0_i$,\, $\calp^\diamond := \bigcup_\ink \calp^\diamond_i$ and
$\calp^H := \bigcup_\ink \calp^H_i$, so in particular
$|\calp^0| = \sum_{i \in [k]} \bigl|\mathfrak{g}^{-1}(i)\bigr| =  |V_0|$.

Our choice of $\calp^0$ allows us to choose, for each $P \in \calp^0$, an image $\varphi_1(v_2^P) \in \mathfrak{g}^{-1}\bigl(\psi(v_1^P)\bigr)$
so that the chosen images $\varphi_1(v_2^P)$ for each $P \in \calp^0$ are all distinct. We also set $\varphi_1(v_1^P) \deq \psi(v_1^P)$ for each $P \in \calp^0$. So $\varphi_1$ is a homomorphism from the paths in $\calp^0$ to $R^\star$ whose restriction to $\{v_2^P : P \in \calp^0\}$ is a bijection from that set to~$V_0$, which will ensure~\ref{l:tl/a-V0-paths}. Moreover, the neighbours in $Q$ of vertices mapped to $V_0$ are the vertices $v_1^P$ for $P \in \calp$, of which at most $\bigl|\mathfrak{g}^{-1}(i)\bigr| \le 3\eps m/\alpha$ are mapped to each $i \in [k]$; this will ensure~\ref{l:tl/a-V0-neigh}.

Similarly, our choice of $\calp^\Diamond$ allows us to choose a map $\mathfrak{h} : \calp^\diamond \to \cald$ so that for each $\ink$ each $P \in \calp^\diamond_i$ has $\mathfrak{h}(P) \in \cald_i$ and so that each $\Diamond \in \cald$ has $|\mathfrak{h}^{-1}(\Diamond)| = 2n/k^2$. This means that for each $\Diamond \in \cald$ we may do the following. Let $u$ be the prefix of $\Diamond$ (so $\Diamond \in \cald_u$), and let $\{v, v'\}$ be the middle of $\Diamond$. Choose $n/k^2$ paths $P \in \mathfrak{h}^{-1}(\Diamond)$ and for each set $\varphi_1(v_2^P) := v$; for each of the remaining $n/k^2$ paths $P \in \mathfrak{h}^{-1}(\Diamond)$ set $\varphi_1(v_2^P) := v'$. Also set $\varphi_1(v_1^P) := u = \psi(v_1^P)$ for every $P \in \mathfrak{h}^{-1}(\Diamond)$. So $\varphi_1$ gives a homomorphism from the paths in $\calp^\diamond$ to $R^\star[[k]]$ which maps $n/k^2$ paths to each branch of each diamond $\Diamond \in \cald$.

Finally, for each vertex $u \in V(F_1)$ for which $\varphi_1(u)$ has not yet been defined, set $\varphi_1(u) \deq \psi(u)$ (so $u$~remains allocated as in the outcome of Algorithm~\ref{a:alloc}). Since we previously set $\varphi_1(v_1^P) = \psi(v_1^P)$ for each path $P \in \calp^0 \cup \calp^\diamond$, we then have that $\varphi_1$ is a homomorphism from $F_1$ to $R^\star$. Moreover, the only vertices $u \in V(F_1)$ which may have $\psi(u) \neq \varphi_1(u)$ are vertices $v_2^P$ for paths $P \in \calp^0 \cup \calp^\diamond$.
Consequently, for each $\ink$ the number of vertices $u \in V(F_1)$ with $\psi(u) = i$ and $\varphi_1(u) \neq i$ is at most $|\calp^0_{i^*}| + |\calp^\diamond_{i^*}|$, where $i^* =i-1$ if $v_1^P \rarr v_2^P$ and $i^* = i+1$ if $v_1^P \larr v_2^P$.
Similarly, the number of vertices $v \in V(F_1)$ with $\varphi_1(v) = i$ and $\psi(v) \neq i$ is at most $(4/\eta)( n/k^2)$ since $v$ is in at most $4/\eta$ diamonds in $\cald$. We conclude that for each $\ink$ we have
\begin{equation}\label{e:2pathbound}
    \bigl|\,|\varphi_1^{-1}(i)| - |\psi^{-1}(i)|\,\bigr| \leq \frac{3 \eps m}{\alpha} + \frac{2n}{k^2}\cdot\frac{4}{\eta} + \frac{4}{\eta}\cdot \frac{n}{k^2} \leq \frac{4 \eps m}{\alpha}.
\end{equation}

Now consider instead the case where the paths in $\calp$ are bare paths of order~$7$. In this case, for each $i \in [k]$ we choose pairwise disjoint subsets $\calp_i^0,\, \calp_i^\diamond,\, \calp_i^H \subseteq \calp_i$ with $|\calp_i^0| = |V_0|/k \pm 1$, $|\calp_i^\diamond| = 2n|\cald|/k^3 \pm 1$ and $|\calp_i^H| = g$; the precise values of $|\calp_i^0|$ and $|\calp_i^\diamond|$ for $i \in [k]$ are chosen so that $\sum_{\ink}|\calp^0_i| = |V_0|$ and $\sum_{\ink}|\calp^\diamond_i| = 2n|\cald|/k^2$. By~\eqref{e:l1-b4-rerouting} it is possible to make these choices since
\[
  |\calp_i^0| + |\calp_i^\diamond| + |\calp_i^H|  \leq \frac{|V_0|}{k} + 1 + \frac{2n|\cald|}{k^3} + 1 + g \leq 2 + \frac{\eps n}{k} + \frac{2n}{k^2} + \left\lceil\frac{\lambda m}{2^{10}}\right\rceil \leq \frac{\lambda m}{2^9} \leq |\calp_i|.
\]
As before we let $\calp^0 := \bigcup_\ink \calp^0_i$, $\calp^\diamond := \bigcup_\ink \calp^\diamond_i$ and
$\calp^H := \bigcup_\ink \calp^H_i$, so in particular
$|\calp^0| = |V_0|$ and $|\calp^\diamond_i| = 2n|\cald|/k^2$.

Fix a bijection $p: \calp^0 \to V_0$. For each $P \in\calp^0$ set $\varphi_1(v_4^P) = p(P)$. Also, for each $P \in\calp^0$ set $\varphi_1(v_1^P) = \psi(v_1^P)$ and $\varphi_1(v_7^P) = \psi(v_7^P)$. Next, let $\bullet,\circ\in\{-,+\}$ be such that $v_3^P\in N_P^\bullet(v_4^P)$
and $v_5^P\in N_P^\circ(v_4^P)$, and choose $\varphi_1(v_3^P)\in N_\rstar^\bullet(\varphi_1(v_4^P))$
and $\varphi_1(v_5^P)\in N_\rstar^\circ(\varphi_1(v_4^P))$. Since $\rstar$ satisfies property \ref{rg-rstar-deg-ii}
of~Lemma~\ref{l:reduced-graph}, for each $P \in \calp^0$ there are at least $\alpha k/2$
options for the choices of $\varphi_1(v_3^P)$ and of $\varphi_1(v_5^P)$.
So by Lemma~\ref{l:greedy-cover} we may make these choices
so that for each $\ink$ at most $1 + |V_0|/(\alpha k/2) \le 3\eps m /\alpha $ paths $P \in \calp^0$ have $\varphi_1(v_3^P) = i$ and at most $3\eps m/ \alpha$ paths $P \in \calp^0$ have $\varphi_1(v_5^P) = i$.
Next, for each $P \in \calp^0$, choose $\varphi_1(v_2^P)$
to be an appropriate common neighbour of $\varphi_1(v_1^P)$ and $\varphi_1(v_3^P)$ in $R^\star[[k]]$ (here `appropriate' means respecting the direction of the edges between $v_1^P$ and $v_2^P$ and between $v_2^P$ and $v_3^P$),
and likewise choose $\varphi_1(v_6^P)$
to be an appropriate common neighbour of $\varphi_1(v_5^P)$ and $\varphi_1(v_7^P)$ in $R^\star[[k]]$. Since $\delta^0(R^\star[[k]]) \ge (1/2+\eta)k$, there are at least $2 \eta k$ options for each image, so by Lemma~\ref{l:greedy-cover} we may make these choices in such a way that for each $\ink$ at most $1+ \eps n/(2 \eta k) \leq \eps m/\eta$ paths $P \in \calp^0$ have $\varphi_1(v_2^P) = i$ and at most $\eps m/\eta$ paths $P \in \calp^0$ have $\varphi_1(v_6^P) = i$.
Then $\varphi_1$ is a homomorphism from the paths in $\calp^0$ to $R^\star$ whose restriction to $\{v_4^P : P \in \calp^0\}$ is a bijection from that set to $V_0$, which will ensure~\ref{l:tl/a-V0-paths}. Moreover, the neighbours in $Q$ of vertices mapped to $V_0$ are the vertices $v_3^P$ and $v_5^P$ for $P \in \calp$, of which at most $6\eps m/\alpha$ are mapped to each $i \in [k]$; this will ensure~\ref{l:tl/a-V0-neigh}.

Next
choose a map $f:\calp^\diamond\mapsto \cald$ such that for each $\Diamond\in\cald$
we have $|f^{-1}(\Diamond)| = 2 n/k^2$ (this is possible since $|\calp^\diamond| = 2n|\cald|/k^2$).
For each $\Diamond \in \cald$ we do the following. Let $u$, $\{v,v'\}$ and $w$ be the prefix, middle and suffix of $\Diamond$. Choose $n/k^2$ paths $P \in f^{-1}(\Diamond)$ and for each set $\varphi_1(v_3^P) = u, \varphi_1(v_4^P) = v,\varphi_1(v_5^P) = w$; for each of the remaining $n/k^2$ paths $P \in  f^{-1}(\Diamond)$ set $\varphi_1(v_3^P) = u$, $\varphi_1(v_4^P) = v'$, $\varphi_1(v_5^P) = w$.
In both cases also set $\varphi_1(v_1^P) = \psi(v_1^P)$ and $\varphi_1(v_7^P) = \psi(v_7^P)$.
Then, choose $\varphi_1(v_2^P)$ and $\varphi_1(v_6^P)$ exactly as we did for paths $P \in \calp^0$; as before there are at least $2 \eta k$ options for each image, so we may make these choices so that for each $\ink$ at most $(2 n|\cald|/k^2)/(2 \eta k) + 1\leq 2 n/(\eta k^2)$
paths $P \in \calp^\diamond$ have $\varphi_1(v_2^P) = i$ and at most $2 n/(\eta k^2)$ paths $P \in \calp^\diamond$ have $\varphi_1(v_6^P) = i$. Then $\varphi_1$ gives a homomorphism from the paths in $\calp^\diamond$ to $R^\star[[k]]$ which maps $n/k^2$ paths to each branch of each diamond $\Diamond \in \cald$.

As in the previous case our final step is to set $\varphi_1(u) \deq \psi(u)$ for each vertex $u \in V(F_1)$ for which $\varphi_1(u)$ has not yet been defined (so $u$ remains allocated as in the outcome of Algorithm~\ref{a:alloc}). Since we previously set $\varphi_1(v_1^P) = \psi(v_1^P)$ and $\varphi_1(v_7^P) = \psi(v_7^P)$ for each $P \in \calp^0 \cup \calp^\diamond$, we then have that $\varphi_1$ is a homomorphism from $F_1$ to $R^\star$. Moreover, the only vertices $u \in V(F_1)$ which may have $\psi(u) \neq \varphi_1(u)$ are the pendant vertices $v_2^P, \dots, v_6^P$ of paths $P \in \calp^0 \cup \calp^\diamond$.
Consequently, for each $\ink$ the number of vertices $u \in V(F_1)$ with $\psi(u) = i$ and $\varphi_1(u) \neq i$ is at most $5 \max_{i \in [k]} |\calp^0_i| + 5 \max_{i \in [k]} |\calp^\diamond_i| \leq 5(\eps n/k + 1) + 10n|\cald|/k^3 \leq 6 \eps m$.
Similarly, the number of vertices $v \in V(F_1)$ with $\varphi_1(v) = i$ and $\psi(v) \neq i$ is at most $6 \eps m/\alpha + 2 \eps m/\eta +  (2n/k^2)\cdot (4/\eta) + 4n/(\eta k^2) \leq 3 \eps m/\eta$ since $v$ is in at most $4/\eta$ diamonds in $\cald$. We conclude that for each $\ink$ we have
\begin{equation}\label{e:7pathbound}
    \bigl|\,|\varphi_1^{-1}(i)| - |\psi^{-1}(i)|\,\bigr| \leq 6 \eps m + \frac{3\eps m}{\eta} \leq \frac{4 \eps m}{\eta}.
\end{equation}

In both cases we have obtained a homomorphism $\varphi_1$ from $F_1$ to $R^\star$ with desired properties for the images of vertices of paths in $\calp^0$ and $\calp^\diamond$. Moreover, combining~\eqref{e:f1-b4-rerouting} with~\eqref{e:2pathbound} or~\eqref{e:7pathbound} (according to the case), in both cases for each $\ink$ we have
\begin{align}
  \bigl|\varphi_1^{-1}(i)\bigr|
  & =
    \frac{v(F_1)}{k} \pm   \left(\frac{2n}{\log n} + \frac{4\eps m}{\eta}
    \right)
  =
    \frac{v(F_1)}{k}\pm \frac{5\eps m}{\eta}.
  \label{e:varphi-phase-1}
\end{align}

 \smallskip
\noindent\textbf{The second phase.}
In our second allocation phase,
 we use an auxiliary graph which is a weighted blow-up of~$J$
 to build an allocation~$\varphi_2$ of~$F_2$ in $J$.
 This will compensate for the non-uniform usage of clusters,
 drastically reducing the differences between the numbers
 of vertices mapped to each cluster.
Let $n_2 \deq v(F_2)$, and for each $\ink$ set
  \begin{align}
    \alpha_i
    &\deq
      \frac{1}{n_2}\left( \frac{n - |V_0|}{k} - |\varphi_0^{-1}(i)| - |\varphi_1^{-1}(i)|\right)
      \qquad\text{and}\qquad
      b_i\deq \alpha_i \log\log\log n_2.
      \label{e:treelike/alpha-i-def}
  \end{align}
(We remark that the $\log \log \log n_2$ term in the definition of $b_i$ is chosen simply to grow very slowly as a function of $n_2$; any other sufficiently slowly growing function of~$n_2$ would work equally well.)
Since $\sum_\ink |\varphi_0^{-1}(i)| = |\Vground|$, $\sum_\ink |\varphi_1^{-1}(i)| = v(F_1) - |V_0|$ and $n = v(F_1)+v(F_2) + |\Vground|$, we have
  \begin{equation}\label{e:sum-alpha_i}
    \sum_{\ink}\alpha_i=1.
  \end{equation}
  By~\eqref{e:|A_0|},~\eqref{e:varphi-phase-1},~\eqref{e:treelike/alpha-i-def} and the facts that $n_2 \geq n/3$ and $|\varphi_0^{-1}(i)| \leq |\Vground|$, for each $\ink$ we have
  \begin{align} \label{e:alpha_i-lbound}
    \alpha_i
   &\geq
    \frac{1}{n_2} \left(\frac{n}{k} - \frac{|V_0|}{k} - |\Vground| - \frac{v(F_1)}{k} - \frac{5 \eps m}{\eta}\right)  \\&=
    \frac{1}{n_2} \left(\frac{n_2}{k} - \frac{|V_0|}{k} - \frac{k-1}{k} |\Vground| - \frac{5 \eps m}{\eta}\right)
    \geq \frac{1}{k} - \frac{6 \eps n}{\eta  n_2 k} \geq  \frac{1}{k} \left( 1 - \frac{20\eps}{\eta}\right).\nonumber
  \end{align}
  Let $B$ be a (blow-up) graph of
  $\rstar[[k]]$, obtained by replacing each $i\in[k]$  by a
  set $B_i$ with precisely $b_i$~vertices,
  with~$x\farc y\in E(B)$ if and only if $x\in B_i$, $y\in B_j$
  and $i\farc j\in E(\rstar)$.
  Note that $v(B)=\log\log\log n_2$ by~\eqref{e:treelike/alpha-i-def} and~\eqref{e:sum-alpha_i}.
  Also $B$~contains a spanning $d'_J$-regular expander subdigraph~$\Jblow$
  by Lemma~\ref{l:reg-expander}, since
  \begin{align*}
    \delta^0(B)
    \geq \delta_0\bigl(\rstar[[k]]\bigr)\cdot \min_{\ink} b_i
    & \overset{
      \phantom{\eqref{e:alpha_i-lbound}}}{=}
      \left(\frac{1}{2} + \eta\right) k \min_{\ink}\,\alpha_iv(B)\\
    &   \overset{
      \eqref{e:alpha_i-lbound}}{\geq} \left(\frac{1}{2} + \eta\right) k
      \cdot v(B)\cdot
      \frac{1}{k}\left(1-\frac{20 \eps}{\eta}\right)\nonumber\\
       &
 \overset{
      \phantom{\eqref{e:alpha_i-lbound}}}{\geq}
      \left(\frac{1}{2}+\frac{\eta}{2}\right)v(B).
  \end{align*}

We apply Algorithm~\ref{a:alloc}
to~$F_2$ and $B$ to obtain a map $\psi_B : V(F_2) \to V(B)$. This application is simpler
than the application of  Algorithm~\ref{a:alloc} in Phase 1. Specifically,
we apply Algorithm~\ref{a:alloc}
with $\cale=\emptyset$ and with $Z=\bigcup_{\,T \in \calt_2\,}\{r^T,s^T\}$.
Our choice of allocation for vertices of~$Z$ is essentially the same as in Phase 1: for each $T \in \calt_2$ we
set $\psi_B(r^T)$ to be an arbitrary vertex in $B_{\varphi_{\mathrm{root}}(r^T)}$,
and for each $T \in \calt_2$ with a secondary attachment we choose $\psi_B(s^T)$ in both $\bigcup_{i \in I(T)} B_i$ and the appropriate neighbourhood
of the image of $p^T$ (recall that~$p^T$ is the parent of $s^T$ in $T$,
so $\psi_B(p^T)$ has already been defined when Algorithm~\ref{a:alloc} considers $s^T$).
In particular, the latter condition ensures that the
map $\psi_B$ is a homomorphism from $F_2$ to $B$.
Moreover, applying Lemma~\ref{l:vtx-distribution} with $S = V(F_2)$ we find that with high probability we have for each $x \in V(B)$ that
  \begin{align}\label{e:2nd-alloc-distrib}
    |\psi_B^{-1}(x)| = \frac{n_2}{v(B)}   \pm \left(\frac{n_2}{\log n_2}
    + \bigl(|Z|+3n_2^{1/3}\bigr)\Delta(F_2)^{56v(B)^5\log \log n_2}\right)
    =   \frac{n_2}{v(B)}\pm \frac{2n}{\log n},
  \end{align}
  where we use the bounds
  $|Z|\le 2n^{0.999}$ and~$3n^{0.999}\Delta(F_2)^{56v(B)^5\log \log n_2} \le n/\log n$.

  We recover the desired allocation~$\varphi_2$ of $F_2$
  to~$\rstar[[k]]$ by ``collapsing'' the allocation of each blown-up vertex.
  More precisely, we define~$\varphi_2: V(F_2) \to [k]$ by putting~$\varphi_2(v)=i$
  for each $v\in V(F_2)$ whenever $\psi_B(v)\in B_i$. So~$\varphi_2$ is a homomorphism from $F_2$ to $\rstar[[k]]$ with $\varphi_2(r^T)=\varphi_{\mathrm{root}}(r^T)$ for each $T \in \calt_2$
  and $\varphi_2(s^T)\in I(T)$ for each $T \in \calt_2$ with a secondary attachment. From~\eqref{e:2nd-alloc-distrib} it follows that for each $i \in [k]$ we have
  \begin{align}\label{e:2nd-alloc-distrib_2}
    |\varphi_2^{-1}(i)| =  \sum_{x \in B_i} |\psi_B^{-1}(x)| = \frac{b_i n_2}{v(B)}   \pm \frac{2nb_i}{\log n} = \alpha_i n_2 \pm \frac{2n \log \log n}{\log n}.
  \end{align}

 \smallskip
\noindent\textbf{The conclusion.}
Let $\varphi_{\mathrm{join}}: V(Q) \to R^\star$ be the map formed by combining $\varphi_0$, $\varphi_1$ and $\varphi_2$. More precisely, for each $x \in V(Q)$ we set $\varphi_{\mathrm{join}}(x) = \varphi_0(x)$ if $x \in \Vground$, $\varphi_{\mathrm{join}}(x) = \varphi_1(x)$ if $x \in V(F_1)$ and $\varphi_{\mathrm{join}}(x) = \varphi_2(x)$ if $x \in V(F_2)$.
For each~\ink\ we then have
  \newcommand{\oversetphantom}{\overset{\phantom{\eqref{e:2nd-alloc-distrib}}}{=}}
  \begin{align}
    |\varphi_{\mathrm{join}}^{-1}(i)|
        & \oversetphantom
      |\varphi_0^{-1}(i)| + |\varphi_1^{-1}(i)| + |\varphi_2^{-1}(i)|
      \nonumber\\
    & \overset{\eqref{e:2nd-alloc-distrib_2}}{=}
      |\varphi_0^{-1}(i)| + |\varphi_1^{-1}(i)|
            + \alpha_in_2
      \pm \frac{2n\log\log n}{\log n}
      \nonumber\\
    & \overset{\eqref{e:treelike/alpha-i-def}}{=}
      |\varphi_0^{-1}(i)| + |\varphi_1^{-1}(i)|
            + \frac{n - |V_0|}{k} - |\varphi_0^{-1}(i)| - |\varphi_1^{-1}(i)|
      \pm \frac{2n\log\log n}{\log n}
      \nonumber\\
    & \oversetphantom
      m \pm \frac{2n\log\log n}{\log n}.
      \nonumber
  \end{align}
  Finally, for each $\ink$ set $\delta_i \deq m - |\varphi_{\mathrm{join}}^{-1}(i)|$, so $\delta_i < n/k^3$ and $\sum_{\ink} \delta_i = 0$. Recall that our choice of~$\varphi_1$ ensured that for each $\Diamond \in \cald$ at least $n/k^2$ paths $P \in \calp^\diamond$ were mapped to each branch of $\cald$. So we may apply Lemma~\ref{l:shifting}. Writing $\calp^{\diamond}$ also for the oriented graph which is the disjoint union of the paths in $\calp^{\diamond}$\!, this yields a homomorphism $\varphi$ from $\calp^{\diamond}$ to $R^\star[[k]]$ with $|\varphi^{-1}(i)| = |\varphi_{\mathrm{join}}^{-1}(i)| + \delta_i$ for each $\ink$ and such that $\varphi(x) = \varphi_{\mathrm{join}}(x)$ for all vertices of each $P \in \calp^\diamond$ except for the central vertex $v_4^P$ if $P$ has order~7 or the leaf vertex $v_2^P$ if $P$ has order 2. Setting also $\varphi(x) = \varphi_{\mathrm{join}}(x)$ for every other vertex $x$ of $Q$ we obtain a map $\varphi(x) : V(Q) \to V(R^\star)$ extending $\varphi_0$ such that, for every $\ink$, we have
  \[
    |\varphi^{-1}(i)| = |\varphi_{\mathrm{join}}^{-1}(i)| + \delta_i = m,
  \]
so $\varphi$ satisfies~\ref{l:tl/a-i}.

To conclude the proof of the claim, let us argue that~$\varphi$ satisfies the remaining stated properties.
Recall that $\varphi_1$ is a homomorphism from $F_1$ to $\rstar$ and $\varphi_2$ is a homomorphism from $F_2$ to $\rstar$; it follows that for each $T \in \calt$ the restriction of $\varphi_{\mathrm{join}}$ to $T$ is a homomorphism from $T$ to $R^\star$, and so the same is true of $\varphi$.
Next observe that for each vertex $u \in V(F)$, our applications of Algorithm~\ref{a:alloc} allocated all in-children of $u$ to the same vertex of $R^\star$ and all out-children of $u$ to the same vertex of $R^\star$, except in the cases where the edge between $u$ and the child was an edge of $\cale$ or if the child was a secondary attachment; each of these exceptions can occur for at most one child of $u$. Counting also the parent of $u$, we conclude that $\Delta(\psi), \Delta(\psi_B) \leq 5$; it follows that $\Delta(\varphi) \leq 5$, so we have~\ref{l:tl/phi-deg}.

Our choice of $\calp^H$ immediately gives~\ref{l:tl/a-iii}. Moreover, each path $P \in \calp^H$ was allocated canonically by $\psi$; since for each $P \in \calp^H$ and $v \in V(P)$ we have $\varphi(v) = \varphi_{\mathrm{join}}(v) = \varphi_1(v) = \psi(v)$, it follows that we have~\ref{l:tl/a-ii}. Now recall that our choice of the sets $\calp^0_i$ and the map $\varphi_1$ ensured that~\ref{l:tl/a-V0-paths}, and~\ref{l:tl/a-V0-neigh} held with $\varphi_1$ in place of~$\varphi$. Since we have $\varphi(v) = \varphi_{\mathrm{join}}(v) = \varphi_1(v)$ for all vertices $v$ in paths in $\calp^0$, we have~\ref{l:tl/a-V0-paths} and~\ref{l:tl/a-V0-neigh}.  Finally, for~\ref{l:tl/a-st} recall for each $i \in \{1, 2\}$ that the choice of $\varphi_i$ ensured that for every $T \in \calt_i$ we had $\varphi_i(r^T) = \varphi_{\mathrm{root}}(r^T)$, so $\varphi(r^T) = \varphi_{\mathrm{root}}(r^T)$ also. Moreover, if $T$ had a secondary attachment $s^T$, then we had $\varphi_i(s^T) \in I(T)$, and so $\varphi(s^T) \in I(T)$ also.

\subsection{Proof of Claim~\ref{cl:embed}}
\label{s:main-proof/proof-cl:embedding}
In this section we
describe and analyse a greedy algorithm
that extends $\varrho$ to an embedding of $Q-M$ into~$G$.
This algorithm embeds each vertex of $F$ in turn, in the order of $\prec$.
Call a vertex $v \in V(F)$ \defi{special} if
$v$ is an attachment of some $T \in \calt$, if $v$ is
a neighbour of a distinguished vertex,
or if $v \in \{v_1^P, v_7^P\}$ for some~$P\in\calp^H$.
The algorithm treats these vertices of $G$ specially, embedding them into
fixed sets which are chosen beforehand for that purpose; another pre-chosen set will
never be embedded to by the algorithm, and this will provide the necessary superregularity properties.
Specifically, we reserve pairwise disjoint subsets
$\treeAttach,B,\pathRoot,\pathLast$ and~$\pathOther$
of~$V(G)\setminus V_0$ obtained from the following claim.

\begin{claim}\label{cl:embed/reserve}
  There exist pairwise disjoint subsets $\treeAttach,\, B,\, \pathRoot,\,\pathLast, \pathOther \subseteq V(G)\setminus \bigl(V_0\cup \varrho(\Vground)\bigr)$
  such that for all $\ink$ and all $T\in\calt$
  the following statements hold.
  \begin{enumerate}
  \item\label{i:reserved-set-size}
    $ |\treeAttach\cap      V_i|
    = |B\cap             V_i|
    = |\pathOther\cap   V_i| = g/10$ and
    $ |\pathRoot\cap     V_i|
    = |\pathLast\cap  V_i| = g+2\beta m$.
  \item\label{i:large-degree-reserved}
    For all $x\in V_i$ and all $Y\in \{\pathRoot,\pathLast,\pathOther\}$
    we have $\deg^-(x,Y\cap V_\imm),\, \deg^+(x,Y\cap V_\ipp) \ge |Y|d/2k$.
  \item\label{i:degree-from-V0-in-B}
    For all $x\in V_0$, all $\bullet\in\{-,+\}$ and all~$j\in N_\rstar^\bullet(x)$, we have
    $\deg^\bullet_G(x,B\cap V_j)\ge |B|/2k$.
  \item\label{i:large-degree-reserved-for-attachments}
    For each attachment~$a$ of $T$, and $\bullet\in\{-,+\}$ if~$u\in N_Q^\bullet(a)$ is the (unique) neighbour of~$a$ outside of~$T$ and $\circ$ is the opposite sign to $\bullet$,
    then
    \[
      \deg^\circ(\varrho(u),\treeAttach\cap V_{\varphi(a)}) \ge |\treeAttach|\eta/2k.
    \]
  \end{enumerate}
\end{claim}

\begin{claimproof}
 For each \ink~let $V_i'=V_i\setminus \varrho(\Vground)$, and choose a collection of pairwise disjoint subsets $\treeAttach\cap      V'_i$, $B\cap   V'_i$, $\pathOther\cap   V'_i$, $\pathRoot\cap  V'_i$, $\pathLast\cap  V'_i$ of~$V'_i$, with sizes dictated by~\ref{i:reserved-set-size}, uniformly at random among all such collections and independently of the choices for each $i' \neq i$. In particular this means that each set is chosen uniformly at random among all subsets of $V'_i$ of the specified size. Define each of $\treeAttach,B,\pathRoot,\pathLast$ and~$\pathOther$ to be the union over all $i \in [k]$ of the corresponding intersections. So~\ref{i:reserved-set-size} holds by our choice of the sets; we now use concentration inequalities to show that with high
  probability we also have the properties~\ref{i:large-degree-reserved}--\ref{i:large-degree-reserved-for-attachments}.
Let~$\cala=\{\treeAttach,B,\pathRoot,\pathLast,\pathOther\}$.
  Fix \ink, and recall that $G[V_\imm\farc V_i]$ and~$G[V_i\farc V_\ipp]$ are both
  $(d,\eps)$-superregular. Hence, for each~$x\in V_\imm$, each
  $y\in V_\ipp$ and all $X\in\cala$,
  we have that $\deg^+(x,X\cap V_i')$ and $\deg^-(y,X\cap V_i')$ are
  random variables with hypergeometric distribution and with
  expectation at least $|X\cap V_i'|(d-\eps)-|\Vground|\geq 2|X\cap V_i|d/3$.
  So, by Theorem~\ref{t:exp}, the probability that any one of these random
  variables has value
  strictly less than~$|X\cap V_i|d/2$
  decreases
  exponentially with~$n$. By taking a union bound over all $i \in [k]$, all $x \in V_i$
  and each $X \in \{\pathRoot,\pathLast,\pathOther\}$, it
  follows that with high probability all these random
  variables have value at least $|X\cap V_i|d/2 = |X|d/2k$
  (for each choice of~$X$), which implies that
  \ref{i:large-degree-reserved}
  holds with high probability.
  Similar union bound arguments complete the proof of the claim, using
  Lemma~\ref{l:reduced-graph}\,\ref{rg-rstar-V0-edge}
  for~\ref{i:degree-from-V0-in-B},
  and using both~\eqref{e:allocation-r^T} and~\eqref{e:allocation-s^T}
  for~\ref{i:large-degree-reserved-for-attachments}.
\end{claimproof}

In the case where the paths in $\calp$ are bare paths of order~$7$, fix sets $\treeAttach,\, B,\, \pathRoot,\,\pathLast$ and~$\pathOther$ as in Claim~\ref{cl:embed/reserve}. We will embed attachments
of trees in~$\calt$ into~\treeAttach,
neighbours of distinguished vertices into~$B$, vertices $v_1^P$ of paths $P \in \calp^H$ into~\pathRoot\ and vertices $v_7^P$ of paths $P \in \calp^H$ into~\pathLast,
while \pathOther\ will be used
to ensure the superregularity properties of~\ref{cl:embed}\,\ref{i:sreg}.
In the other case, where the paths in $\calp$ have order~$2$, we instead fix sets $\treeAttach,\, B,\, \pathRoot,$ and~$\pathOther$ with the properties in Claim~\ref{cl:embed/reserve} but take $\pathLast$ to be empty.
We embed to these sets as in the previous case, except now there are no vertices $v_7^P$ for $P \in \calp^H$ to be embedded to $\pathLast$.
In both cases all other vertices of $Q - M$ will be embedded outside of these sets (recall that distinguished vertices have already been embedded to~$V_0$).
We also insist that each attachment of each tree in $\calt$ is embedded in the appropriate neighbourhood of the image of its neighbour in $\Vground$, and each neighbour of a distinguished vertex is embedded within the appropriate neighbourhood of the image of the distinguished vertex.
For an easy description of these rules for embedding,
for each $v\in V(F)$ which is not a distinguished vertex we set
  \[
    Z_v=\begin{cases}
      V_{\varphi(v)} \cap A\cap N_G^\bullet\bigl(\varrho(\hat r^T)\bigr)
      & \text{if $v=r^T$ for some $T \in \calt$ and $r^T \in N_Q^\bullet(\hat r^T)$,}\\
      V_{\varphi(v)} \cap A\cap N_G^\bullet\bigl(\varrho(\hat s^T)\bigr)
      & \text{if $v=s^T$ for some $T \in \calt$ and $s^T \in N_Q^\bullet(\hat r^T)$,}\\
      V_{\varphi(v)} \cap B\cap N_G^\bullet\bigl(\varrho(z)\bigr)
      & \text{if $v$ is a neighbour of a distinguished vertex $z$ and $v \in N_Q^\bullet(z)$,}\\
      V_{\varphi(v)} \cap \Pi_1
      & \text{if $v=v_1^P$ for some $P\in\calp^H$,}\\
      V_{\varphi(v)} \cap \Pi_7
      & \text{if $v=v_7^P$ for some $P\in\calp^H$, and}\\
      V_{\varphi(v)} \setminus (\treeAttach\cup B\cup \pathRoot \cup \pathLast \cup \pathOther)
      & \text{if $v$ is not special,}
    \end{cases}
  \]
and we will embed $v$ within $Z_v$ unless $v \in M$ (in which case we do not embed $v$ at all). We do this by the following greedy algorithm.

\medskip

\phase{Setup.}
Let $\verticesUsedAtStart$ be the initial image of~$\varrho$
(that is, $\verticesUsedAtStart$ contains the images of vertices in $\Vground$
and of distinguished vertices). Let $t_1 \prec t_2 \prec \dots \prec t_{v(F)}$ be the vertices of $F$\!, ordered as in $\prec$.
Beginning at time $\tau = 1$, we take the following steps.

\phase{1) Update available vertices.}
Say that a vertex $u \in V(F)$ is \defi{open} if $u \prec \ttau$ but $u$ has a child $v$ with $\ttau \prec v$.
Let
\[
  \open(\tau) := \{u \in V(F) : u \mbox { is open}\},
\]
and
\[
  \used(\tau)\deq
  \bigl(\,\verticesUsedAtStart\cup
 \{\varrho(t_\sigma): \sigma < \tau, t_\sigma \notin M\}\,\bigr)
\]
so $\used(\tau)$ is the set of vertices of~$G$ which have already been occupied by the
image of a vertex of $Q$.
Let
\[
  \reserved(\tau)= \bigcup_{u \in \open(\tau)} \,\,\bigcup_{v\in C(u)} R(v)
\]
be the set of vertices which are currently reserved for children of open vertices, and for each $w \in V(F)$ with $\ttau \prec w$ let
\begin{align*}
  Z_{w,\tau}
  &   =  Z_w \setminus \bigl( \used(\tau)\cup\reserved(\tau)\bigr).
  \end{align*}
For vertices $w \in V(F)$ for which a reserved set $R(w)$ has not yet been selected, $Z_{w, \tau}$ is the set which $w$ is presently permitted to reserve vertices from (we make no use of the sets $Z_{w, \tau}$ for those vertices $w$ for which~$R(w)$ has already been selected). Finally, for each $w \in \{\ttau\} \cup C(\ttau)$ set
\begin{align*}
  \calu^+(w) =  \bigl\{ Z_{x,\tau}: x \in C^+(w)\bigr\} \mbox{ and } \calu^-(w) =  \bigl\{ Z_{x,\tau}: x \in C^-(w)\bigr\}
\end{align*}
and let $S_w \subseteq F$ be the oriented star with centre~$w$ whose leaves are the children of~$w$.
Say that a vertex $\vtau \in V_{\varphi(\ttau)}$ is \defi{$\ttau$-good} if for each $\circ\in\{-,+\}$ and each $w\in C^\circ(\ttau)$ we have
    $\deg^\circ(\vtau, Z_{w,\tau})\ge \gamma m$.

\phase{2) Embed $\vtau$.} If $\ttau$ is a distinguished vertex or a vertex of $M$, then do nothing at this step. Otherwise, embed $\ttau$ as follows.
\begin{enumerate}[leftmargin=3em,label={\hfill(2.\arabic*)} ,ref={(2.\arabic*)}]
\item\label{i:step3a} If $\ttau = r^T$ for some $T \in \calt$, then choose a $\ttau$-good vertex $\vtau \in Z_{\ttau, w}$ and set $\varrho(\ttau) = \vtau$.
\item\label{i:step3b} Otherwise, choose a $\ttau$-good vertex $\vtau \in R(\vtau)$ and set $\varrho(\ttau) = \vtau$.
\end{enumerate}

\phase{3) Declare reserved sets for children of $\ttau$.}
For each $\bullet\in\{-,+\}$ and each child $w \in C^\bullet(\ttau)$ which is not a distinguished vertex,
choose a set $R(w) \subseteq Z_{w,\tau}$ with $|R(w)| = \sqrt{m}$ which is
$\bigl(\calu^+(w),\calu^-(w),\beta,\gamma,\varphi,m\bigr)$-good for $S_{w}$, and make these choices so that the sets $R(w)$ for $w \in C(\ttau)$ are pairwise disjoint. Moreover, if $\ttau$ is neither a distinguished vertex nor an element of $M$ then we insist additionally that $R(w) \subseteq N_G^\bullet(\vtau)$.

\phase{4) Loop.} If $\tau = v(F)$ then terminate; otherwise increment $\tau$ and return to Step 1.

\bigskip

\medskip\noindent\textbf{Proof that the algorithm runs successfully.} Observe that the Setup and Steps 1 and 4 consist solely of definitions. We shall prove that (I) it is possible to make the choices required by the embedding
algorithm at Steps 2 and 3, and (II) that if the choices of the algorithm can be made, then it produces an embedding~$\varrho$ of $Q - M$ in $G$ with the properties stated in the claim.

We first prove~(II). At the start of the algorithm,
the images of $\Vground$ and of the distinguished vertices are already fixed under~$\varrho$, which is an embedding of the subgraph of $Q$ induced by these vertices into $G$ (in~particular, $\varrho$ bijectively maps the distinguished vertices to~$V_0$). So to show that the outcome of the algorithm is a embedding~$\varrho$ from $Q - M$ to $G$, it suffices to show that $\varrho$ is injective and also that for each vertex $\ttau \in F$ we have
\begin{enumerate}[label=(\alph*)]
    \item $\varrho(\ttau) \in N^+_G(\varrho(u))$ for each $u \in N^-_Q(\ttau)$ which was already embedded at time $\tau$, and
    \item $\varrho(\ttau) \in N^-_G(\varrho(u))$ for each $u \in N^+_Q(\ttau)$ which was already embedded at time $\tau$.
\end{enumerate}
To do this, observe first that since the algorithm considered the vertices of $Q$ in an ancestral order, for each $\ttau \in V(F)$ with a parent $t_\sigma$ in $F$ the parent $t_\sigma$ had already been embedded at time $\tau$. Furthermore, the image $\varrho(\ttau)$ of $\ttau$ was chosen within the set $R(\ttau)$ of vertices reserved for the embedding of $\ttau$, which in turn was previously chosen to be a subset of the appropriate neighbourhood of $t_\sigma$ at time $\sigma$. Observe also that the set $R(\ttau)$ was chosen to avoid $\used(\sigma)$, and no vertices we embedded into $R(\ttau)$ between times $\sigma$ and $\rho$, so no vertex was embedded to $\varrho(\ttau)$ before time $\tau$. Moreover, for the $\bullet \in \{-,+\}$ with $\ttau \in N_Q^\bullet(t_\sigma)$ we have $\varrho(\ttau) \in N_G^\bullet(\varrho(t_\sigma))$, as desired. It remains to demonstrate the desired property in the case where a neighbour of $\ttau$ in $Q$ other than a parent of $\ttau$ was already embedded at time $\tau$. This can occur in the following three ways.
\begin{enumerate}[label=(\arabic*)]
\item $\ttau = r^T$ for some $T \in \calt$. In this case $\hat r^T$ is the only neighbour of $\ttau$ which has previously been embedded, and the definition of $Z_\ttau$ ensures that $\ttau$ is embedded in the appropriate neighbourhood of $\varrho(\hat r^T)$. Moreover, the embedding of $\ttau$ was chosen to avoid $\used(\tau)$, so no vertex was embedded to $\varrho(\ttau)$ before time $\tau$. Together with the previous observation that this holds for vertices $\ttau$ with a parent in $F$, this shows that $\varrho$ is injective.
\item $\ttau = s^T$ for some $T \in \calt\!$. In this case $\hat s^T$ and the parent~$p^T$ of $\ttau$ in $T$ are the only neighbours of~$\ttau$ which have previously been embedded. The definition of $Z_\ttau$ ensures that $\ttau$ is embedded in the appropriate neighbourhood of $\hat \varrho(s^T)$.
\item $\ttau$ is the parent of a distinguished vertex~$v$. In this case the parent of~$\vtau$ and $v$ are the only neighbours of~$\vtau$ which have previously been embedded. Again the definition of $Z_\ttau$ ensures that $\ttau$ is embedded in the appropriate neighbourhood of $\varrho(v)$.
\end{enumerate}
So the output of the algorithm is indeed an embedding $\varrho$ of $Q - M$ in $G$. Also our choice of the sets $Z_v$ ensures that we have~\ref{i:respects-alloc}. Now suppose that the paths in $\calp$ have order two. In this case Claim~\ref{cl:allocate}~\ref{i:Fi-i},~\ref{i:Fi-ii} and~\ref{i:Fi-iii} together imply that for each $\ink$ precisely $m$ vertices of $Q$, precisely $g$ vertices of $M$, and precisely~$g$ vertices~$v_1^P$ of paths $P \in \calp^H$ are allocated to $V_i$, and then~\ref{cl:embed_2path} follows from~\ref{i:respects-alloc}. Similarly, if the paths in $\calp$ have order seven then Claim~\ref{cl:allocate}~\ref{i:Fi-i},~\ref{i:Fi-ii} and~\ref{i:Fi-iii} together imply that for each $\ink$ precisely $m$ vertices of~$Q$, precisely~$5g$ vertices of $M$, precisely $g$ vertices $v_1^P$ of paths $P \in \calp^H$ and precisely $g$ vertices $v_7^P$ of paths $P \in \calp^H$ are allocated to $V_i$, so we have~\ref{cl:embed_7path}.

For~\ref{i:sreg}, suppose first that the paths in~$\calp^H$ have order 2.
This means that $\pathLast,W_1,\dots,W_k$ are all empty.
Since the embedding respects the allocation, for each~\ink\ precisely $g$ vertices are embedded to $\pathRoot\cap V_i$, so $|U_i|=g$. Moreover, among the $m$ vertices $u \in V(Q)$ with $\varphi(u) = i$, the vertices which have not been embedded in~$V_i$ are precisely the $g$ vertices $v_2^P$ for $P \in \calp^H$ with $\varphi(v) = i$, so $|V_i^\star|=g$ as well. Since $g=\bigl\lceil\frac{\lambda m}{2^{10}}\bigr\rceil$, it follows that $G[V_{\imm}^\star\rarr U_i]$,
$G[U_i\rarr V_{\ipp}^\star]$ and $G[V_{\imm}^\star\rarr V_i^\star]$ are each $(d,\eps')$-regular by Lemma~\ref{l:slice-pair}.
To check the degree condition required for superregularity in~\ref{i:sreg}, note that
by Claim~\ref{cl:embed/reserve}\,\ref{i:large-degree-reserved} and the fact that precisely~$g$ vertices $v_1^P$ of paths $P \in \calp^H$ are embedded to $\Pi_1 \cap V_i$, for each~\ink\ and each~$x\in V_{i+1}$ we have
\begin{align*}
  \deg^-(x, U_i)
  & \ge \frac{|\pathRoot|d}{2k}-\bigl(|\pathRoot \cap V_i|-g\bigr) \ge \frac{kgd}{2k}-2\beta m \ge \beta g = \beta|U_i|,
\end{align*}
and similarly for each~\ink\ and each~$x\in V_{i-1}$ we have
$\deg^+(x, U_i) \geq \beta |U_i|$. In the case where the paths in~$\calp^H$ have order~$7$, we obtain identical bounds with $W_i$ and $\Pi^7$ in place of $U_i$ and $\Pi^1$ respectively, as in this case for each $\ink$ precisely~$g$ vertices $v_7^P$ of paths $P \in \calp^H$ are embedded to $\Pi_1 \cap V_i$. Moreover, since no vertex is ever embedded to $\pathOther$, we have $\pathOther \cap V_i\subseteq V_i^\star$ for each~\ink. This means that by Claim~\ref{cl:embed/reserve}\,\ref{i:large-degree-reserved} for each~\ink\ and each~$x\in V_i$ we have
\begin{align*}
  \deg^-(x,V_\imm^\star)
 \ge \deg^-(x,\pathOther \cap V_{i-1}) \ge \frac{kg}{10}\cdot \frac{d}{2k} \ge \beta g =\beta|V_\imm^\star|,
\end{align*}
and similarly $\deg^+(x,V_{i+1}^\star)
\ge \beta|V_{i+1}^\star|$.
It follows that each of the graphs in~\ref{i:sreg} is $(\beta, \eps')$-superregular, concluding the proof of (II).

\bigskip\noindent
We now turn to proving~(I).
The following invariant plays a crucial role in our analysis:
for each $w\in V(F)$ which is not distinguished and each time~$\tau$ in the execution of the algorithm we have
\begin{equation}\label{e:Zw}
|Z_{w,\tau}|=\bigl|\,Z_w\setminus\bigl(\used(\tau)\cup\reserved(\tau)\bigr)\,\bigr|\ge \beta m.
\end{equation}
To verify~\eqref{e:Zw}, note that since $\prec$ is a tidy ancestral order of $F$, for each $\tau$ we have $|\open(\tau)| < \log_2 n$ and so $|\reserved(\tau)| \le (\log_2 n) \Delta(Q) \sqrt{m} \le n^{2/3}$. Moreover,
Claim~\ref{cl:embed/reserve} and the properties of~$\varphi$ guaranteed by Claim~\ref{cl:allocate} imply that
for each~\ink\ we have
\[|\pathRoot \cap \varphi(V(Q) \setminus M) \cap V_i| = g = |\pathLast \cap \varphi(V(Q) \setminus M) \cap V_i| = g,
\]
except in the case where paths in $\calp$ have order two, in which case we do not have the latter equality because~$\Pi_7$ is empty.
For each $\ink$ we also have
\[
  \bigl| A   \cap \varphi\bigl(V(Q)\setminus M\bigr) \cap V_i\bigr|          \stackrel{\phantom{\ref{l:tl/a-V0-neigh}}}{\le} 2|\calt|\stackrel{{\eqref{e:|calt|}}}{\le} 2n^{0.999} \leq \beta m,
\]
and, by Claim~\ref{cl:allocate}\ref{l:tl/a-V0-neigh}.
\[
  \bigl| B \cap \varphi\bigl(V(Q)\setminus M\bigr) \cap V_i\bigr|   \leq \frac{6\eps m}{\alpha} \leq \beta m,
\]
Finally,
\[
  \bigl| \bigl(V_i\setminus (\treeAttach\cup B\cup \pathRoot \cup \pathLast \cup \pathOther)\bigr) \cap \varphi\bigl(V(Q)\setminus M\bigr)\bigr|
  \le \Bigl|\, V_i \cap \varphi\bigl(\,V(Q)\setminus\!\!\! \bigcup_{P \in \calp^H}\!\!\! P\,\bigr)\,\Bigr|,
\]
and we can bound this quantity by $|V_i| - 2g$ if the paths in $\calp$ each have order two and by $|V_i|-7g$ if the paths in $\calp$ each have order seven. Together with the sizes of the sets $A,\, B,\, \Pi_1,\, \Pi_7$ and $\Pi_\mathrm{other}$ and bounds on degrees given in Claim~\ref{cl:embed/reserve}\ref{i:degree-from-V0-in-B} and~\ref{i:large-degree-reserved-for-attachments} (recalling also that $\Pi_7$ is empty if the paths in $\calp$ have order two), these bounds imply~\eqref{e:Zw}. In particular, it follows that when we define the families $\calu^+(w)$ and $\calu^-(w)$, each set within these families contains at least $\beta m$ vertices.

In Step~2 we wish to choose a $\ttau$-good vertex~$\vtau$. If $\ttau = r^T$ for some $T \in \calt$, then we need to do this with~$\vtau \in Z_{\ttau, w}$. Since $|Z_{\ttau, \tau}| \ge \beta m$ and $|Z_{w, \tau}| \ge \beta m$ for each $w \in C(\ttau)$ by~\eqref{e:Zw}, and $\Delta(\varphi^T)\le 5$ by~Claim~\ref{cl:allocate}, we may apply Lemma~\ref{l:goodness} to obtain a $\bigl(\calu^+(\vtau),\calu^-(\vtau),\beta,\gamma,\varphi,m\bigr)$-good set for~$S_{\ttau}$ in~$Z_{\ttau}$. By definition each vertex in this set is $\ttau$-good, so we can choose $\vtau$ as required. On the other hand, if $\ttau$ is not the root of a tree in $\calt$, then $\ttau$ has a parent $t_\sigma$ in $F$, and we wish to choose $\vtau$ in $R(\ttau)$, the set previously reserved at time $\sigma$ for the future embedding of $\ttau$. The set $R(\ttau)$ was chosen at time $\sigma$ to be a subset of $\subseteq Z_{\ttau,\sigma}$ which was $\bigl(\calu^+(\vtau),\calu^-(\vtau),\beta,\gamma,\varphi,m\bigr)$-good for~$S_{\ttau}$. Since no vertex was embedded in or reserved from $R(\ttau)$ between times $\sigma$ and $\tau$ we have $R(\ttau) \subseteq Z_{\ttau,\sigma}$, so by definition of $\bigl(\calu^+(\vtau),\calu^-(\vtau),\beta,\gamma,\varphi,m\bigr)$-good and~\eqref{e:Zw} we can choose $\vtau \in R(\ttau)$ as required.

Finally, in Step~3 we wish to reserve a set $R(w)$ for each child~$w\in C(\ttau)$. If $\ttau$ is neither a distinguished vertex nor a vertex of $M$, then we have just selected a $\ttau$-good vertex $\vtau$ for the image of $\ttau$, and we require that $R(w) \subseteq N_G^\bullet(v_\tau) \cap Z_{w, \tau}$, where $\bullet \in \{+, -\}$ is such that $w \in C^\bullet(\ttau)$. The $\ttau$-goodness of $\vtau$ ensures that $|N_G^\bullet(v_\tau) \cap Z_{w, \tau}| \geq \gamma m$ for each $\bullet \in \{+, -\}$ and $w \in C^\bullet(\ttau)$. Along with~\eqref{e:Zw} and the fact that $\Delta(\varphi^T)\le 5$ this enables us to apply Lemma~\ref{l:goodness} to obtain the desired $R(w)$ for each $w \in C(\ttau)$ (since we only need $|N_G^\bullet(v_\tau) \cap Z_{w, \tau}| \geq \gamma m/2$ to apply Lemma~\ref{l:goodness}, we may do this so that the sets $R(w)$ are also pairwise disjoint, as required). If instead $\ttau$ is a distinguished vertex or a vertex of $M$, then we instead require just that $R(w) \subseteq Z_{w, \tau}$. By~\eqref{e:Zw} we have $|Z_{w, \tau}| \geq \beta m$ for each $w \in C(\ttau)$, and so again we may apply Lemma~\ref{l:goodness} to obtain the desired $R(w)$ for each $w \in C(\ttau)$.
This concludes the proof of Claim~\ref{cl:embed}.

\smallskip

\begin{remark}
  The embedding algorithm can be significantly simplified
  if the goal is to embed an almost spanning structure
  (say, with at most $(1-\alpha)n$ vertices).
  In particular the sets $B,\pathRoot,\pathLast$ and~$\pathOther$
  would no longer be needed, and we can completely avoid embedding vertices to~$V_0$.
\end{remark}

\section{Acknowledgements}

We thank the anonymous reviewers for their thoughtful comments,
and for suggesting a shift of the focus of the paper toward the more general
result. We believe that our consequent restructuring of the content of this manuscript has improved the quality of the paper, with clearer and more concise arguments.

\bibliographystyle{plain}
\bibliography{semidegree.bib}
\end{document}